\documentclass[reqno, final,a4paper]{amsart}

\usepackage{color}
\usepackage[colorlinks=true,allcolors=blue
]{hyperref}
\usepackage{amsmath, amssymb, amsthm,bbm}
\usepackage{mathrsfs}
\usepackage{mathtools}
\usepackage{cases}
\usepackage[noabbrev,capitalize,nameinlink]{cleveref}
\crefname{equation}{}{}
\usepackage{textalpha}
\usepackage{fullpage}
\usepackage[noadjust]{cite}
\usepackage[table]{xcolor}
\usepackage{graphics}
\usepackage{pifont}
\usepackage{tikz}
\usepackage[T1]{fontenc}
\usetikzlibrary{arrows.meta}

\usepackage{environ}
\usepackage{url}
\usepackage[linesnumbered,ruled,vlined]{algorithm2e}
\usepackage[noend]{algpseudocode}
\usepackage[labelfont=bf]{caption}
\usepackage{framed}
\usepackage[framemethod=tikz]{mdframed}
\usepackage{pgfplots}
\usepgfplotslibrary{groupplots}
\usepackage{appendix}
\usepackage{graphicx}
\usepackage[textsize=tiny]{todonotes}

\usepackage{cancel}
\crefformat{enumi}{#2#1#3}
\crefrangeformat{enumi}{#3#1#4 to~#5#2#6}
\crefmultiformat{enumi}{#2#1#3}
{ and~#2#1#3}{, #2#1#3}{ and~#2#1#3}
\allowdisplaybreaks
\usepackage{paralist}
\usepackage[shortcuts]{extdash}

\let\originalleft\left
\let\originalright\right
\renewcommand{\left}{\mathopen{}\mathclose\bgroup\originalleft}
\renewcommand{\right}{\aftergroup\egroup\originalright}

\let\originalbig\big
\renewcommand{\big}{\mathopen{}\mathclose\originalbig}

\newcommand*{\claimproofname}{Proof of claim}

\crefname{equation}{}{} %remove ``Equation''
\AtBeginEnvironment{appendices}{\crefalias{section}{appendix}} %appendices

\colorlet{refkey}{orange!20}
\colorlet{labelkey}{blue!30}

\crefname{algocf}{Algorithm}{Algorithms}

\numberwithin{equation}{section}
\newtheorem{theorem}{Theorem}[section]
\newtheorem{proposition}[theorem]{Proposition}
\newtheorem{lemma}[theorem]{Lemma}
\newtheorem{claim}[theorem]{Claim}
\crefname{claim}{Claim}{Claims}
\newtheorem{corollary}[theorem]{Corollary}

\newtheorem*{question*}{Question}
\newtheorem{fact}[theorem]{Fact}

\theoremstyle{definition}
\newtheorem{definition}[theorem]{Definition}

\newtheorem*{definition*}{Definition}
\newtheorem*{theorem*}{Theorem}

\theoremstyle{remark}
\newtheorem{remark}[theorem]{Remark}
\newtheorem*{remark*}{Remark}

\newcommand{\one}{\mathbbm{1}}

\newcommand{\mb}{\mathbb}
\newcommand{\mbf}{\mathbf}

\newcommand{\mc}{\mathcal}
\newcommand{\mf}{\mathfrak}
\newcommand{\mr}{\mathrm}

\newcommand{\on}{\operatorname}

\newcommand{\eps}{\varepsilon}
\renewcommand{\Pr}{\mb P}

\newcommand{\hide}[1]{}

\allowdisplaybreaks

\title{Nonlinear lower-tail large deviations at criticality}

\author[Kwan]{Matthew Kwan}

\address{Institute of Science and Technology Austria (ISTA).}
\email{matthew.kwan@ist.ac.at}

\author[Pham]{Huy Tuan Pham}

\address{Department of Mathematics, University of Chicago.}
\email{htpham@uchicago.edu}

\thanks{
Matthew Kwan was supported by ERC Starting Grant ``RANDSTRUCT'' No.\ 101076777. Huy Tuan Pham was supported by a Clay Research Fellowship, a Sloan Research Fellowship and NSF Career grant DMS-2543870. Part of this work was done when Pham was visiting the Institute for Advanced Study and supported by the Bell System Fellowship. 
}

\begin{document}

\maketitle
\begin{abstract}
We introduce general methods to derive lower tail large deviation principles, for a variety of nonlinear problems of combinatorial interest. These methods are especially effective in ``critical'' regimes, where the nonlinearities are not sparse enough for Poissonian lower tails (e.g., Janson's inequality does not provide a sharp tail bound), but the situation is not so dense that combinatorial effects dominate (e.g., we cannot directly apply the theory of hypergraph containers or the relative entropy framework of Kozma--Samotij). In particular, these are regimes where one expects interesting phase transitions to occur. 

Our results have a number of consequences related to subgraphs in random graphs, answering various questions of Warnke and Jenssen--Perkins--Potukuchi--Simkin.
For example, consider a random graph $G\sim \mb G(n,p)$, and let $L_\triangle(n,p)=\log\Pr[G\text{ is triangle-free}]$. 
The first-order asymptotics of $L_\triangle(n,p)$ have long been known in all parameter ranges except the critical regime where $p$ has order of magnitude $1/\sqrt n$. We are able to fill this gap: for any fixed $c>0$, we show that $n^{-3/2}L_\triangle(n,c/\sqrt n)$ converges to a limit expressed in terms of a two-parameter variational problem, which has a single phase transition at $c\approx 4.341$. In contrast, we show that there is no such phase transition for the probability that $G$ is $C_{2\ell}$-free, for any fixed even cycle $C_{2\ell}$.

We also obtain some applications in combinatorial design theory. Addressing conjectures of Glock--K\"uhn--Lo--Osthus, Kelly, and Kwan--Sah--Sawhney--Simkin, we obtain new estimates on the number of order-$n$ Steiner triple systems with no Pasch configuration  and the number of order-$n$ Latin squares with no  $2\times 2$ Latin subsquare (both tight up to a factor of $\exp(o(n^2))$).
\end{abstract}
\section{Introduction}\label{sec:intro}
What is the probability that an Erd\H os--R\'enyi random graph $G\sim \mb G(n,p)$ is triangle-free? Curiously, the answer to this question changes qualitatively around $p=1/\sqrt n$. Indeed, writing $X_{K_3}(G)$ for the number of triangles in $G$, it is well-known that
\begin{numcases}{\log\Pr[G\text{ is triangle-free}]=}
-(1+o(1))\,\mb EX_{K_3}(G)&if $p\ll 1/\sqrt n$,\label{eq:subcritical-triangles}\\
\phantom-(1+o(1))\,\log\,\Pr[G\text{ is bipartite}]&if $p\gg 1/\sqrt n$,\label{eq:supercritical-triangles}
\end{numcases}
where we write $f\ll g$ to indicate that\footnote{Here we view $p$ as a function of $n$, and asymptotics are as $n\to \infty$.} $f=o(g)$.

That is to say, there are ``subcritical'' and ``supercritical'' regimes of $p$, where the triangle-freeness probability behaves completely differently.
In the subcritical regime, the probability is of ``Poisson type'' (note that a $\on{Poisson}(\mu)$ random variable is zero with probability $\exp(-\mu)$), and the bound in \eqref{eq:subcritical-triangles} can be proved by quite generic large deviation inequalities (see e.g.\ \cite{JLR90}), only using that $X_{K_3}(G)$ is a sum of random variables with ``not too much dependence''. On the other hand, in the supercritical regime, the probability is described by extremal graph theory considerations: recall that Mantel's theorem (see e.g.\ \cite[Theorem~1.1.1]{Zha23}) says that complete bipartite graphs have the largest possible number of edges among all triangle-free graphs on $n$ vertices. Correspondingly, most proofs of \eqref{eq:supercritical-triangles} use tools from extremal graph theory, such as Szemer\'edi's regularity lemma and the hypergraph container method (see \cite{ST15,BJM15,Luc00}).

In the ``critical regime'' where $p$ has order of magnitude $1/\sqrt n$, it is not hard to see that the expressions in \eqref{eq:subcritical-triangles} and \eqref{eq:supercritical-triangles} both have order of magnitude $-n^{3/2}$, and indeed it is well-known that $\log\Pr[G\text{ is triangle-free}]$ does have order of magnitude $-n^{3/2}$. However, it seems to be much more difficult to characterise the precise asymptotic value of $\log\Pr[G\text{ is triangle-free}]$ in this regime: we must see some kind of transition between the ``Poisson-type'' behaviour in \eqref{eq:subcritical-triangles} and the ``extremal graph theory'' behaviour in \eqref{eq:supercritical-triangles}. Motivated by statistical physics considerations, recently Jenssen, Perkins, Potukuchi and Simkin~\cite{JPPSold,JPPSnew} made some interesting conjectures about the triangle-freeness probability in the critical regime. They managed to make some progress towards these conjectures, but the picture was still very incomplete.

In this paper we introduce general methods to study ``critical'' lower tail large deviation problems of the above type. For such problems, our methods allow one to describe the rate function in terms of the solution to an analytic variational problem. For example (very roughly speaking, before we've had a chance to properly introduce the relevant concepts), in  the critical regime for triangle-freeness described above, the asymptotic value of $\log \Pr[G\text{ is triangle-free}]$ is determined by the optimal tradeoff between two ``entropic costs'' that we can ``pay'' to make our graph triangle-free: on one hand, we can pay some cost to change the edge-profile of our random graph in an inhomogeneous way (that may indirectly result in fewer triangles), and given the resulting inhomogeneous random graph we can pay some Poissonian cost to suppress the number of triangles down to zero.

The variational problems produced by our method can have a diverse range of different behaviours, and it seems that their solutions can have quite intricate combinatorial structure. As such, we do not have a completely general way to solve these variational problems explicitly. However, we present a few different techniques which are effective in certain special cases. As an illustration, we put some serious effort towards the relevant optimisation problem for the triangle-freeness probability, and we managed to obtain a complete solution for this particular problem. This involved first reducing the problem to the existence of a suitable finite-dimensional ``dual certificate function'', and then using AI tools and heavy computational interval arithmetic to actually find such a certificate function and verify its properties.
As a result, we are able to confirm the existence of a single phase transition at a critical point $c^*\approx 4.341$ (as conjectured by Jenssen--Perkins--Potukuchi--Simkin~\cite{JPPSold}; see \cref{fig:triangle-kink}). 

In the rest of this introduction, we present various different applications of our methods. In \cref{subsec:subgraphs} we collect several different results concerning lower tail large deviations for subgraph counts in random graphs. All of these can be interpreted as special cases of a general framework we introduce in \cref{subsec:general}, concerning lower tail large deviations for the number of hyperedges in a random vertex-subset of a hypergraph. Then, in \cref{subsec:design-theory} we present some results in combinatorial design theory.

\begin{figure}
\tikzstyle{vertex}=[circle,draw=black,fill=black,inner sep=0,minimum size=0.2cm,text=white,font=\footnotesize]
\begin{centering}
\begin{tikzpicture}[
  declare function={
    tcZ(\h)=exp(-\h);
    tcT(\h)=exp(-2*\h*tcZ(\h)/(1-tcZ(\h))^2);
    tcAlpha(\h)=2*\h/(3*tcT(\h)^2*(1-tcZ(\h))^2);
    tcValue(\h)=1-tcT(\h)*(1+tcZ(\h))/2
                 -\h*tcT(\h)*(tcZ(\h)^3+3*tcZ(\h))
                   /(3*(1-tcZ(\h))^2);
  }
]
\def\tcCStar{4.340842036084886}
\def\tcValueStar{0.4999589178983022}
\begin{axis}[
  name=psiMain,
  width=11cm, height=5cm, scale only axis,
  xmin=0, xmax=6, ymin=0, ymax=0.58,
  axis x line*=bottom, axis y line*=left,
  xlabel={$c$}, ylabel={$\Psi_{\on{Bin}}(K_3,c,0)$},
  xtick={0,1,2,3,4,4.340842036084886,5,6},
  xticklabels={$0$,$1$,$2$,$3$,$4$,$c^*$,$5$,$6$},
  ytick={0,0.1,0.2,0.3,0.4,0.5},
  scaled ticks=false,
  tick align=outside,
  tick label style={font=\small},
  label style={font=\normalsize},
  clip=false,
]
  % The dashed line marks a change of minimiser, not a discontinuity in value.
  \draw[densely dashed, thin]
    (axis cs:\tcCStar,0) -- (axis cs:\tcCStar,\tcValueStar);

  % Constant minimiser: 0 <= b <= b_*.
  \addplot[black, line width=1.1pt, no marks,
    variable=\b, domain=0:0.4422384218823732, samples=401]
    ({sqrt(3*\b)*exp(3*\b)}, {1-(1+2*\b)*exp(-3*\b)});

  % Nonconstant minimiser: h_* <= h <= h_end, where c(h_end)=6.
  \addplot[black, line width=1.1pt, no marks,
    variable=\h, domain=9.390369335634312:17.999979713288505, samples=201]
    ({sqrt(3*tcAlpha(\h))}, {tcValue(\h)});

  \addplot[black, only marks, mark=*, mark size=1.8pt]
    coordinates {(\tcCStar,\tcValueStar)};
  \node[anchor=south, align=center, font=\small, yshift=5pt]
    at (axis cs:\tcCStar,\tcValueStar)
    {%$c_*\approx4.341$
    };

  \node[font=\small, align=center] at (axis cs:2.17,0.075)
    {Constant minimiser};
  \node[font=\small, align=center] at (axis cs:5.17,0.075)
    {Balanced two-block\\minimiser};
\end{axis}

\end{tikzpicture}
\end{centering}
\caption{\label{fig:triangle-kink}A plot of the asymptotic value of $-(p\binom n2)^{-1}\log \Pr[G\text{ is triangle-free}]$, when $G\sim \mb G(n,c/\sqrt n)$. This is called $\Psi_{\on{Bin}}(K_3,c,0)$ in the notation of \cref{cor:main-graphs}. The non-analytic point at $c^*\approx 4.341$ corresponds to a phase transition in the structure of the optimiser of a certain variational problem. We remark that the curve is \emph{not} flat to the right of the critical point, but this is difficult to see in the plot.}
\end{figure}

\subsection{Lower tails for subgraph counts in random graphs}\label{subsec:subgraphs}
For a general introduction to large deviations (of subgraph counts) in random graphs, see for example the monograph of Chatterjee~\cite{Cha17}. This topic has long been studied in combinatorics, with particular attention on the probability that a random graph is $H$-free. More recently, the topic has also attracted considerable attention in probability theory, as a rich testing ground in which to extend the classical theory of large deviations (for sums of independent random variables) to ``nonlinear'' settings.

\subsubsection{Background on the subcritical and supercritical regimes}
For context, we start by recalling the most general versions of \eqref{eq:subcritical-triangles} and \eqref{eq:supercritical-triangles}, for general subgraphs and general lower tails. To state these results we will need some notation and definitions.
\begin{itemize}
    \item For a graph $H$ (viewed as a set of edges), let $e(H)=|H|$ and $v(H)$ be the number of edges and vertices of $H$, respectively. Let $X_H(G)$ be the number of isomorphic copies of $H$ in a graph $G$.
    \item Let $\mb G(n,m)$ be the Erd\H os--R\'enyi random graph distribution with $n$ vertices and $m$ edges, and (overloading notation as usual) let $\mb G(n,p)$ be the binomial random graph distribution with $n$ vertices and edge probability $p$.
    \item For two discrete probability distributions $\mu,\nu$ on the same sample space $\mc X$, we write $\on{D}_{\mr{KL}}(\mu \,\|\,\nu)$ for the Kullback--Leibler divergence (relative entropy) of $\mu$ from $\nu$:
    \[\on{D}_{\mr{KL}}(\mu\,\|\,\nu)=\sum_{x\in \mc X}\mu(x)\log\frac{\mu(x)}{\nu(x)}.\]
    \item For a weighting $\vec q=(q_e)_{e\in K_n}\in [0,1]^{K_n}$ of the edges of the complete graph $K_n$, let $\mb G(n,\vec q)$ be the random graph distribution on subgraphs of $K_n$ where each potential edge $e\in K_n$ is present with probability $q_e$ independently. Also, let
\[\on{D}_{\mr{KL}}(\vec q\,\|\,p)=\on{D}_{\mr{KL}}(\mb G(n,\vec q)\,\|\,\mb G(n,p))=\sum_{e\in K_n}\left(q_e\log\left(\frac {q_e}p\right)+ (1-q_e)\log\left(\frac {1-q_e}{1-p}\right)\right).\]
    \item For $L,M\ge 0$, let
 \[\Phi_L(M)=\min_{\ell\in [0,L]}\on{D}_{\mr{KL}}(\on{Poisson}(\ell)\,\|\,\on{Poisson}(M))=\phi(L/M)M,\]
where
\begin{equation}
    \phi(z)=\begin{cases}
1-z+z\log z&\text{if }0\le z\le 1,\\0&\text{ if }z>1.\end{cases}\label{eq:phi}
\end{equation}
\end{itemize}
In all these definitions we use the conventions $0\log 0=0$ and $0/0=1$ and $L/0=\infty$ for $L>0$.

The most general theorems in the direction of \eqref{eq:subcritical-triangles} (in the ``subcritical regime'') were proved by Janson and Warnke~\cite{JW16} (see also \cite{JLR90}). For a fixed graph $H$ with at least two edges, and fixed $\delta\in (0,1]$, suppose that $p\in [0,1]$ satisfies 
\begin{equation}
    n^{v(H)-v(F)}p^{e(H)-e(F)}\ll 1\quad \text{for all proper nonempty subgraphs }F\subsetneq H,\label{eq:subcritical-condition}
\end{equation} and let $G\sim \mb G(n,p)$. Janson and Warnke proved that if $L=0$ or if $\mb EX_H(G)\gg 1$ and $L\le (1-\delta)\mb EX_H(G)$, we have
\begin{equation}\log \Pr[X_H(G)\le L]=-(1+o(1))\Phi_L(\mb EX_H(G))\label{eq:subcritical-general}.\end{equation}
It is easy to deduce that if $m\in \mb N$ satisfies \cref{eq:subcritical-condition} for $p=m/\binom n2$, then the same estimate holds for $G\sim \mb G(n,m)$.
One should interpret \cref{eq:subcritical-general} as saying that, from the point of view of lower tail large deviations, $X_H(G)$ behaves like a Poisson distribution: the likelihood of the event $X_H(G)\le L$ is described by the minimum possible KL divergence of a Poisson distribution with mean at most $L$, from a Poisson distribution with mean $\mb EX_H(G)$. In general, one expects Poisson approximation to be effective when dealing with random variables that count rare events which are ``almost independent''; the condition \cref{eq:subcritical-condition} should be interpreted as saying that the copies of $H$ typically interact so sparsely that they are almost independent.

In the direction of \eqref{eq:supercritical-triangles} (the ``supercritical regime''), there is a more significant difference between $\mb G(n,p)$ and $\mb G(n,m)$. In the setting of $\mb G(n,p)$, it was proved by Kozma and Samotij~\cite[Theorem~2]{KS23}\footnote{This is not exactly how \cite[Theorem~2]{KS23} is stated; one needs to combine it with the so-called graph removal lemma (see e.g.\ \cite[Theorem~2.6.5]{Zha23}).} that for a fixed graph $H$ with at least two edges, and fixed $\delta\in (0,1]$, if $p\in [0,1-\delta]$ satisfies 
\begin{equation}
    n^{v(F)-2}p^{e(F)-1}\gg 1\quad\text{for all subgraphs }F\subseteq H\text{ with at least two edges,}\label{eq:supercritical-condition}
\end{equation} then for $G\sim \mb G(n,p)$ and $L\le (1-\delta)\mb EX_H(G)$ we have
\begin{equation}
    \log\Pr[X_H(G)\le L]= -(1+o(1))\min_{\vec q\in \mc Q_{L}}\on{D}_{\mr{KL}}(\vec q\,\|\,p),\label{eq:supercritical-general-p}
\end{equation}
where $\mc Q_L$ is the set of all weightings $\vec q\in [0,p]^{K_n}$ such that for $G_{\vec q}\sim \mb G(n,\vec q)$ we have $\mb E X_H(G_{\vec q})\le L$. Similar results with more restrictive conditions on $p$ were previously obtained by many authors (see e.g.\ \cite{CV11,CD20,CD16,Eld18}).

The setting of $\mb G(n,m)$ is treated in upcoming work of Kozma, Samotij,  Shapira and Wagner\footnote{Personal communication with Wojciech Samotij.}. They prove that, for a fixed graph $H$ with at least two edges,  chromatic number $\chi\ge 3$ and fixed $\delta\in (0,1]$, if $m\le (1-(\chi-1)^{-1}-\delta)\binom n2$ satisfies \cref{eq:supercritical-condition} for $p=m/\binom n2$ then for $G\sim \mb G(n,m)$ and $L\le (1-\delta)\mb EX_H(G)$ we have
\begin{equation}
    \log\Pr[X_H(G)\le L]= -(1+o(1))\min_{\vec q\in \mc Q_{L,m}}\on{D}_{\mr{KL}}(\vec q\,\|\,p),\label{eq:supercritical-general-m}
\end{equation}
where $\mc Q_{L,m}$ is the set of all weightings $\vec q\in [0,1]^{K_n}$ such that the sum of entries of $\vec q$ is exactly $m$, and such that for $G_{\vec q}\sim \mb G(n,\vec q)$ we have $\mb E X_H(G_{\vec q})\le L$.

One should interpret \cref{eq:supercritical-general-m,eq:supercritical-general-p} as saying that the probability of the event $X_H(G)\le L$ is described by the ``cheapest way'' for $G$ to resemble a product distribution $\mb G(n,\vec q)$ which typically has $X_H(G)\le L$. The variational problems in \cref{eq:supercritical-general-p,eq:supercritical-general-m} (which correspond to the so-called \emph{mean-field approximation}) are difficult to solve in general, but in the special case $L=0$ (when we are interested in the $H$-free probability) it is not hard to see that there is an asymptotically optimal weighting $\vec q$ which is a rescaling of the characteristic vector of a complete $(\chi-1)$-partite subgraph of $K_n$. This implies that $\log \Pr[G\text{ is }H\text{-free}]=(1+o(1))\log\Pr[G\text{ is }(\chi-1)\text{-partite}]$, generalising \eqref{eq:supercritical-triangles}. Actually, this special case of \cref{eq:supercritical-general-m,eq:supercritical-general-p} was established much earlier by \L uczak~\cite{Luc00} (conditional on the so-called \emph{K\L R conjecture}, a suitable version of which was later proved by Balogh--Morris--Samotij~\cite{BMSW16} and Saxton and Thomason~\cite{ST15}). See \cite{LZ15,Zha17} for some (very incomplete) analysis of the variational problem in \cref{eq:supercritical-general-p} for general $L$.

\subsubsection{The critical regime}
There is always a gap between the ``subcritical regime'' where \cref{eq:subcritical-general} holds, and the ``supercritical regime'' where \cref{eq:supercritical-general-m,eq:supercritical-general-p} hold. Of particular interest is the case where $H$ is \emph{strictly 2-balanced}: namely, when $H$ has at least three edges and $d_2(F)<d_2(H)$ for every proper subgraph $F\subsetneq H$ with at least two edges, where we write $d_2(F)=(e(F)-1)/(v(F)-2)$. In this case, there is only a single order of magnitude between the subcritical and supercritical regimes. Indeed, this was the case for the triangle-free probability discussed earlier in the introduction: the only regime left open was when $p$ had order of magnitude $1/\sqrt n$.

For general strictly 2-balanced $H$, the only regime left open by \cref{eq:supercritical-general-m,eq:supercritical-general-p,eq:subcritical-general} is where $p$ has order of magnitude $n^{-1/d_2(H)}$ or $m$ has order of magnitude $n^{2-1/d_2(H)}$. We call this the ``critical regime''. This regime is interesting because (unless $H$ is bipartite, as we will discuss later), the behaviour in \cref{eq:subcritical-general} is completely different from the behaviour in \cref{eq:supercritical-general-m,eq:supercritical-general-p}, and in the critical regime we must see some kind of transition between these behaviours. 

Very recently, some exciting progress at the ``lower edge'' of the critical regime was made by Jenssen, Perkins, Potukuchi and Simkin~\cite{JPPSold,JPPSnew}, first for triangles and then for all strictly 2-balanced graphs $H$. Specifically, for a fixed strictly 2-balanced graph $H$ and a fixed ``relative deviation parameter'' $z\in [0,1)$, they managed to compute the asymptotic value of $\log\Pr[X_H(G)\le z\mb EX_H(G)]$, assuming that $G\sim \mb G(n,p)$ for $p\le c_{H,z}n^{-1/d_2(H)}$ or that $G\sim \mb G(n,m)$ for $m\le c_{H,z}n^{2-1/d_2(H)}$ (for some $c_{H,z}>0$ depending only on $H$ and $z$). 
They accomplished this by exploiting connections to statistical physics: roughly speaking, they interpret the conditional distribution of $G$ given $X_H(G)\le z\mb EX_H(G)$ as an instance of a certain Gibbs measure related to the \emph{hard core model}. Via interpolation, they then reduce the approximation of the logarithm of the corresponding partition function to the estimation of marginals of the relevant Gibbs measure. In a certain regime where the ``activity parameter'' of the model is below the so-called \emph{Gibbs uniqueness threshold} for a certain tree associated with $H$, such marginals can be estimated via a fixed point of a \emph{belief propagation} operator. This uniqueness threshold dictates the upper bound on $p$ or $m$ in terms of $H$ and $z$ that is required for the arguments in \cite{JPPSold,JPPSnew} to work. 

In this paper, we are able to obtain variational formulas for lower tails in the entire critical regime, interpolating in a very simple way between the estimate in \cref{eq:subcritical-general} and the estimates in \cref{eq:supercritical-general-m,eq:supercritical-general-p}.
\begin{theorem}
    \label{thm:main-graphs}Let $H$ be a strictly 2-balanced graph which is not the 2-edge path\footnote{The 2-edge path anyway does not satisfy our definition of ``strictly 2-balanced'', but we emphasise this point since there is some inconsistency in the literature on this matter (some authors consider the 2-edge path to be strictly 2-balanced).}, and let $\varepsilon,\delta>0$. Let $X_H(G)$ be the number of copies of $H$ in a graph $G$, and let $X_H(\vec q)$ be the expected value of $X_H(G)$ for $G\sim \mb G(n,\vec q)$. Then, the following hold for sufficiently large $n$. 
    
    \begin{enumerate}
        \item Let $p\in [\delta n^{-1/d_2(H)},\delta^{-1}n^{-1/d_2(H)}]$, let $G\sim \mb G(n,p)$, and let $0\le L\le (1-\delta)\mb EX_H(G)$. Then
\[
\log \Pr[X_H(G)\le L]=-\min_{\vec{q}\in \mc Q}\Big(\on{D}_{\mr{KL}}(\vec{q}\,\|\,p)+\Phi_L(X_H(\vec q))\Big)\pm \varepsilon p\binom n2,
\]
where $\mc Q=[0,p]^{K_n}$ is the set of edge-weightings of $K_n$ with weights in the range $[0,p]$.
\item Let $m\in [\delta n^{2-1/d_2(H)},\delta^{-1}n^{2-1/d_2(H)}]$, let $G\sim \mb G(n,m)$, and let $0\le L\le (1-\delta)\mb EX_H(G)$. Then
\[
\log \Pr[X_H(G)\le L]=-\min_{\vec{q}\in \mc Q_{m}}\Big(\on{D}_{\mr{KL}}(\vec{q}\,\|\,m/{\textstyle{\binom n2}})+\Phi_L(X_H(\vec q))\Big)\pm \varepsilon m,
\]
where $\mc Q_m\subseteq[0,1]^{K_n}$ is the set of edge-weightings of $K_n$ whose weights sum to exactly $m$.
    \end{enumerate}
\end{theorem}
To see that \cref{thm:main-graphs} interpolates between \cref{eq:subcritical-general} and \cref{eq:supercritical-general-p} (and similarly for the $\mb G(n,m)$ versions), note that as we decrease $p$ the term $\on{D}_{\mr{KL}}(\vec{q}\,\|\,p)$ becomes more and more important compared to the term $\Phi_L(X_H(\vec q))$. Since $\on{D}_{\mr{KL}}(\vec{q}\,\|\,p)$ is minimised by the all-$p$ weighting, in the limit $p=o(n^{-1/d_2(H)})$ we recover the estimate in \cref{eq:subcritical-general}. On the other hand, as we increase $p$ the term $\Phi_L(X_H(\vec q))$ becomes more and more important, so we converge to the problem of minimising $\on{D}_{\mr{KL}}(\vec{q}\,\|\,p)$ subject to the constraint $\Phi_L(X_H(\vec q))=0$ (which is an equivalent way to state the variational problem in \cref{eq:supercritical-general-p}).

Similarly to previous results in the nonlinear large deviations literature, \cref{thm:main-graphs} tells us that in order to determine the lower tail rate function, it suffices to solve a purely analytic variational problem.
Unfortunately, it seems to be difficult to actually solve the relevant variational problems, even when $L=0$. However, we are able to do this in some special cases. We will discuss this in the next few subsections, but 
first it is worth observing that it is possible to rephrase the variational problems in \cref{thm:main-graphs} as ``continuum'' variational problems, as follows.
\begin{corollary}
    \label{cor:main-graphs}Fix a strictly 2-balanced graph $H$ which is not the 2-edge path, and fix $c> 0$ and $z\in [0,1)$. Let $\on{Aut}(H)$ be the automorphism group of $H$, let $t_H(W)=\int_0^1\dots \int_0^1\prod_{ij\in H}W(x_i,x_j)\on dx_1\dots\on d x_{v(H)}$ be the homomorphism density of $H$ in $W$ and let 
    \[\psi_z(t)=\begin{cases}
        t-z+z\log(z/t)&\text{if }t> z,\\
        0&\text{if }0\le t\le z,
    \end{cases}\](so, with notation as in \cref{eq:phi} we have $\psi_z(t)=t\phi(z/t)$).
    
    \begin{enumerate}
        \item Suppose $pn^{1/d_2(H)}\to c$ and $L/\mb EX_H(G)\to z$ as $n\to\infty$. Then, for $G\sim \mb G(n,p)$ we have
        \[\log \Pr[X_H(G)\le L]= -(\Psi_{\mr{Bin}}(H,c,z)+o(1))p\binom n2,
        \]
        where
        \[\Psi_{\mr{Bin}}(H,c,z)=\min_{W\in \mc W} \bigg(
        \int_{[0,1]^2}\big(W(\vec x)\log W(\vec x)-W(\vec x)+1\big)\on d\vec x+\frac{2c^{e(H)-1}}{|\on{Aut}(H)|}\psi_z(t_H(W))\bigg),\]
    with $\mc W$ being the set of all measurable symmetric functions $W:[0,1]^2\to [0,1]$.
        \item Suppose $2mn^{1/d_2(H)-2}\to c$ and $L/\mb EX_H(G)\to z$ as $n\to\infty$. Then, for $G\sim \mb G(n,m)$ we have
        \[\log \Pr[X_H(G)\le L]= -(\Psi_{\mr{Unif}}(H,c,z)+o(1))m,
        \]
        where
        \[\Psi_{\mr{Unif}}(H,c,z)=\min_{W\in \tilde{\mc W}} \bigg(
        \int_{[0,1]^2}W(\vec x)\log W(\vec x)\on d\vec x+\frac{2c^{e(H)-1}}{|\on{Aut}(H)|}\psi_z(t_H(W))\bigg),\]
        with $\tilde{\mc W}$ being the set of all measurable symmetric functions $W:[0,1]^2\to [0,\infty)$ satisfying $\int_{[0,1]^2} W(\vec x)\on d\vec x=1$.
    \end{enumerate}
\end{corollary}

\begin{remark}\label{rem:minimum-attained}
The previous variational problems we have seen all concern continuous functions on compact subsets of finite-dimensional Euclidean space, so they clearly each attain a minimum value. Here, it is less obvious that minimum values are actually attained.
To prove this, one should work with a particular metric called the \emph{cut metric} (see e.g.\ \cite{Lov12}). One can check that, with respect to the cut metric topology: the space of symmetric measurable functions $W:[0,1]^2\to [0,1]$ is compact (modulo measure-preserving relabellings), $t_H(W)$ is a continuous function of $W$, and $\int W(x,y)\log W(x,y)\on dy\on dx$ is a lower semicontinuous function of $W$. These facts together imply that a minimum value is attained in the variational problem in (1). To prove the same for (2), one needs an additional argument (along the lines of \cref{lem:truncation-lower}) to reduce to a variational problem over \emph{bounded} $W$. However, this will not actually be needed for any of the applications in this paper.
\end{remark}

One advantage of \cref{cor:main-graphs} is that $\Psi_{\mr{Bin}}(H,c,z)$ and $\Psi_{\mr{Unif}}(H,c,z)$ do not depend on $n$, so \cref{cor:main-graphs} shows that quantities of the form $\log \Pr[X_H(G)\le L]$, when appropriately normalised, tend to a limit as $n\to \infty$. There are some ``soft'' methods, based on subadditivity, for proving the existence of certain limits in sparse random graphs~\cite{BGT13}, but several authors have observed that such methods do not seem to apply in the setting of \cref{cor:main-graphs}. In particular, letting $G\sim \mb G(n,c/\sqrt n)$, Warnke~\cite{War16} had asked whether $n^{-3/2}\log\Pr[X_{K_3}(G)=0]$ converges to a limit, and Jenssen, Perkins, Potukuchi and Simkin~\cite[Question~1]{JPPSold} had later (but seemingly independently) asked the more general question of whether $n^{-3/2}\log\Pr[X_{K_3}(G)\le z\mb E X_{K_3}(G)]$ converges to a limit (in both cases we are holding $c,z$ fixed as $n\to \infty$).

In the next few sections we describe some explicit solutions to the variational problems in \cref{cor:main-graphs}, in some special cases.

\subsubsection{The Sidorenko property}\label{subsec:sidorenko}

Recalling the notation in \cref{cor:main-graphs}, we say that a graph $H$ has the \emph{Sidorenko property} if $t_H (W) \ge  t_{K_2}(W)^{e(H)}$ for all symmetric measurable functions $W:[0,1]^2\to [0,1]$.

For context, Sidorenko~\cite{Sid91} famously conjectured that every bipartite graph has the Sidorenko property (this conjecture was also independently made by Erd\H os and Simonovits~\cite{Sim84} in a slightly different form). Sidorenko's conjecture has been verified for a range of different graphs $H$ (including for example even cycles, complete bipartite graphs, and bipartite graphs with a vertex complete to the other side; see \cite{Sid91,Sid93, CFS10, CL17, Sze14, LS25}). It is easy to see that a non-bipartite graph cannot have the Sidorenko property.

If $H$ is a bipartite graph with the Sidorenko property, it turns out to be very easy to solve the $z=0$ cases of the variational problems in \cref{cor:main-graphs}. Namely, the optimising function $W$ is always constant-valued.
\begin{proposition}\label{prop:sidorenko}
    Let $H$ be a strictly 2-balanced graph with the Sidorenko property and let $c> 0$. Then
    \[\Psi_{\on{Bin}}(H,c,0)=\min_{s\in [0,1]}\bigg({s\log s-s+1+\frac{2c^{e(H)-1}}{|\on{Aut}(H)|}s^{e(H)}}\bigg)\]
    and
    \[\Psi_{\on{Unif}}(H,c,0)=\frac{2c^{e(H)-1}}{|\on{Aut}(H)|}.\]
\end{proposition}
We remark that in the above formula for $\Psi_{\on{Bin}}(H,c,0)$, the minimising value of $s$ cannot be expressed in terms of elementary functions, but it can be expressed in terms of the so-called \emph{Lambert $W$ function} (which satisfies $x=W(x)\exp(W(x))$). Also, we remark that the above formula for $\Psi_{\on{Unif}}(H,c,0)$ implies that in the setting of \cref{cor:main-graphs}(2) we have the Poisson-type tail bound $\log \Pr[X_H(G)=0]=-(1+o(1))\mb E X_H(G)$.

In the setting of \cref{prop:sidorenko}, it is easy to see that $\Psi_{\on{Bin}}(H,c,0)$ and $\Psi_{\on{Unif}}(H,c,0)$ are both analytic functions of $c$ for the entire range $c\ge 0$; in particular we do not see the kind of sharp phase transition illustrated in \cref{fig:triangle-kink}. In the case where $H$ is an even cycle, this confirms a conjecture of Jenssen, Perkins, Potukuchi and Simkin~\cite[Conjecture~1]{JPPSnew}.

Actually, by combining\footnote{To apply these results for totally general $p$, we need to use the fact that for any infinite subsequence of values of $n$ there is a subsubsequence along which $pn^{1/d_2(H)}$ tends to a limit in $[0,\infty ]$. Then we need to use the ``subsubsequence principle'', which says that if for every subsequence of a sequence there is a subsubsequence converging to a limit $a$, then the entire sequence must converge to the same limit.} \cref{prop:sidorenko,cor:main-graphs} with \cref{eq:supercritical-general-p,eq:subcritical-general}, one can prove that there is a single formula that describes the asymptotic $H$-freeness probability for \emph{all} $p\in (0,1)$: namely, for $G\sim \mb G(n,p)$ we have
\[\log \Pr[G\text{ is }H\text{-free}]=(1+o(1))\frac{\log(1-p)}p\cdot \min_{s\in [0,1]}\bigg(p\binom n2(s\log s-s+1)+\mb E X_H(G)\cdot  s^{e(H)}\bigg).\]

\subsubsection{General graphs in the early critical regime}\label{subsec:early-critical}
It was proved by Jenssen, Perkins, Potukuchi and Simkin~\cite[Corollary~1.3]{JPPSnew} that whenever $H$ is strictly 2-balanced and \emph{not} bipartite, the $H$-freeness probability experiences at least one phase transition: in our notation (given \cref{cor:main-graphs}), it is \emph{not} the case that $\Psi_{\mr{Bin}}(H,c,0)$ is an analytic function of $c$ in the entire range $c\in [0,\infty)$. In particular the conclusion of \cref{prop:sidorenko} does not hold for all $c$.

However, we are able to show that for totally general graphs, the conclusion of \cref{prop:sidorenko} does hold for \emph{sufficiently small} $c$.
\begin{theorem}\label{thm:H-early}
    Let $H$ be a strictly 2-balanced graph with $e(H)=h$ and $|\on{Aut}(H)|=a$, and let $0< c\le (ae/(2h))^{1/(h-1)}$. Then
    \[\Psi_{\on{Bin}}(H,c,0)=\min_{s\in [0,1]}\bigg({s\log s-s+1+\frac{2c^{e(H)-1}}{|\on{Aut}(H)|}s^{e(H)}}\bigg).\]
\end{theorem}
A similar theorem (for the $H$-free probability in a random graph $G(n,cn^{-1/d_2(H)})$) was proved by Jenssen--Perkins--Potukuchi--Simkin~\cite[Theorem~1.2]{JPPSnew}, though their result requires the stronger assumption $c\le (ae/(2h(h-1)))^{1/(h-1)}$. In particular, combining \cref{thm:H-early} with \cref{cor:main-graphs}, we surpass the so-called \emph{Gibbs uniqueness threshold} which was a barrier for the methods in \cite{JPPSnew}.

\subsubsection{Triangles}\label{subsec:triangles}
As mentioned earlier in the introduction, 
we were also able to solve the variational problem defining $\Psi_{\on{Bin}}(K_3,c,0)$ for all $c\ge 0$, thereby fully characterising the probability that a random graph $\mb G(n,p)$ is triangle-free. Namely, it turns out that the variational problem defining $\Psi_{\on{Bin}}(K_3,c,0)$ is always minimised by a function $W_{r,t}\in \mc W$ that takes one value $r$ on the set $[0,1/2]^2\cup [1/2,1]^2$, and another (possibly the same) value $t$ on the set\footnote{Strictly speaking this is not well-defined on the boundaries, but measure-zero sets are irrelevant in this context.} $([0,1/2]\times [1/2,1])\cup ([1/2,1]\times [0,1/2])$.

\begin{theorem}
\label{thm:triangles}
For all $c> 0$ we have
\[\Psi_{\on{Bin}}(K_3,c,0)=\min_{r,t\in [0,1]}\Bigg(\frac{(r\log r-r+1)+(t\log t-t+1)}2+\frac{c^2}{3}\bigg(\frac14r^3+\frac34rt^2\bigg)\Bigg).\]
\end{theorem}
The value of $\Psi_{\on{Bin}}(K_3,c,0)$ is illustrated in \cref{fig:triangle-kink}; note that there is a very visible phase transition at some critical point $c^*\approx 4.341$. To understand why this occurs, we need to consider the stationary points of the two-variable function $F:(r,t)\mapsto (r\log r-r+1)/2+(t\log t-t+1)/2+(c^2/3)((1/4)r^3+(3/4)rt^2)$, which we depict in \cref{fig:triangle-stationary}. When $c$ is small (namely, when $c$ is less than about $3.350$), there is only one stationary point $(r,t)$, which is constant-valued (i.e., $r=t$). However, when $c$ is larger there are two additional non-constant stationary points. One of these additional stationary points is always a local minimum, and there is a critical value of $c$ (about 4.341) beyond which this additional local minimum overtakes the constant-valued stationary point to become a global minimum.

It is a simple calculus exercise to make the above discussion rigorous. Namely, one can prove that $\Psi_{\on{Bin}}(K_3,c,0)$ is an analytic function of $c$ except at a single point $c^*$, where the derivative is discontinuous. This addresses a conjecture of Jenssen, Perkins, Potukuchi and Simkin~\cite[Conjecture~1]{JPPSold} (their conjecture in that paper was slightly more general, but the special case addressed here was explicitly highlighted in \cite{Per26}).

\begin{figure}[tb]
  \centering
\begingroup
\definecolor{branchblue}{HTML}{1F77B4}
\definecolor{branchorange}{HTML}{FF7F0E}
\definecolor{branchgreen}{HTML}{2CA02C}
\begin{tikzpicture}[
  declare function={
    % On a non-diagonal branch: u = log(t/r), L(u) = -log(t).
    branchL(\u)=2*\u*exp(-\u)/(1-exp(-\u))^2;
    branchA(\u)=2*\u*exp(2*branchL(\u))/(3*(1-exp(-\u))^2);
    branchC(\u)=sqrt(3*branchA(\u));
    branchR(\u)=exp(-\u-branchL(\u));
    branchT(\u)=exp(-branchL(\u));
  }
]
\begin{groupplot}[
  group style={group size=2 by 1, horizontal sep=0.20\linewidth},
  scale only axis,
  width=0.34\linewidth,
  height=0.30\linewidth,
  xmin=0, xmax={sqrt(30)},
  % Identical linear y-scales; slight padding keeps r close to zero visible.
  ymin=-0.02, ymax=1.02,
  xtick={0,1,2,3,4,5},
  ytick={0,0.2,0.4,0.6,0.8,1},
  xlabel={$c$},
  tick align=outside,
  tick label style={font=\small},
  label style={font=\small},
  % Cancel the default 90-degree rotation of the vertical-axis labels.
  ylabel style={rotate=-90},
  grid=major,
  major grid style={black!12},
  axis line style={black!70},
  every axis plot/.append style={line width=0.9pt, no markers},
  samples=251,
  variable=\u,
]
\nextgroupplot[ylabel={$r$}]
% Reference value c = sqrt(3*6.280969) = 4.34084173865... (unlabelled).
\addplot[black!65, dotted, line width=0.7pt, forget plot]
  coordinates {({sqrt(3*6.280969)},-0.02) ({sqrt(3*6.280969)},1.02)};
% Diagonal branch: c = sqrt(v) exp(v), r = t = exp(-v).
\addplot[branchblue, domain=0:1.498399816117]
  ({sqrt(\u)*exp(\u)},{exp(-\u)});
% Two non-diagonal branches; the parameter ranges meet at their common endpoint.
\addplot[branchorange, domain=1.761273229251:4.185469869036]
  ({branchC(\u)},{branchR(\u)});
\addplot[branchgreen, domain=4.185469869036:14.999715437629]
  ({branchC(\u)},{branchR(\u)});

\nextgroupplot[ylabel={$t$}]
% Reference value c = sqrt(3*6.280969) = 4.34084173865... (unlabelled).
\addplot[black!65, dotted, line width=0.7pt, forget plot]
  coordinates {({sqrt(3*6.280969)},-0.02) ({sqrt(3*6.280969)},1.02)};
\addplot[branchblue, domain=0:1.498399816117]
  ({sqrt(\u)*exp(\u)},{exp(-\u)});
\addplot[branchorange, domain=1.761273229251:4.185469869036]
  ({branchC(\u)},{branchT(\u)});
\addplot[branchgreen, domain=4.185469869036:14.999715437629]
  ({branchC(\u)},{branchT(\u)});
\end{groupplot}
\end{tikzpicture}
\endgroup
  \caption{Stationary points of the objective function $F(r,t)$. The value of $r$ is depicted on the left and the value of $t$ is depicted on the right. The vertical dotted line indicates the critical point $c^*\approx 4.341$ at which the minimiser transitions between two different branches.}
  \label{fig:triangle-stationary}
\end{figure}

\begin{remark}\label{rem:structure-intro}
    In \cite{JPPSold}, in addition to investigating $\Pr[G\text{ is triangle-free}]$, the authors were also interested in the conditional distribution of $G$ given that it is triangle-free. Namely, they were interested in the typical max-cut ratio for a graph $G$ sampled from this conditional distribution, and they conjectured that this typical max-cut ratio changes abruptly at the same critical point $c^*\approx 4.341$. We believe it should be possible to describe the typical max-cut ratio with our methods (by showing that the event that $G$ is triangle-free and also has an atypical max-cut ratio is much less likely than the event that $G$ is triangle-free), but we have not investigated this in detail. See \cref{rem:structure} for some brief comments on how to adapt our proof approach to study structural properties of conditional distributions.
\end{remark}

The proof of \cref{thm:triangles} is computer-assisted (involving a large amount of interval arithmetic to certify a particular four-variable inequality). 
It seems that in principle one can use our approach to attack the variational problems in \cref{cor:main-graphs} for more general $H$ and $z$ (and for $\mb G(n,m)$ as well as $\mb G(n,p)$). 
We have decided not to pursue this direction further in the present paper, but we think it would be very interesting to develop our approach into a more systematic computational framework. 

\subsubsection{Edge-critical graphs in the late critical regime}
In \cref{thm:H-early}, we have seen how to solve the variational problem for the $H$-freeness probability whenever $c$ is sufficiently small. It turns out that in quite a general context we can also solve this variational problem when $c$ is sufficiently large. Specifically, we can do this for any graph $H$ which is \emph{edge-critical}, in the sense that one can delete an edge to decrease the chromatic number (this condition often appears in the study of $H$-free graphs, perhaps most famously in \emph{Simonovits' theorem}~\cite{Sim68}).

For large $c$, our conclusion will be a generalisation of the conclusion of \cref{thm:triangles}: the variational problem defining $\Psi_{\on{Bin}}(H,c,0)$ is minimised by a function $W_{r,t}^{(k)}\in \mc W$ of a certain ``block'' form. Namely, let $I_i=[(i-1)/k,i/k]$, and let $W_{r,t}^{(k)}:[0,1]^2\to [0,1]$ be the function satisfying
\[W_{r,t}^{(k)}(x,y)=\begin{cases}
    r&\text{when $x,y$ lie in a common part $I_i$,}\\
    t&\text{when $x,y$ lie in different parts $I_i$.}
\end{cases}\]

\begin{theorem}\label{thm:H-late}
    Let $H$ be a strictly 2-balanced edge-critical graph with chromatic number $k+1$. Then there is $c_H>0$ such that the following holds. For all $c\ge c_H$ we have
\[\Psi_{\on{Bin}}(H,c,0)=\min_{r,t\in [0,1]}\Bigg(\frac1k(r\log r-r+1)+\bigg(1-\frac1k\bigg)(t\log t-t+1)+\frac{2c^{e(H)-1}}{|\on{Aut}(H)|}t_H(W_{r,t}^{(k)})\Bigg).\]
\end{theorem}
\begin{remark}\label{rem:edge-critical-necessary}
    \cref{thm:H-late} is \emph{not} true in general without the edge-criticality assumption. Indeed, let $H$ be the graph with 6 vertices $a_1,a_2,a_3,b_1,b_2,b_3$, all edges of the form $a_ia_j$ and $b_ib_j$ and $a_ib_i$, and one additional edge $a_1b_2$ (this is 10 edges in total). One can check that this graph is strictly 2-balanced and has chromatic number 3, but for large $c$ the variational problem defining $\Psi_{\on{Bin}}(H,c,0)$ is not minimised by any $W^{(2)}_{r,t}$. One way to see this is to consider a slight generalisation of the functions $W^{(2)}_{r,t}$, in which we allow two different within-part densities (instead of taking both within-part densities to be the same value $r$). This example was found by ChatGPT 6 Pro.
\end{remark}
\begin{remark}
    To keep the proof of \cref{thm:H-late} as simple as possible, we wrote it in a non-quantitative way and therefore it does not give an explicit value of $c_H$. But, it should be possible to adapt this proof to obtain an explicit value of $c_H$, if desired. 
\end{remark}

\subsection{General theorems for hypergraph lower tails}\label{subsec:general}
In this section we state some general theorems about lower tails in random subsets of hypergraphs. Namely, given a hypergraph satisfying certain ``boundedness'' conditions (of the type appearing in hypergraph container theorems~\cite{ST15,BJM15} or in the random greedy hypergraph independent set process~\cite{BB16}), we characterise the lower tail large deviation rate function of the number of hyperedges in a random vertex subset (either a random subset of a fixed size, or a binomial random subset where each vertex is included independently with some probability). By defining appropriate hypergraphs, we will be able to use the general theorems in this section to deduce all the results in \cref{subsec:subgraphs} on subgraph counts in random graphs.

First, we introduce some general hypergraph notation. For an $N$-vertex hypergraph $\mathcal{H}$ (viewed as a set of hyperedges):
\begin{itemize}
\item Let $\on V(\mathcal{H})$ be the vertex set of $\mathcal{H}$. 
When $\mathcal{H}$ is clear from
context, we often just write $V$ instead of $\on V(\mathcal{H})$.
\item For a vertex subset $W\subseteq V$ let $\mathcal{H}[W]$ be the subgraph
of $\mathcal{H}$ induced by $W$.
\item For $S\subseteq V$, let $\deg_{\mathcal{H}}(S)$ be the number of
hyperedges of $\mathcal{H}$ that include $S$.
\item Let $\Delta_{s}(\mathcal{H})$ be the maximum of $\deg_{\mathcal{H}}(S)$ over all $S\subseteq V$
of size $s$.
\item For vertices $u,v\in V$, let $\on{pair-deg}_{\mathcal{H}}(u,v)$ be the number of 
sets $S\subseteq V\setminus\{u,v\}$ such that $S\cup \{u\}$ and $S\cup \{v\}$ are both hyperedges of $\mc H$.
\item Let $\Gamma(\mc H)$ be the maximum of $\on{pair-deg}_{\mathcal{H}}(u,v)$ over all distinct vertices $u,v\in V$.
\item For $m\in \{1,\dots,N\}$, let $\on{Unif}(V,m)$ be the uniform probability distribution on subsets of $V$ of size exactly $m$.
\item For $p\in[0,1]$, let $\on{Bin}(V,p)$ be the probability distribution on subsets of $V$ in which each element is included with probability $p$ independently.
\item More generally, for $\vec{q}=(q_{v})_{v\in \on V(\mathcal{H})}\in[0,1]^{V}$, let $\on{Bin}(V,\vec{q})$ be the probability distribution on subsets of $V$ obtained by including each element $v\in V$
with probability $q_{v}$ independently.
\item For $\vec{q}\in[0,1]^{V}$, 
let $\mc H(\vec q)=\mb E|\mc H[R]|$, where $R\sim \on{Bin}(V,\vec q)$.
\end{itemize}
We also recall some notation from \cref{subsec:subgraphs}: for a finite set $V$ and $\vec q=(q_v)_{v\in V}\in [0,1]^V$ and $p\in [0,1]$, we write
\[\on{D}_{\mr{KL}}(\vec q\,\|\,p)=\on{D}_{\mr{KL}}(\on{Bin}(V,\vec q)\,\|\,\on{Bin}(V,p))=\sum_{v\in V}\left(q_v\log\left(\frac {q_v}p\right)+ (1-q_v)\log\left(\frac {1-q_v}{1-p}\right)\right),\]
and for $L,M\ge 0$, let
 \[\Phi_L(M)=\min_{\ell\in [0,L]}\on{D}_{\mr{KL}}(\on{Poisson}(\ell)\,\|\,\on{Poisson}(M))=\phi(L/M)M,\]
where
\[\phi(z)=\begin{cases}
1-z+z\log z&\text{if }0\le z\le 1,\\0&\text{ if }z>1.\end{cases}\]
(in all these definitions we use the conventions $0\log 0=0$ and $0/0=1$ and $L/0=\infty$ for $L>0$.)

Then, our boundedness condition is as follows.
\begin{definition}\label{def:bounded}
Say that an $N$-vertex $r$-uniform hypergraph $\mathcal{H}$ is \emph{$(\lambda,C;p)$-degree-bounded}
if
\begin{align*}
\Delta_{1}(\mathcal{H}) \le C p^{1-r},\qquad 
\Delta_{s}(\mathcal{H}) \le C \lambda p^{s-r}\text{ for all }s\in\{2,\dots,r-1\},\qquad 
\Gamma(\mc H)\le C\lambda p^{1-r}.
\end{align*}
\end{definition}

We can now state the promised general lower tail estimates. We start with upper bounds on the lower tail.
\begin{theorem}
\label{thm:main-upper}For any $r\in \mb N_{\ge 2}$ and $C,\varepsilon>0$, there are $\lambda>0$ and $N_0\in \mb N$ such that the following holds. Consider $p\in (0,\lambda)$ and let $\mathcal{H}$ be an $N$-vertex $(\lambda,C;p)$-degree-bounded $r$-uniform hypergraph, with $N\ge N_0$. Then, for $R\sim \on{Bin}(V,p)$ and any $L\ge 0$,
\[
\log \Pr\big[|\mc H[R]|\le L\big]\le -\min_{\vec q\in [0,p]^V}\Big(\on{D}_{\mr{KL}}(\vec{q}\,\|\,p)+\Phi_L(\mc H(\vec q))\Big)+\varepsilon pN.
\]
\end{theorem}

\begin{theorem}
\label{thm:main-m-upper}For any $r\in \mb N_{\ge 2}$ and $C,\varepsilon>0$, there are $\lambda>0$ and $m_0\in \mb N$ such that the following holds. Consider $m\in [m_0,\lambda N]$, and let $\mathcal{H}$ be an $N$-vertex $(\lambda,C;m/N)$-degree-bounded $r$-uniform hypergraph. Then, for $R\sim\on{Unif}(V,m)$ and $L\ge 0$,
\[
\log \Pr\big[|\mc H[R]|\le L\big]\le  -\min_{\vec{q}}\Big(\on{D}_{\mr{KL}}(\vec{q}\,\|\,m/N)+\Phi_L(\mc H(\vec q))\Big)+\varepsilon m,
\]
where the minimum is over all $\vec q\in [0,1]^V$ whose sum of entries is exactly $m$.
\end{theorem}

Next, we state matching lower bounds on the lower tail, which require slightly stronger assumptions (roughly speaking, we need to make a quantitative assumption on how fast $\lambda$ decays with $N$).

\begin{theorem}
\label{thm:main-lower}For any $r\in \mb N_{\ge 2}$ and $C,\varepsilon,\nu>0$, the following holds for sufficiently large $N$. Let $p\in(N^{\nu-1},N^{-\nu})$, and let $\mathcal{H}$ be an $N$-vertex $(N^{-\nu},C;p)$-degree-bounded $r$-uniform hypergraph. Then, for $R\sim \on{Bin}(V,p)$ and $L\ge 0$,
\[
\log \Pr\big[|\mc H[R]|\le L\big]\ge -\min_{\vec{q}\in [0,p]^V}\Big(\on{D}_{\mr{KL}}(\vec{q}\,\|\,p)+\Phi_L(\mc H(\vec q))\Big)-\varepsilon pN.
\]
\end{theorem}

\begin{theorem}
\label{thm:main-m-lower}For any $r\in \mb N_{\ge 2}$ and $C,\varepsilon,\nu>0$, the following holds for sufficiently large $N$. Let $m\le N$ be such that $m/N \in(N^{\nu-1},N^{-\nu})$, and let $\mathcal{H}$ be an $N$-vertex $(N^{-\nu},C;m/N)$-degree-bounded $r$-uniform hypergraph. Then, for $R\sim\on{Unif}(V,m)$ and $L\ge 0$,
\[
\log \Pr\big[|\mc H[R]|\le L\big]\ge  -\min_{\vec{q}}\Big(\on{D}_{\mr{KL}}(\vec{q}\,\|\,m/N)+\Phi_L(\mc H(\vec q))\Big)-\varepsilon m,
\]
where the minimum is over all $\vec q\in [0,1]^V$ whose sum of entries is exactly $m$.
\end{theorem}

\begin{remark}\label{rem:log-factors}
We believe that it should not really be necessary to have stronger assumptions for the lower bound than the upper bound: the assumptions of \cref{thm:main-upper} should be enough for the conclusion of \cref{thm:main-lower}, and similarly the assumptions of \cref{thm:main-m-upper} should be enough for the conclusion of \cref{thm:main-m-lower}. The reason \cref{thm:main-lower,thm:main-m-lower} need stronger quantitative assumptions is that they depend on analysis of a certain random process by Guo and Warnke~\cite{GW}, which needs a certain amount of slack in its error terms. However, the Guo--Warnke analysis considers a slightly more general setting than we really need, and in our setting we do not believe that this slack is really essential. 
\end{remark}

Apart from the applications to subgraph counts in random graphs described in \cref{subsec:subgraphs} (and obvious generalisations to subgraph counts in random hypergraphs), one interesting context in which \cref{thm:main-m-lower,thm:main-m-upper,thm:main-upper,thm:main-lower} can be applied is the study of arithmetic progressions (and other configurations) in random subsets of the integers. In the interest of keeping the length of this paper under control, we will not go into details here; see for example \cite{KS23,JPPSnew}.

\subsection{Applications to combinatorial design theory}\label{subsec:design-theory}
Our method is quite flexible, and can be applied in some situations that do not quite fit into the framework of \cref{subsec:general}. We record two further such applications that we find to be of particular interest. These concern \emph{Steiner triple systems} and \emph{Latin squares}, important objects in combinatorial design theory.
\subsubsection{Steiner triple systems}
A \emph{Steiner triple system} of order $n$ is a 3-uniform hypergraph on the vertex set $\{1,\dots,n\}$, such that each pair of vertices is included in exactly one triple. A \emph{Pasch configuration} in a Steiner triple system is a set of 4 triples which only involve 6 vertices (this configuration is unique up to isomorphism, see \cref{fig:pasch}). A Steiner triple system is said to be \emph{anti-Pasch} if it has no Pasch configuration.

\begin{figure}
\tikzstyle{vertex}=[circle,draw=black,fill=black,inner sep=0,minimum size=0.2cm,text=white,font=\footnotesize]
\begin{centering}
\begin{tikzpicture}[scale=0.4,baseline=5pt]

\def\rot{0}
\def\dir{-1}
\foreach \x in {1,...,3}{%
  \coordinate (\x-1) at ({\rot+\dir*360/3 * (\x - 1)-90}:1) {};
  \coordinate (\x-2) at ({\rot+\dir*360/3 * (\x + 0.5)-90}:4) {};
  \node [vertex] at (\x-1) {};
  \node [vertex] at (\x-2) {};
}
\begin{scope}[opacity=0.25, transparency group]
\fill (1-1) -- (2-1) -- (3-1) -- (1-1);
\draw[line width=3mm,line join=round,line cap=round] (1-1) -- (2-1) -- (3-1) -- (1-1);
\end{scope}

\begin{scope}[opacity=0.25, transparency group]
\fill (1-1) -- (2-2) -- (3-2) -- (1-1);
\draw[line width=3mm,line join=round,line cap=round] (1-1) -- (2-2) -- (3-2) -- (1-1);
\end{scope}

\begin{scope}[opacity=0.25, transparency group]
\fill (2-1) -- (3-2) -- (1-2) -- (2-1);
\draw[line width=3mm,line join=round,line cap=round] (2-1) -- (3-2) -- (1-2) -- (2-1);
\end{scope}

\begin{scope}[opacity=0.25, transparency group]
\fill (3-1) -- (1-2) -- (2-2) -- (3-1);
\draw[line width=3mm,line join=round,line cap=round] (3-1) -- (1-2) -- (2-2) -- (3-1);
\end{scope}

\end{tikzpicture}
\qquad\qquad\qquad\qquad
\begin{tabular}{|c|c|c|c|c|}
\hline 
4 & \cellcolor{black!25}\textbf{1} & 5 & \cellcolor{black!25}\textbf{3} & 2\tabularnewline
\hline 
5 & \cellcolor{black!25}\textbf{3} & 2 & \cellcolor{black!25}\textbf{1} & 4\tabularnewline
\hline 
2 & 4 & 1 & 5 & 3\tabularnewline
\hline 
3 & 5 & 4 & 2 & 1\tabularnewline
\hline 
1 & 2 & 3 & 4 & 5\tabularnewline
\hline 
\end{tabular}

\par\end{centering}
\caption{\label{fig:pasch}On the left is a Pasch configuration, with 4 triples on 6 vertices. This is the smallest configuration $H$ such that the number of instances of $H$ can vary among different Steiner triple systems of the same order. On the right is a $2\times 2$ Latin subsquare in an order-5 Latin square. This can be interpreted as a Pasch configuration in the $3$-partite $3$-uniform hypergraph associated with the Latin square.}
\end{figure}

The set of $n$ for which order-$n$ anti-Pasch Steiner triple systems exist was determined in a sequence of several papers by different authors~\cite{Doy81,Bro77,GGW00,Lin00}. Answering a conjecture of Grable and Phelps~\cite{GP96}, it was proved by Simkin~\cite{Sim18} (using results of Keevash~\cite{Kee18,Kee14} and Kwan~\cite{Kwa20}) that only a vanishing proportion of Steiner triple systems are anti-Pasch. It was conjectured by Glock, K\"uhn, Lo and Osthus~\cite{GKLO20} that the number of order-$n$ anti-Pasch Steiner triple systems is $n^{n^2/6}\exp\big(-(9/24)n^2+o(n^2)\big)$ (they also made a more general conjecture about the number of order-$n$ Steiner triple systems avoiding ``short cycles'' up to a given size). Bohman and Warnke~\cite{BW19} speculated the same (but did not explicitly call it a conjecture).

The lower bound in this conjecture was recently proved by Kwan, Sah, Sawhney and Simkin~\cite[Theorem~1.3]{KSSS24}, and we can obtain a corresponding upper bound using the ideas in \cref{thm:main-m-upper}.
\begin{theorem}\label{thm:STS}
    Let $n$ be such that an order-$n$ Steiner triple system exists. The number of order-$n$ anti-Pasch Steiner triple systems is
    \[n^{n^2/6}\exp\big(-(9/24)n^2+o(n^2)\big).\]
    Equivalently, a uniformly random order-$n$ Steiner triple system is anti-Pasch with probability
    \[\exp\big(-(1/24)n^2+o(n^2)\big).\]
\end{theorem}
(For the equivalence of the two statements in this theorem, we are using the fact that there are $n^{n^2/6}\exp\big(-(1/3)n^2+o(n^2)\big)$ different order-$n$ Steiner triple systems, which was proved in a combination of papers by Linial--Luria~\cite{LL13} and Keevash~\cite{Kee18}).
\begin{remark}
Heuristic calculations suggest that the expected number of Pasch configurations in a random order-$n$ Steiner triple system is $(1/24)n^2+o(n^2)$, which would make \cref{thm:STS} a ``Poisson-type'' tail bound. It is known that this expected number is \emph{at least} $(1/24)n^2+o(n^2)$ (see \cite{Sim18}), but it is surprisingly difficult to prove the corresponding upper bound. Some general machinery to compute expected values of this type is currently work in progress by the first author with Lenka Kopfov\'a, Richard Montgomery and Kalina Petrova.
\end{remark}

\begin{remark}\label{rem:high-girth}
To elaborate on the more general conjecture of Glock--K\"uhn--Lo--Osthus mentioned above: recall that a Pasch configuration has 4 triples and only 6 vertices. In general, an \emph{Erd\H os $k$-configuration} has $k$ triples and only $k+2$ vertices. Without going into detail, it turns out that Erd\H os $k$-configurations in a Steiner triple system are in certain respects analogous to $k$-cycles in a graph, and it is of interest to understand Steiner triple systems with ``girth'' at least $g$ (i.e., having no Erd\H os $k$-configuration for $4\le k<g$). Glock--K\"uhn--Lo--Osthus had conjectured an asymptotic formula for the number $\on{STS}_g(n)$ of order-$n$ Steiner triple systems with girth at least $g$, and Kwan, Sah, Sawhney and Simkin~\cite{KSSS24} managed to prove the corresponding lower bound (i.e., $\on{STS}_g(n)$ is at least as large as conjectured). Using the ideas in this paper, we can obtain upper bounds on $\on{STS}_g(n)$ in terms of certain analytic variational problems. We are unfortunately not able to solve these variational problems in general, but surprisingly it seems that for sufficiently large $g$ these variational problems yield upper bounds on $\on{STS}_g(n)$ that are strictly weaker than conjectured. This was discovered by ChatGPT 6 Pro, and we have made the details available at \url{https://hexagonmath.org/2611.00166}; it is related to the fact that the $3\times 3$ grid hypergraph does not have a certain ``positivity'' property~\cite{Sid22}, but is very numerous among Erd\H os $k$-configurations for large $k$. This observation does not disprove the Glock--K\"uhn--Lo--Osthus conjecture (as our methods do not provide a corresponding lower bound), but we view this as strong evidence that the Glock--K\"uhn--Lo--Osthus conjecture is false.
\end{remark}

\subsubsection{Latin squares}\label{subsubsec:latin}
    An order-$n$ \emph{Latin square} is an $n\times n$ array filled with the symbols $\{1,\dots,n\}$, such that every symbol appears exactly once in each row and in each column. A Latin square is said to have \emph{property $N_2$} if it has no $2\times 2$ Latin subsquare (see \cref{fig:pasch}). It is said to have \emph{property $N_\infty$} if it has no nontrivial Latin subsquare at all (i.e., no Latin subsquare whose order is not 1 or $n$).
    
The set of $n$ for which there exist order-$n$ Latin squares with property $N_2$, or with property $N_\infty$, was determined in a sequence of many papers by different authors \cite{McL75,KLR75,KT76,  Wan01,MWW07,AM82,Hei80,AW23}. (In particular, the $N_\infty$ problem was fully resolved only very recently by Allsop and Wanless~\cite{AW23}, answering an old conjecture of Hilton). Enumeration questions in this direction were first considered by McKay and Wanless~\cite{MW99} (in the language of random Latin squares). It has been conjectured that the number of order-$n$ Latin squares with property $N_2$, and the number with property $N_\infty$, are both of the form $n^{n^2}\exp\big(-(9/4)n^2+o(n^2)\big)$ (for property $N_2$, this conjecture was first made by Kwan, Sah, Sawhney and Simkin~\cite[Conjecture~1.6]{KSSS23}, and for property $N_\infty$ this conjecture was seemingly first made by Kelly~\cite[Conjecture D(b)]{Kel+}). The lower bound in the $N_2$ conjecture was established by Kwan, Sah, Sawhney and Simkin~\cite[Theorem~1.1]{KSSS23}, but there is still no matching lower bound for the $N_\infty$ conjecture (the best lower bound is of the form $n^{n^2-o(n^2)}$, featuring in very recent work of Postle and Lesgourgues~\cite{LP26}). 

We can prove the conjectured upper bound (for the $N_2$ problem, which implies the same upper bound for the $N_\infty$ problem)  using basically the same approach as for \cref{thm:STS}.
\begin{theorem}\label{thm:latin}
    The number of order-$n$ Latin squares with property $N_2$ is
    \[n^{n^2}\exp\big(-(9/4)n^2+o(n^2)\big).\]
    Equivalently, a uniformly random order-$n$ Latin square has property $N_2$ with probability
    \[\exp\big(-(1/4)n^2+o(n^2)\big).\]
\end{theorem}
(For the equivalence of the two statements in this theorem, we are using the fact that there are $n^{n^2}\exp\big(-2n^2+o(n^2)\big)$ different order-$n$ Latin squares, which is a consequence of permanent estimates of Bregman and of Egorychev and Falikman; see \cite[Theorem~17.3]{vLW01}).
\begin{remark}
    Answering a question of McKay and Wanless, it was proved by Kwan, Sah and Sawhney~\cite{KSS22} (see also \cite{MW99,KS18}) that the expected number of $2\times 2$ Latin subsquares in a random Latin square is $(1/4)n^2+o(n^2)$, so \cref{thm:latin} is a ``Poisson-type'' tail bound.
\end{remark}

\begin{remark}
    To explain the connection between \cref{thm:STS} and \cref{thm:latin}: a Latin square can be equivalently viewed as a set of $n^2$ triples in the 3-partite 3-uniform hypergraph $K_{n,n,n}^{(3)}$, such that no two triples share more than one vertex (the three parts correspond to the rows, columns and symbols in the usual definition of a Latin square). In this language, a $2\times 2$ Latin subsquare is the same thing as a Pasch configuration.
\end{remark}

\subsection{Relation to the upper tail}
In this paper we only focus on lower tail large deviations. 
However, it is worth taking a moment to briefly discuss the upper tail. 
It turns out that for the kinds of combinatorial random variables we consider in this paper, the upper tail has totally different behaviour to the lower tail and its study presents different challenges. A key point of difference is that the upper tail is often controlled by ``localised'' events, involving small structures that comprise only an $o(1)$-fraction of our random object yet have an outsized impact on our random variable under consideration. For example, suppose we are interested in the number of triangles in a random graph $G\sim \mb G(n,p)$, for $p=o(1)$. If $G$ contains a clique on $\ell$ vertices, this provides $\binom \ell3$ triangles, which can be significantly larger than the expected number of triangles even when $\ell=o(n)$. This type of phenomenon cannot happen for the lower tail.

Despite this fundamental difference, there are some analogies between the upper and lower tails. Similarly to the lower tail, for the upper tail  there is often a ``subcritical'' regime in which upper tails behave in a ``Poisson-type'' way, and a ``supercritical'' regime in which the upper tail behaviour is determined by the solution to a mean-field variational problem. This latter regime is usually called the ``localised'' regime, when the solution to the variational problem concerns small substructures. Specialising to the setting of $H$-counts in random graphs (for a fixed connected regular 
graph $H$): following work by many authors (e.g.\ \cite{CD16, Eld18, Aus19, Aug20,CD20, HMS22, BB23, CDP24,LZ17, BGLZ17,LZ15,CV11}), we now know the entire localised regime in which upper-tail large deviations are described by the mean-field variational problem, and this variational problem has been solved for essentially the entire localised regime.
For example, in the case where $H=K_3$ is a triangle, it was proved by Harel, Mousset and Samotij~\cite{HMS22} that if $G\sim \mb G(n,p)$ then the asymptotics of $\log \Pr[X_{K_3}(G)\ge (1+\eta)\mb E X_{K_3}(G)]$ are described by the mean-field variational problem for all $p\gg \log n/n$, while for $p\ll \log n/n$ the asymptotics of $\log \Pr[X_{K_3}(G)\ge (1+\eta)\mb E X_{K_3}(G)]$ have a Poisson-type form. These same authors also made a conjecture for the ``critical'' regime where $p$ has order of magnitude $\log n/n$.
Note that this critical regime for the upper tail problem has a totally different location than the critical regime for the lower tail problem (which, as we have seen, occurs when $p$ has order of magnitude $1/\sqrt n$).

While the work of Harel, Mousset and Samotij~\cite{HMS22} follows a combinatorial moment-based approach (going back to Janson, Oleszkiewicz, and Ruci\'nski~\cite{JOR}), it is also worth mentioning the different approach of Cook, Dembo and the second author in \cite{CDP24}, which proceeds via the development of an efficient sparse regularity method (this approach has weaker quantitative aspects, but applies in more general settings; indeed, it can simultaneously handle joint upper and lower tails in graphs and hypergraphs). In this paper we do not explicitly use the regularity method, but we use tools of Kozma--Samotij~\cite{KS23} which are very closely related to the regularity method (and actually, an earlier version of our proofs did employ the sparse regularity method). We hope that some of the ideas in the present work might be helpful for further progress in nonlinear large deviations, both for lower and upper tails.

\subsection{Further directions}
One obvious direction for further research is to try to actually solve the variational problem in \cref{cor:main-graphs} in full generality. Even for the triangle-freeness problem in \cref{thm:triangles}, for which we now have a full solution, it would be great to have a simpler, more conceptual solution.

Another obvious direction is to bridge the gap between the bound in \cref{eq:subcritical-general} and the bounds in \cref{eq:supercritical-general-m,eq:supercritical-general-p}, for general $H$ (not just when $H$ is strictly 2-balanced). Our methods seem to be robust enough for this, but for general $H$ they naturally produce quite intimidating variational problems involving the interaction between many different edge-weightings, and we wonder whether this can be simplified.

As already mentioned in \cref{rem:high-girth}, another direction is to try to generalise \cref{thm:STS} to higher-girth Steiner triple systems. A first step would be to rigorously disprove the Glock--K\"uhn--Lo--Osthus conjecture, which we now believe to be false; we suspect this might be possible via a non-regular generalisation of recent work on ``conflict-free hypergraph matchings''~\cite{DP26,GJKKL24}, together with absorption techniques (see e.g.\ \cite{KSSS24,DP24}).

Finally, in the direction of \cref{subsubsec:latin}, the main remaining open problem is to prove a lower bound matching \cref{thm:latin} on the number of order-$n$ Latin squares with property $N_\infty$ (this would resolve a conjecture of Kelly~\cite[Conjecture~D(b)]{Kel+}). We wonder if this might be possible using some generalisation of the methods in the recent work of Postle--Lesgourgues mentioned earlier in this introduction.

\subsection{Organisation of the paper}In \cref{sec:overview} we start with a high-level overview of the ideas in this paper, focusing on our new methods to prove upper bounds on lower-tail large deviation probabilities.

Then, \cref{thm:main-upper,thm:main-m-upper} are proved in \cref{sec:books,sec:upper}. Specifically, in \cref{sec:books} we introduce the notion of a ``book hypergraph'', which will be used to organise certain statistics. The rest of the proofs of \cref{thm:main-upper,thm:main-m-upper} appear in  \cref{sec:upper}. We remark that, for expository purposes, we begin \cref{sec:upper} with a proof of the upper bound in \cref{thm:main-upper} in the special case $L=0$ (before giving a proof of the general-$L$ case). We believe this is the clearest illustration of our general method.

Next, in \cref{sec:lower} we prove \cref{thm:main-lower,thm:main-m-lower}, and in \cref{sec:reductions} we give the quick deductions of \cref{thm:main-graphs} and \cref{cor:main-graphs} from \cref{thm:main-upper,thm:main-m-upper,thm:main-lower,thm:main-m-lower}.

In \cref{sec:variational} we explain how to solve the variational problems in \cref{cor:main-graphs} in some special cases: namely, we prove \cref{prop:sidorenko} and \cref{thm:triangles,thm:H-early,thm:H-late}.

In \cref{sec:STS} we show how to adapt the ideas in \cref{sec:books,sec:upper} to the settings of Steiner triple systems and Latin squares, to prove \cref{thm:STS,thm:latin}.

Finally, we have three appendices: in \cref{app:KS} we give the details of a minor adaptation of some machinery of Kozma and Samotij, in \cref{app:numerical} we describe a certain ``certificate function'' used in the proof of \cref{thm:triangles}, and in \cref{app:semi-random} we explain how to adapt the analysis in \cite{GW} (on a certain semi-random independent set process) to provide enumeration estimates. 

\subsection{Notation}
We use standard asymptotic notation throughout, as follows. For functions $f=f(n)$ and $g=g(n)$, we write $f=O(g)$ to mean that there is a constant $C$ such that $|f|\le C|g|$, $f=\Omega(g)$ to mean that there is a constant $c>0$ such that $f(n)\ge c|g(n)|$ for sufficiently large $n$, and $f=o(g)$ to mean that $f/g\to0$ as $n\to\infty$. Subscripts on asymptotic notation indicate quantities that should be treated as constants.

We also use notation of the form $f=1\pm\varepsilon$ to mean
$1-\varepsilon\le f\le1+\varepsilon$. We also frequently (and slightly abusively) use notation like $S\setminus v$ instead of $S\setminus \{v\}$, to denote the effect of removing an element $v$ from a set $S$.

\subsection{Acknowledgments} 
First, we would like to acknowledge the independent work of Matt Jenssen, Will Perkins, Aditya Potukuchi and Michael Simkin, who obtained some related results from a totally different (statistical physics) point of view. Namely, their paper \cite{JPPSverynew} focuses on the probability that a binomial random subset of a hypergraph contains no hyperedges, as this has a natural interpretation in terms of the so-called hard-core model on the hypergraph. In particular, \cite{JPPSverynew} contains alternative proofs of the $L=0$ cases of \cref{thm:main-upper,thm:main-lower} (without the mild quantitative shortcoming described in \cref{rem:log-factors}), from which they are able to deduce the $L=0$ case of \cref{thm:main-graphs}(1) and (essentially) the $z=0$ case of \cref{cor:main-graphs}(1). Their paper also has results describing the structure of a typical sample from the zero-hyperedge event (cf.\ \cref{rem:structure,rem:structure-intro}). We have coordinated to announce our two papers simultaneously, to emphasise their independence. We would like to thank the authors of \cite{JPPSverynew} for collegial discussions about their work and its relationship to ours.

We would like to thank Lutz Warnke for useful discussions about his papers \cite{GW} (with He Guo) and \cite{JW16} (with Svante Janson). We would also like to thank Lutz for confirming that a certain estimate in \cite{GW} could be adapted in the way we needed (this appears as \cref{lem:app-Y-size}).

An initial version of the proofs of \cref{thm:main-upper,lem:STS-reduction-p} required certain machinery related to the so-called \emph{K\L R conjecture}. At the time, the necessary machinery was not available in the literature, but suitable theorems were being worked on by two independent groups of authors (one group consisted of Gady Kozma, Wojciech Samotij, Asaf Shapira and Adam Wagner, and the other group consisted of Peter Allen, Julia B\"ottcher, Joanna Lada and Domenico Mergoni; the latter has now appeared on the arXiv~\cite{ABLC}). We would like to thank Peter Allen and Wojciech Samotij for many very helpful conversations and clarifications about their work; although we eventually managed to remove the need for this machinery, what we had learned from these conversations still played an invaluable role in the final proofs.

We would also like to thank Robert Seiringer for general discussions about variational problems, and David Conlon for helpful discussions around Sidorenko's conjecture. In particular we would like to thank David for helping us simplify our original AI-generated proof of \cref{lem:STS-analytic} (described in the following subsection).

\subsection{Statement on AI use}\label{sec:AI}
This paper is entirely human-written, but various different AI tools were used in different ways during the preparation of this paper. Most notably, ChatGPT was very helpful for solving several of the analytic variational problems produced by our method. In detail:
\begin{itemize}
    \item The proof of \cref{thm:H-early} requires \cref{prop:calculus}, which is the verification of an explicit inequality via elementary calculus. We had originally done this in a more complicated way, but in the proofreading phase ChatGPT 6 Pro noticed that our proof could be simplified.
    \item The proof of \cref{thm:triangles} (solving the variational problem for triangle-freeness) was initially obtained in conversation with ChatGPT 5.6 Pro. Specifically, we suggested the relaxed variational problem in \cref{eq:cheap-variational-relaxation} (initially phrased in an unnecessarily complicated way, with a local minimality constraint we were handling with a Lagrange multiplier), and ChatGPT determined that this relaxation was too lossy to prove \cref{thm:triangles} for all $c$. We then suggested to consider the additional positive-semidefiniteness constraints in \cref{eq:variational-new-constraint} (i.e., we suggested to prove \cref{prop:numerical}), and with a lot of encouragement ChatGPT was able to autonomously discover appropriate certifying functions (using various heuristics and numerical optimisation tools) and prove \cref{prop:numerical}.
    \item This proof of \cref{prop:numerical} mostly concerns the verification of an inequality concerning an explicit 4-variable function. This verification involved a large quantity of case distinctions and interval arithmetic, and as non-experts in the field of validated numerics we were not able to extract any significant mathematical insight from it. So, in the final stages of preparing this paper, we decided that instead of describing the proof in detail, it would be better to include a Lean certificate of its correctness\footnote{We remark that there are software packages specifically designed for this type of computational verification of an inequality (such as IBEX~\cite{IBEX16}, and certain packages in HOL Light which were used to verify Hales' proof of the Kepler conjecture~\cite{SH13}). Lean is not particularly well-optimised for this kind of thing, but we decided to use Lean as our computations were not too demanding and Lean has recently become the standard for formal verification in mathematics.}.
    This Lean certificate was produced using Harmonic's \emph{Aristotle}~\cite{aristotle}, with a blueprint produced by ChatGPT 6 Pro. Along the way, ChatGPT managed to significantly simplify the certifying functions, and it noticed that our original statement of \cref{prop:numerical} (which initially involved terms corresponding to a local optimality constraint) was unnecessarily complicated. 
    \item An initial version of the proof of \cref{lem:STS-analytic} (solving the variational problem relevant for \cref{thm:latin,thm:STS}) was produced completely autonomously by ChatGPT 5.4 Pro. It was quite simple, but in our opinion the sequence of manipulations in this proof was rather unexpected, and it did not feel like something we were on the path to coming up with ourselves. We have tried our best to contextualise these manipulations in terms of previous literature, and in particular David Conlon showed us how to simplify the proof using the Sidorenko property of $K_{2,2,2}^{(3)}$. 
    \item In the course of working on this paper, we tested different versions of ChatGPT on various different variational problems that arise from our methods. In particular, this led to the observations in \cref{rem:high-girth,rem:edge-critical-necessary}.
\end{itemize}
Apart from this assistance with variational problems, all conceptual development, results and proofs are our own work. However, during the process of writing this paper, conversations with ChatGPT 5.5 Pro, ChatGPT 5.6 Pro and ChatGPT 6 Pro were helpful for catching minor details we had initially missed, support with writing proofs of routine lemmas, creating figures, and copy-editing.

\section{Overview of proof ideas}\label{sec:overview}
In this paper we present a general method to reduce lower-tail large deviation problems to variational problems, and additionally we demonstrate some techniques to actually solve these variational problems. 
For most of this overview section, we focus on the task of reducing to a variational problem, giving a very high-level overview of the context and motivation for our method. We will sweep a very large number of details under the rug, leaving more in-depth exposition for the later sections.

\subsection{Enumeration by sequential choices}\label{subsec:sequential}
As motivation, it is instructive to start with a graph that is \emph{not} covered by \cref{cor:main-graphs}: let $H$ be the path with two edges. Then, the problem of determining $\Pr[G\text{ is }H\text{-free}]$ for $G\sim \mb G(n,m)$ is equivalent to the problem of determining the number of matchings with $m$ edges on a set of $n$ vertices. One way to compute this number is to count the number of ways to choose such a matching edge-by-edge: after choosing $i<m$ edges of the matching, there are always exactly $\binom {n-2i}2$ ways to choose the next matching edge. Multiplying the number of choices at each step, and dividing by the number of orders in which we could have chosen a given matching, we see that the desired number of matchings is precisely \[\frac{\binom n2\binom {n-2}2\binom {n-4}2\dots \binom {n-2(m-1)}2}{m!}.\]

Unfortunately, this simple-minded approach only works in very special situations, because in more general contexts the number of choices at a given step depends on the choices we made at previous steps. For example, if we are interested in the number of ways to choose an $m$-edge triangle-free graph on a set of $n$ vertices, in a similar edge-by-edge fashion, then the number of ways to choose the third edge will depend on whether the first and second edge share a vertex or not.

Perhaps surprisingly, there are some contexts where, even though the number of choices at each step can depend on previous choices, we can use double-counting arguments to show that the number of choices at each step is likely to follow a particular trajectory. For example, we sketch an approach to estimate the number of order-$n$ Steiner triple systems, which appears in \cite{KSS23} but is strongly inspired by arguments in \cite{LL13,Kwa20,Rad97}. Note that order-$n$ Steiner triple systems are precisely $m$-edge $H$-free 3-uniform hypergraphs on a set of $n$ vertices, where $m=\binom n2/\binom 32$ and $H$ consists of two triples intersecting in two vertices. If we are interested in the number of ways to choose an order-$n$ Steiner triple system in a triple-by-triple fashion, there are certainly many different possibilities for the number of choices at the $i$-th step. However, if we fix any particular order-$n$ Steiner triple system $S$, randomly order its triples, and consider the graph $G_i$ consisting of the pairs of vertices that are not covered by any of the first $i$ triples, then it is not hard to see that $G_i$ is likely to have about $\binom n3(1-i/m)^3$ triangles (using a concentration inequality, and the fact that every pair of vertices is covered by exactly one triple of $S$). These triangles describe ways to add an additional triple to the first $i$ triples of $S$, without violating the $H$-free condition.
With this kind of reasoning, one can see that for almost all the orderings of order-$n$ Steiner triple systems, the number of valid choices for the $(i+1)$-th triple is at most about $\binom n3(1-i/m)^3$, and therefore the number of Steiner triple systems is at most about
\[\frac{\binom n3(1-0/m)^3\cdot \binom n3(1-1/m)^3\cdot \binom n3(1-2/m)^3\cdot \dots\cdot \binom n3(1-(m-1)/m)^3}{m!}.\]
This is only an upper bound, since it does not account for the possibility that there might be ways to add an additional triple to the first $i$ triples of a Steiner triple system $S$, in a way that cannot be further extended to any Steiner triple system (or can be extended in only a small number of ways). However, it does turn out that this upper bound is close to optimal, because there is a related \emph{random process} which provides a corresponding lower bound. Namely, consider the random process that adds triples one at a time, sampled uniformly from all the triples that do not conflict with previously chosen triples. Some now-standard analysis of this process~\cite{Kwa20,BFL10,BFL15} (using the so-called \emph{differential equations method}) shows that this process will typically follow the same trajectory as above, leading to about $\binom n3(1-i/m)^3$ choices for the $(i+1)$-th triple. (Actually, in this particular setting, the differential equations method breaks down near the end of the process when $i$ is very close to $m$, and to handle this one needs totally different tools, e.g.\ the machinery of Keevash~\cite{Kee18,Kee14}, to show that there is typically a completion to a Steiner triple system. This is not really relevant to the present paper and we will not discuss it further).

\subsection{Different profiles for sequential enumeration}The sequential enumeration techniques described in \cref{subsec:sequential} seem to only be suitable in very ``homogeneous'' situations. Namely, they are effective when we are interested in a family of sets $\mc F$ with the following property. If we write $\mc F_\ell$ for the collection of all size-$\ell$ subsets of all $S\in \mc F$, then for \emph{every} $S\in \mc F$, if we consider a random ordering of $S$ and write $S_i\in \mc F_i$ for the first $i$ elements of $S$, then the typical approximate number of extensions of $S_i$ to a set in $\mc F_{i+1}$ should only depend on $i$ (not on the choice of $S$).

This property fails dramatically for all the problems highlighted in \cref{sec:intro}. For example, consider the problem of counting triangle-free graphs with $m$ edges. For each such graph $G$, let $G_i$ be the first $i$ edges of a random ordering of $G$. For each $i$, we are interested in the (typical approximate value of) the number $A(G_i)$ of non-edges $xy\notin G_i$ for which there is no path $xzy$ in $G_i$: this is precisely the number of ways in which we can add an edge to $G_i$ without creating a triangle. Unfortunately, the typical value of $A(G_i)$ can vary significantly for different $G$, as follows. We parameterise $i=\alpha n^{3/2}$.
\begin{itemize}
\item The so-called \emph{triangle-free process} (see e.g.\ \cite{Boh09,BK21,FGM20}) produces graphs which ``resemble a random graph but have no triangles''. In particular, if one takes an $m$-edge graph $G$ obtained from the triangle-free process, then $G_i$ resembles a random graph with edge probability $i/\binom n2$, and one can show that typically
\[A(G_i)\approx \binom n2\Big(1-\big(i/{\textstyle{\binom n2}}\big)^2\Big)^{n}\approx \binom n2\exp(-4\alpha^2)\] (where we are using ``$\approx$'' in an informal sense).
    \item Fix a partition of the vertex set into two parts of sizes $\lfloor n/2\rfloor$ and $\lceil n/2\rceil$, and suppose $G$ is a uniformly random $m$-edge bipartite graph between these parts. Then, $G_i$ is a random graph in which the edges between the two parts are present with probability about $i/(n/2)^2$ and one can show that typically 
    \[A(G_i)\approx \Big(\frac n2\Big)^2+2\binom {n/2}2\Big(1-\big(i/(n/2)^2\big)^2\Big)^{n/2}\approx \binom n2\frac{1+\exp(-8\alpha^2)}2.\]
    \item Suppose $G$ is a complete bipartite graph with parts of size $\sqrt m$ (assuming $m$ is a perfect square). Then, almost all edges can be added to $G$ without creating a triangle, so $A(G_i)\approx \binom n2$.
\end{itemize}
The key conceptual idea for how to handle this diversity is simply to \emph{break down the set of all $m$-edge triangle-free graphs $G$ according to the typical statistics of $A(G_1),\dots,A(G_m)$}.
For example, if we were able to partition our graphs into $\exp(o(m))$ classes, such that within each class the typical statistics of $A(G_1),\dots,A(G_m)$ are roughly the same, then we could use sequential enumeration techniques to estimate the size of each class. Up to lower-order terms, the number of $m$-edge triangle-free graphs would then be the size of the largest class, so we would have succeeded in describing the number of such graphs in terms of a variational problem (over classes).

At the current level of abstraction, this is not a totally new idea. For example, similar ideas appear in the work of Simkin~\cite{Sim23} on the number of ``$n$-queen configurations''. The challenge is to find ``the right way to break things down''.

\subsection{Breaking down the lower-tail event, in the supercritical regime}\label{subsec:supercritical-intro}
The kind of problem we face here is highly reminiscent of the kind of problem one faces in the \emph{supercritical regime}, where tools from extremal graph theory come into play. We take a moment to recall some of the ideas in this direction, specialising to the setting of large deviations for the number of copies of $H$ in a random graph.

First, suppose we are interested in graphs with say $n^2/10$ edges (this is the ``dense'' case, deep into the supercritical regime). Famously, \emph{Szemer\'edi's regularity lemma} says that every dense graph $G$ can be partitioned into a very small number of parts (of the same size), such that edges between pairs of parts behave in a ``random-like'' way. In particular, the number of copies of $H$ in $G$ can be approximated using the densities of the edges between the parts.

We can use this to partition the space of all $m$-edge graphs into a small number of classes, in such a way that the graphs in each class have roughly the same number of copies of $H$: two graphs are in the same class if their between-part densities are similar. Then, it is a very straightforward matter to estimate the number of graphs in each class (we do not need any clever double-counting techniques, we simply multiply the numbers of ways to choose appropriate sets of edges between the parts). This is (a caricature of) the approach famously used by Chatterjee and Varadhan~\cite{CV11} to describe large deviations of dense random graphs in terms of variational problems over \emph{graphons} (in the spirit of \cref{cor:main-graphs}, a \emph{graphon} is a continuum object that can be used to describe the distribution of edge-densities between different parts of a graph). However, this kind of regularity-method approach to enumeration goes back much further, at least to Erd\H os, Frankl and R\"odl~\cite{EFR86}.

In sparser contexts, Szemer\'edi's regularity lemma (or rather, its associated \emph{graph counting lemma}) does not apply, but many other tools have been developed along these lines (see e.g.\ \cite{Con14} for a survey on the general topic of extremal graph theory in sparse random settings). For example, in his proof of \eqref{eq:supercritical-triangles}, \L uczak~\cite{Luc00} proved the triangle case of what is now usually known as the \emph{K\L R conjecture}, which is a sparse version of the graph counting lemma (with many caveats). The general-$H$ version of \eqref{eq:supercritical-triangles} was proved by Saxton--Thomason~\cite{ST15} and Balogh--Morris--Samotij~\cite{BJM15}, using the celebrated \emph{hypergraph container method} (which gives a totally different way to break down the set of $H$-free graphs: every $H$-free graph is a subgraph of one of a relatively small number of ``containers'' which themselves have few copies of $H$). 

Most recently, Kozma and Samotij~\cite{KS23} developed a beautiful general framework for supercritical lower-tail large deviations, which does \emph{not} explicitly partition graphs into classes, but rather works directly on the level of probability distributions. Namely, Kozma and Samotij consider a random graph \emph{conditioned on an upper bound on the number of copies of $H$}. Using an entropic energy-increment argument reminiscent of the proof of Szemer\'edi's regularity lemma, they show that ``from the point of view of copies of $H$'', this conditional distribution can be approximated by a mixture of a small number of inhomogeneous random graphs (in which the edges are present independently, albeit possibly with different probabilities). This very transparently leads to the variational problem in \cref{eq:supercritical-general-p}.

All of these methods fundamentally only work in the supercritical regime, where $p\gg n^{-1/d_2(H)}$ or $m\gg n^{2-1/d_2(H)}$. Indeed, in one way or another they all depend on the property that the number of copies of $H$ in a subgraph of a random graph can be approximated in terms of the edge distribution of that subgraph, and this is simply not true outside of the supercritical regime.

\subsection{Breaking down the lower-tail event, in the critical regime}\label{subsec:critical-threats}

Returning back to the critical regime (and to the special case of triangle-free graphs), our goal was to find a way to break down the space of $m$-edge triangle-free graphs according to the typical statistics of $A(G_1),\dots,A(G_m)$.
Recall that $A(G_i)$ counts pairs of vertices for which there is no path of length 2 between them (these length-2 paths are ``triangle-threats'' that would cause a triangle if an edge were added between their endpoints). The key observation is that, using the inclusion-exclusion principle, in order to describe the typical values of $A(G_1),\dots,A(G_m)$, it suffices to describe the number of $k$-tuples of length-2 paths with the same endpoints (i.e., the number of ``$k$-fold triangle-threats''). In other words, writing $H^{(k)}$ for the graph consisting of $k$ length-2 paths with the same endpoints, we need to know the number of copies of $H^{(k)}$, for each\footnote{in practice, for a desired level of accuracy in our final estimate, there is some number $A$ for which it suffices to understand all $k\le A$.} $k$. Crucially, these graphs $H^{(k)}$ have a \emph{different supercritical regime} than the triangle $K_3$: it is not hard to see that $d_2(H^{(k)})=2-1/k<2=d_2(K_3)$. So, if we are in the critical regime for $H$, we are actually in the supercritical regime for the graphs $H^{(k)}$. We can therefore use the tools described in \cref{subsec:supercritical-intro} to break down the space of graphs according to their number of copies of each $H^{(k)}$ (and therefore according to the typical values of $A(G_1),\dots,A(G_m)$, as desired). This same kind of phenomenon occurs when one is counting $H$-free graphs for any strictly 2-balanced graph $H$ (in general, $H^{(k)}$ should be replaced with a family of graphs that encodes all the ways to combine $k$ different ``$H$-threats''), and variations of this phenomenon occur in the general settings of \cref{thm:main-m-upper,thm:main-upper} and in the settings of \cref{thm:STS,thm:latin}.

There are several different ways one can actually put the pieces together, using the various tools described in \cref{subsec:supercritical-intro}. In a previous (working) version of this paper, we had executed this plan using a very general \emph{transference principle} due to Allen--B\"ottcher--Lada--Mergoni~\cite{ABLC}, which serves a similar role to the K\L R conjecture, allowing one to use regularity-method tools in sparse settings. This involved many technicalities. However, in this version of the paper, we found a way to replace our use of the transference principle with the framework of Kozma and Samotij mentioned in \cref{subsec:supercritical-intro}, which provides much cleaner proofs. Since this framework does not provide a combinatorial partition (but rather a probabilistic decomposition of the lower-tail event), we could not directly use the sequential enumeration method described in \cref{subsec:sequential} as written; instead we came up with an entropic variant of this method.

Ideas of the type described so far are suitable for proving upper bounds on lower-tail probabilities, in terms of variational problems, in very general contexts. In particular, these ideas are essential for the proofs of \cref{thm:main-upper,thm:main-m-upper,thm:STS,thm:latin}. Of course, each of these settings has its own subtleties, and  additional ideas are required (which we do not go into here). To see an illustration of the basic idea, in the simplest form we were able to find, see \cref{subsec:independent-example} (where we prove the $L=0$ case of \cref{thm:main-upper}).

\subsection{Lower bounds on lower tails}
The ideas described in the last few sections are only used to prove \emph{upper bounds} on the lower tail. As discussed in \cref{subsec:sequential}, lower bounds can be proved with appropriate random processes that mirror the statistics of the upper bounds.

In particular, the natural way to prove the lower bounds in \cref{thm:main-m-lower,thm:main-lower} would be to consider a biased random process which builds a random set of vertices in a vertex-by-vertex manner, at each step choosing the next vertex with a carefully calibrated probability distribution over the vertices not yet chosen (where the probability of choosing a vertex depends on the number of hyperedges it would introduce to the current set, and on the optimising vertex-weighting $\vec q$).

In the case where $L=0$ (i.e., we are looking for an independent set in $\mc H$), and our hypergraph $\mc H$ happens to be vertex-regular, and the vertex-weighting we are interested in is uniform (i.e., the entries of $\vec q$ are all the same), then the relevant process is the so-called \emph{random greedy independent set process}, which has now been analysed in nearly full generality by Bennett and Bohman~\cite{BB16} and Bennett and Priestley~\cite{BP25} (see also the analysis of the \emph{$H$-free process} by Bohman and Keevash~\cite{BK10}, which is the special case of this process relevant to \cref{thm:main-graphs}). It would be a significant undertaking to generalise this work to the biased setting we require.

Fortunately, this task can be avoided, by taking advantage of some recent work of Guo and Warnke~\cite{GW}. Specifically, they studied a variant of the random greedy independent set process which produces the necessary statistics even when the hypergraph $\mc H$ is far from being vertex-regular. By combining their analysis with various types of subsampling arguments (to account for non-uniform $\vec q$ and the fact that we may be interested in general lower tails, not just independent sets), we can deduce the desired lower bounds on the lower tail without needing to analyse any new random processes.

\subsection{Solving the variational problems}
\cref{thm:triangles,prop:sidorenko} require us to solve certain special cases of the variational problems in \cref{cor:main-graphs} (which are continuum versions of variational problems in \cref{thm:main-graphs}), and \cref{thm:latin,thm:STS} require us to solve certain slightly different variational problems. In this subsection we discuss the features of these variational problems, natural approaches which work in certain regimes, and some trickier approaches which work outside of them.

For simplicity, for now let us just consider the $L=0$ case of the variational problem in \cref{thm:main-upper,thm:main-lower}. We consider an $r$-uniform hypergraph $\mc H$ satisfying a certain ``degree-boundedness'' condition, and parameters $\gamma>0$ and $p\in [0,1]$, and we are interested in the minimum value of
\[F(\vec q):=\on{D}_{\mr{KL}}(\vec{q}\,\|\,p)+\gamma\mc H(\vec q),\]
over all vertex-weightings $\vec q\in [0,p]^V$.

\subsubsection{The Sidorenko property}
Unsurprisingly, this minimisation problem is easiest to solve if $\on{D}_{\mr{KL}}(\vec{q}\,\|\,p)$ and $\mc H(\vec q)$ are minimised by the same weightings $\vec q$. This is basically the case for the variational problem relevant for \cref{prop:sidorenko}. Indeed, suppose $H$ has the Sidorenko property and $\mc H$ is the hypergraph of copies of $H$ in the complete graph $K_n$. Then, it is easy to see that among weightings with a given sum, both $\on{D}_{\mr{KL}}(\vec{q}\,\|\,p)$ and $\mc H(\vec q)$ are minimised\footnote{Strictly speaking $\mc H(\vec q)$ is not exactly minimised by the constant-valued weighting, due to some subtleties related to the distinction between embedded and homomorphic copies of $H$. This issue disappears when we consider the continuum problems in \cref{cor:main-graphs}.} by the constant-valued weighting with that sum, and as a consequence $F(\vec q)$ is minimised by a constant weighting.

In general, when $\on{D}_{\mr{KL}}(\vec{q}\,\|\,p)$ and $\mc H(\vec q)$ are optimised by different types of weightings $\vec q$, it can become quite difficult to understand how these two terms interact with each other. However, for the variational problems in the proofs of \cref{thm:STS,thm:latin} (in which $\mc H$ describes Pasch configurations in a complete or complete tripartite 3-uniform hypergraph), there is a very slick way to sidestep this difficulty: namely, while the Pasch configuration does not satisfy (a hypergraph generalisation of) the Sidorenko property, it turns out that the complete 3-uniform 3-partite hypergraph $K_{2,2,2}^{(3)}$ does satisfy the Sidorenko property and can be partitioned into two edge-disjoint Pasch configurations. One can use this fact with the Cauchy--Schwarz inequality to obtain a certain Sidorenko-type bound on $\mc H(\vec q)$ (this is reminiscent of ``decomposition and reflection'' applications of the Cauchy--Schwarz inequality in \cite{CL17}), which turns out to (barely!) be enough to solve the relevant variational problem.

\subsubsection{(Local) convexity}

The types of techniques described above, based on Sidorenko-type inequalities, can only be applied in very special circumstances. On the other hand, there are some general techniques that work for any vertex-regular $\mc H$ (i.e., for any $\mc H$ for which there is $D$ such that $\deg_{\mc H}(v)=D$ for all $v\in V$).

First, under this vertex-regularity assumption it is not hard to show that $F(\vec q)$ is minimised by a constant-valued weighting when $\gamma>0$ is sufficiently small. One way to see this is as follows. First, using the vertex-regularity of $\mc H$, one can check that there is a unique constant-valued weighting $\vec q$ which is a stationary point of $F$. Then, observe that $\on{D}_{\mr{KL}}(\vec{q}\,\|\,p)$ is a strictly convex function of $\vec q$. If $\gamma$ is sufficiently small then adding $\gamma\mc  H(\vec q)$ has a negligible impact, so $F$ is also strictly convex, and is therefore minimised by its constant-valued stationary point.

Actually, even when $\gamma$ is not small enough for $F$ to be globally convex, if $F$ is minimised by a constant weighting we might hope to establish some sort of crude bound on $F$ away from the minimiser, and make do with local convexity on a neighbourhood of the minimiser. This does seem to be possible in certain ranges of $\gamma$ (for certain $\mc H$), but the necessary casework seems to very quickly become more and more complicated as the value of $\gamma$ increases.

\begin{remark}
    Although it is not very relevant for the present paper, we remark that a different approach that works for small $\gamma$ is implicit in the work of Jenssen--Perkins--Potukuchi--Simkin~\cite{JPPSnew,JPPSold}. Indeed, while they do not explicitly solve any variational problems, they obtain asymptotics for the lower-tail large deviation rate function under certain conditions, which means their methods can be interpreted as implicitly yielding solutions to our variational problems. Their approach, translated to our setting, is as follows. One can show that any minimising weighting $\vec q$ must be a stationary point of $F$. This yields an equation that a minimising $\vec q$ must satisfy, which can be interpreted as a fixed point equation for a certain operator (in statistical physics language, this operator can be interpreted as a \emph{belief propagation} operator for a certain ``hard-core model'' associated with $\mc H$). If $\gamma$ is sufficiently small (namely, if $\gamma$ is below a so-called \emph{Gibbs uniqueness threshold}), one can show that this operator is a contraction in a suitable norm, so there is a unique fixed point.
\end{remark}

On the other hand, when $\gamma$ is not small, the minimising weighting may look very different from the constant graphon. For example,  if $\mc H$ is the hypergraph of triangles in the complete graph $K_n$, it turns out that the minimising weightings have a 2-block structure. Specifically, we split the vertices into two equal parts, put a constant weight on all the edges between the parts, and a different constant weight on all the edges within the parts. In this case, we can still use (local) convexity arguments to show that among weightings which are ``close to having 2-block structure'', the minimising weighting has true 2-block structure, and for large $\gamma$ one can use techniques from extremal graph theory to rule out minimisers which are not close to having 2-block structure. We use this kind of argument to prove \cref{thm:H-late}.

\subsubsection{Low-dimensional relaxation and positive semidefiniteness}The message of the previous section is that in principle one can solve these types of minimisation problems by exploiting convexity in the neighbourhood of the minimisers, and using crude bounds everywhere else. Unfortunately, the ``crude bounds everywhere else'' part seems to be highly nontrivial, largely because the space of weightings $\vec q$ is so high-dimensional.

So, it is of value to find lower-dimensional relaxations of our minimisation problem. One very simple such relaxation is as follows. Assuming that every vertex of $\mc H$ has degree $D$, 
and writing $\vec q=(q_v)_{v\in V}$ and $i_p(q)=q\log(q/p)+(1-q)\log((1-q)/(1-p))$, we have
\[F(\vec q)=\sum_{e\in \mc H}\bigg(\frac1D\sum_{v\in e}i_p(q_v)+\gamma\prod_{v\in e}q_v\bigg).\]
So, recalling that $\mc H$ is an $N$-vertex $r$-uniform hypergraph, if we were able to prove a lower bound on $(1/D)\sum_{v=1}^ri_p(q_v)+\gamma\prod_{v=1}^rq_v$ that holds for \emph{all} $(q_v)_{v=1}^r\in [0,1]^{r}$, this would imply the same lower bound on $F(\vec q)/|\mc H|$. That is to say, we can prove bounds on our $N$-dimensional variational problem in terms of a much more tractable $r$-dimensional variational problem (that does not actually depend on $\mc H$).

This idea is surprisingly powerful, and we use it to prove \cref{thm:H-early} (solving the $L=0$ cases of the variational problems in \cref{thm:main-graphs}(1) for small $c$).
However, for some $\gamma$ this relaxation is too lossy, and in particular it is not suitable to prove \cref{thm:triangles}. So, in our proof of \cref{thm:triangles} we need one additional idea, inspired by the method of semidefinite flag algebras from extremal combinatorics (see e.g.\ \cite{flags}). Namely, in the variational problem for the triangle-freeness probability, the underlying graph structure allows us to impose certain positive-semidefiniteness constraints on the weights $q_v$, which constrain the problem just enough to prove \cref{thm:triangles}. The actual proof is heavily computer-assisted: it involves local convexity arguments in a neighbourhood of the minimisers, and crude bounds on each of a very large number of small regions of parameter space.

\section{Book hypergraphs}\label{sec:books}
For the proofs of \cref{thm:main-upper,thm:main-m-upper}, it will be important to consider certain ``book hypergraphs'' associated with our hypergraph $\mc H$, which record the combinatorics of tuples of hyperedges of $\mc H$ containing a common vertex (this is the necessary generalisation of the ``$k$-fold triangle-threats'' discussed in \cref{subsec:critical-threats}).

\begin{definition}\label{def:books}
Let $\mc H$ be a hypergraph on the vertex set $V$. An $a$-\emph{book} in $\mc H$ is a \emph{spine} vertex $v$ and a sequence of distinct hyperedges $(e_1,\dots,e_a)$ of $\mc H$ (its \emph{pages}) which all contain $v$. The \emph{page vertices} of the book are the vertices in $(e_1\cup \dots\cup e_a)\setminus v$ (note that a page vertex can be in more than one page). Then,
\begin{itemize}
    \item For $a,k\in \mb N$ and $v\in V$, let $\on{bk}^a_k(\mc H;v)$ be the $k$-uniform multihypergraph defined as follows. For each $a$-book which has spine $v$ and exactly $k$ page vertices, put this set of page vertices as a hyperedge of $\on{bk}^a_k(\mc H;v)$. (Note that, if $\mc H$ is $r$-uniform, such an $a$-book can only exist if $k\le a(r-1)$).
    \item We denote by $\on{bk}^a_k(\mc H)$ the $k$-uniform multihypergraph obtained as the multiset union of $\on{bk}^a_k(\mc H;v)$ for each $v\in V$. 
\end{itemize}
\end{definition}

If $\mc H$ satisfies the degree-boundedness condition in \cref{def:bounded}, then we obtain corresponding bounds on the degrees of its book hypergraphs, as follows.

\begin{lemma}\label{lem:book-deg}
    Assume that $\mc H$ is $(\lambda, C;p)$-degree-bounded, for $\lambda\le 1$, and consider any $A\in \mb N$. Then for all $a\le A$,
    \begin{align}
        \Delta_s(\on{bk}^a_k(\mc H))&\le O_{A,C,r}(1)\cdot (\lambda^{1/(Ar)} p)^{s-1}\cdot 
    p^{-k}\text{ for }s\in \{1,\dots,k\},\label{eq:book-degrees-goal-3}\\
        |\on{bk}^a_k(\mc H;v)|&\le O_{A,C,r}(1)\cdot p^{-k}\text{ for all }v\in V.\label{eq:book-degrees-goal-1}
    \end{align}
    Moreover, for $k<a(r-1)$ we have
    \begin{equation}
    |\on{bk}^a_k(\mc H;v)|\le O_{A,C,r}(1)\cdot \lambda\cdot p^{-k}\text{ for all }v\in V.\label{eq:book-degrees-goal-2}
    \end{equation}
\end{lemma}
\begin{proof}
We may assume $C\ge 1$. We trivially have $\Delta_{r}(\mathcal{H})\le 1$, so $(\lambda,C;p)$-degree-boundedness implies that
\begin{align}
\Delta_{s}(\mathcal{H}) &\le  Cp^{s-r}\text{ for all }s\in\{1,\dots,r\},\label{eq:book-degrees-1}\\
\Delta_{s}(\mathcal{H}) & \le C\lambda p^{s-r}\text{ for }s\in\{2,\dots,r-1\},\label{eq:book-degrees-2}\\
\Gamma(\mathcal{H}) & \le C\lambda p^{1-r}.\label{eq:book-degrees-3}
\end{align}
Fix $a\le A$ and $k\in\{1,\dots,a(r-1)\}$. Say an 
``$(a;k)$-book'' is an $a$-book with exactly $k$ page vertices. 
We represent $(a;k)$-books in the form $(v,e_1,\dots,e_a)$, where $v$ is the spine and $e_1,\dots,e_a$ are the pages.

First, for a vertex $v\in V$, let $\mathcal{B}(v)$ be the set
of $(a;k)$-books with spine $v$ (so $|\on{bk}^a_k(\mc H;v)|=|\mathcal{B}(v)|$). For each $(v,e_{1},\dots,e_{a})\in\mathcal{B}(v)$,
we define ``overlap statistics'' $\ell_{1},\dots,\ell_{a}$, as
follows. Let $\ell_{1}=1$, and for $i\in\{2,\dots,a\}$ let $\ell_{i}=|(e_{1}\cup\dots\cup e_{i-1})\cap e_{i}|\ge1$
be the number of vertices of $e_{i}$ which already appeared in previous
pages. So, $\ell_{1}+\dots+\ell_{a}=ar-k$. We use these statistics
to stratify $\mathcal{B}(v)$: for $\ell_{1},\dots,\ell_{a}$
satisfying $\ell_{1}+\dots+\ell_{a}=ar-k$, let $\mathcal{B}(v;\ell_{1},\dots,\ell_{a})\subseteq\mathcal{B}(v)$
be the set of $(a;k)$-books with spine $v$ and these particular
overlap statistics $\ell_{1},\dots,\ell_{a}$. Then, using \cref{eq:book-degrees-1} we have
\begin{align}
|\mathcal{B}(v;\ell_{1},\dots,\ell_{a})| & \le\Delta_{1}(\mathcal{H})\cdot\binom{r}{\ell_{2}}\Delta_{\ell_{2}}(\mathcal{H})\cdot\binom{2r-\ell_{2}}{\ell_{3}}\Delta_{\ell_{3}}(\mathcal{H})\cdot\dots\cdot\binom{ar-\ell_{1}-\dots-\ell_{a-1}}{\ell_{a}}\Delta_{\ell_{a}}(\mathcal{H})\notag\\
 & \le O_{r,A}\big(\Delta_{\ell_{1}}(\mathcal{H})\cdot\dots\cdot\Delta_{\ell_{a}}(\mathcal{H})\big)=O_{A,C,r}(p^{\ell_{1}-r}\cdot\dots\cdot p^{\ell_{a}-r})=O_{A,C,r}(p^{-k}).\label{eq:book-calculation-for-later}
\end{align}
Since there are at most $r^{a}\le O_{A,r}(1)$ choices for $\ell_{1},\dots,\ell_{a}$,
we deduce \cref{eq:book-degrees-goal-1}.

We can improve this bound if $k<a(r-1)$. Indeed, in this case, every $(a;k)$-book has a pair of pages that intersect in more than one vertex, so we can reorder its pages to make $\ell_2\in \{2,\dots,r-1\}$ (note that the pages of an $a$-book must be distinct, so we cannot have $\ell_2=r$). But if $\ell_2\in \{2,\dots,r-1\}$, we can use  \cref{eq:book-degrees-2} in place of \cref{eq:book-degrees-1} to bound $\Delta_{\ell_2}(\mc H)$, which gives
\[
|\mathcal{B}(v;\ell_{1},\dots,\ell_{a})| \le O_{A,C,r}(\lambda p^{-k}).\]
Since there are only $a!=O_A(1)$ ways to reorder the hyperedges of an $(a;k)$-book, we deduce \cref{eq:book-degrees-goal-2}.

Next, for fixed vertices $v_{1},\dots,v_{s}$, let $\mathcal{D}(v_{1},\dots,v_{s})$
be the set of $(a;k)$-books $(v,e_1,\dots,e_a)$ containing $v_{1},\dots,v_{s}$
among their page vertices, such that $e_{1}$ contains at least one
of these fixed vertices. So, $\on{deg}_{\on{bk}^a_k(\mc H)}(v_1,\dots,v_s)\le a|\mathcal{D}(v_{1},\dots,v_{s})|$.

For each $(v,e_{1},\dots,e_{a})\in\mathcal{D}(v_1,\dots,v_s)$,
we again define overlap statistics, though in a slightly different way: let $\ell_{1}=|\{v_{1},\dots,v_{s}\}\cap e_{1}\setminus v|\in \{1,\dots,r-1\}$
be the number of fixed vertices among the page vertices of $e_{1}$,
and for $i\in\{2,\dots,a\}$, let $\ell_{i}=|(\{v_{1},\dots,v_{s}\}\cup e_{1}\cup\dots\cup e_{i-1})\cap e_{i}|\ge1$
be the number of ``previous'' vertices among the vertices of $e_{i}$.
Then $\ell_{1}+\dots+\ell_{a}=s+ar-1-k$. For $\ell_{1},\dots,\ell_{a}$
satisfying $\ell_{1}+\dots+\ell_{a}=s+ar-1-k$, let $\mathcal{D}(v_{1},\dots,v_{s};\ell_{1},\dots,\ell_{a})\subseteq\mathcal{D}(v_{1},\dots,v_{s})$
be the set of $(a;k)$-books in $\mathcal{D}(v_{1},\dots,v_{s})$
with these particular overlap statistics $\ell_{1},\dots,\ell_{a}$. 

In general, using \cref{eq:book-degrees-1}, we have
\[
|\mathcal{D}(v_{1},\dots,v_{s};\ell_{1},\dots,\ell_{a})|\le O_{A,r}\big(\Delta_{\ell_{1}}(\mathcal{H})\cdot\dots\cdot\Delta_{\ell_{a}}(\mathcal{H})\big)\le O_{A,C,r}(p^{\ell_{1}-r}\cdot\dots\cdot p^{\ell_{a}-r})=O_{A,C,r}(p^{s-k-1}),
\]
which already yields the desired bound $\Delta_{s}(\on{bk}^a_k(\mc H))\le O_{A,C,r}(1)\cdot(\lambda^{1/(Ar)}p)^{s-1}\cdot p^{-k}$ in \cref{eq:book-degrees-goal-3},
in the case $s=1$.

It remains to establish \cref{eq:book-degrees-goal-3} in the case $s\ge 2$. For this, note that if any $\ell_i$ lies in the range $\{2,\dots,r-1\}$ we have the improved bound
\[|\mathcal{D}(v_{1},\dots,v_{s};\ell_{1},\dots,\ell_{a})|\le O_{A,C,r}(\lambda p^{s-k-1})
\]
(using \cref{eq:book-degrees-2} in place of \cref{eq:book-degrees-1} to bound $\Delta_{\ell_i}(\mc H)$). Also, let $\mc D^*(v_1,\dots,v_s)\subseteq\mc D(v_1,\dots,v_s)$ be the set of $(a;k)$-books $(v,e_1,\dots,e_a)$ for which $e_1,e_2$ both contain exactly one of the fixed vertices $v_1,\dots,v_s$ (this implies $\ell_1=1$), and $\ell_2=r$. We have
\[|\mc D^*(v_1,\dots,v_s)|
\le O_{A,C,r}(\lambda p^{s-k-1})
\]
(stratifying according to $\ell_3,\dots,\ell_a$, and using the bound on $\Gamma(\mc H)$ in \cref{eq:book-degrees-3} to control the number of choices for $e_1,e_2$).

So, to prove the desired inequality $\Delta_{s}(\on{bk}_k^a(\mc H))\le O_{A,C,r}(\lambda p^{s-k-1})\le O_{A,C,r}(1)\cdot(\lambda^{1/(Ar)}p)^{s-1}\cdot p^{-k}$, it suffices to show that for every $(a;k)$-book $(v,e_1,\dots,e_a)\in \mathcal{D}(v_{1},\dots,v_{s})$, there is a way to reorder the pages of our book to make some $\ell_i\in \{2,\dots,r-1\}$, or to make it lie in $\mathcal{D}^*(v_{1},\dots,v_{s})$.

To this end, first note that if there were a page containing more than one of the fixed vertices $v_1,\dots,v_s$, we could put it first to make $\ell_1\in \{2,\dots,r-1\}$. So, suppose that every page contains at most one fixed vertex, and reorder the pages so that $e_1$ and $e_2$ contain different fixed vertices. The only way we could fail to have $\ell_{2}\in \{2,\dots,r-1\}$ is if $\ell_2=r$, in which case our reordered $(a;k)$-book lies in $\mathcal{D}^*(v_{1},\dots,v_{s})$.
\end{proof}

\section{Upper bounds on lower tails}\label{sec:upper}
In this section we prove \cref{thm:main-upper,thm:main-m-upper}.
For illustrative purposes, before proving these theorems in full generality, we first present a simpler proof of the $L=0$ case of \cref{thm:main-upper}. But before doing any of this, we start with some preliminaries on entropy and we state a product decomposition theorem of Kozma and Samotij.

\subsection{Preliminaries}
Throughout this section we will need the notion of \emph{entropy} and a few basic facts about it. We provide some brief reminders in this subsection; see e.g.\ \cite{CT91} for more context.

For a (discrete) random element $X$ with support $\mc S_X$, recall that the (base-$e$) \emph{entropy} is defined as
\[\on H[X]=-\sum_{x\in \mc S_X}\Pr[X=x]\log(\Pr[X=x]).\]
For a second random element $Y$ on the same probability space, with support $\mc S_Y$, recall that the \emph{conditional entropy} of $X$ given $Y$ is
\[\on H[X|Y]=\sum_{y\in \mc S_Y}\Pr[Y=y]\,\on H[X\,|\,Y=y].\]
\begin{fact}\label{fact:entropy}
    The following hold for any random sequence $(X_1,\dots,X_L)$.
    \begin{itemize}
        \item (Chain rule) $\on{H}[X_1,\dots,X_L]=\sum_{\ell=1}^L\on H[X_\ell\,|\,X_1,\dots,X_{\ell-1}]$
        \item\vspace{5pt} (Subadditivity) $\on{H}[X_1,\dots,X_L]\le \sum_{\ell=1}^L\on H[X_\ell]$.
    \end{itemize}
\end{fact}
The entropy of a $\on{Bernoulli}(p)$ random variable is given by the \emph{binary entropy function}
\[h(p) = -p\log p - (1-p)\log (1-p),
\]
which is a concave function of $p$.

We will also need (a consequence of) Harris' inequality.
\begin{lemma}\label{lem:harris}
Fix a ground set $V$ and let $\mathcal D$ be a nontrivial ``decreasing property'': a nonempty family of subsets of $V$, closed under taking subsets. Let $Y$ be a sample from $\on{Bin}(V,p)$, conditioned on the event $\mc D$. Then for every $W\subseteq V$, each possible outcome $U$ of $Y\cap W$, and each $S\subseteq V\setminus W$, we have
\[\Pr[S\subseteq Y\,|\,Y\cap W=U]\le p^{|S|}.\]
In particular, for each $v\in V\setminus W$ we have 
\[\Pr[v\in Y\,|\,Y\cap W=U]\le p.\]
\end{lemma}
\begin{proof}
    Let $\mc D_U$ be the collection of subsets $Z\subseteq V\setminus W$ such that $Z\cup U\in \mc D$, and let $R\sim \on{Bin}(V\setminus W,p)$. This is a decreasing property on $V\setminus W$. Then
\[\Pr[S\subseteq Y\,|\,Y\cap W=U]=\Pr[S\subseteq R\,|\,R\in \mc D_U]\le \Pr[S\subseteq R]=p^{|S|},\]
using Harris' inequality (see e.g.\ \cite[Theorem~4.2]{Gri18}).
\end{proof}

\subsection{Product decomposition}
We will use the following slight adaptation of a result of Kozma and Samotij. It says that if we condition $\on{Bin}(V,p)$ on a decreasing property, then from the point of view of any suitably degree-bounded hypergraph, we can approximate the resulting distribution by a mixture of a small number of inhomogeneous product measures. We will need to apply this theorem to the book hypergraphs defined in \cref{def:books}.

\begin{theorem}\label{thm:KS}
    Consider $r_0,C_0, \alpha, \xi > 0$. There exists $\theta>0$ and $N_0\in \mb N$ such that the following holds. 
    Let $V$ be a ground set of size $N\ge N_0$, let $p\in(0,1/2)$ and let $\mc D$ be a nontrivial decreasing property. Let $Y$ be a sample from $\on{Bin}(V,p)$, conditioned on the event $\mc D$. Then there exists $W \subseteq V$ with $|W| \le \alpha N$ such that, for any $k$-uniform hypergraph $\mc G$ equipped with nonnegative weights on its hyperedges, with $k\le r_0$, satisfying\footnote{Here $\Delta_s(\mc G)$ should be interpreted in a weighted sense (we define $\deg_\mc G(S)$ to be the sum of $\on{wt}_\mc G(e)$ over all $e\supseteq S$).}
    \begin{equation}
        \Delta_s(\mc G) \le C_0 (\theta p)^{s-1} p^{-k}\text{ for all }s\in \{1,\dots,k\},\label{eq:KS-degree-condition}
    \end{equation}
    we have    
    \[
        \mb E \Bigg [\sum_{e\in\mc G[V\setminus W]} \on{wt}_{\mc G}(e) \bigg|\mb E[\one_{e\subseteq Y} \mid Y\cap W] - \prod_{v\in e} \mb E[\one_{v\in Y} \mid Y\cap W]\bigg|\Bigg] \le \xi N,
    \]
    where we write $\on{wt}_\mc G(e)$ for the weight in $\mc G$ of a hyperedge $e$.
\end{theorem}
\cref{thm:KS} is an immediate consequence of the methods in \cite{KS23}, though it does not appear explicitly in that paper. So, for completeness we provide a proof in \cref{app:KS}. We remark that it is important that $W$ does not depend on $\mc G$ (in our application, we will need to apply it for many different $\mc G$, arising as book hypergraphs of $\mc H$).

\subsection{Upper bound on the independence probability of a binomial random subset}\label{subsec:independent-example}
Now, as a warm-up, we prove the case of \cref{thm:main-upper} where $L=0$ (i.e., where we are considering the probability that $R$ is an independent set). We start with an entropic lemma related to the ``sequential counting'' point of view described in \cref{subsec:sequential}. Given a random set $X$, we randomly partition it into ``batches'' $X_1,\dots,X_T$, and we consider the choices for each batch given previous batches.

 \begin{lemma}\label{lem:seq-ent}
        Let $X\subseteq V$ be a random subset of a ground set $V$ (with any distribution). For a set $S\subseteq V$, let $\on{ex}(S)\subseteq V$ be a set of ``possible extensions of $S$'', containing every $v\notin S$ for which there is an outcome of $X$ including $v\cup S$. (It may also include other $v$.)
        
        Let $T\in \mb N$. Independently from $X$, consider a uniformly random partition of $V$ into $T$ sets $V_1,\dots,V_T$ (i.e., each element of $V$ is allocated to an independent uniformly random part). Let $X_t=X\cap V_t$ and $X_{<t}=X_1\cup \dots\cup X_{t-1}$ for each $t$. Then
\[\on H[X]\le \frac1T\sum_{t=1}^T\sum_{v\in V} \Pr[v\in X]\Big(1-\log\Pr[v\in X]+\log\Pr[v\in \on{ex}(X_{<t})]\Big).\]
    \end{lemma}
    \begin{proof}
         First, by the chain rule (\cref{fact:entropy}) we have
    \begin{equation}
        \on H[X_1,\dots,X_T]=\on H[X]+\on H[X_1,\dots,X_T\,|\,X]=\on H[X] + \mb E|X| \log T.\label{eq:entropy-setup}
    \end{equation}
    By the chain rule and subadditivity (\cref{fact:entropy}) we have
    \begin{equation}
        \on H[X_1,\dots,X_T] = \sum_{t=1}^{T} \on H[X_{t}\,|\,X_{<t}]\le \sum_{t=1}^{T}\sum_{v\in V}\on H[\one_{v\in X_t}\,|\,X_{<t}]. \label{eq:entropy-setup-2}
    \end{equation}
Now, note that $\on H[\one_{v\in X_t}\,|\,X_{<t}]=\mb E[h(Q_{v,t})]$, where we write $Q_{v,t}$ for the random variable $\mb E[\one_{v\in X_t}\,|\,X_{<t}]$. We have $\mb E Q_{v,t}=\Pr[v\in X_t]=(1/T)\Pr[v\in X]$ and $\Pr[Q_{v,t}\ne 0]\le \Pr[v\in \on{ex}(X_{<t})]$, so by the concavity of the binary entropy function $h$ we have
    \[
         \on H[\one_{v\in X_t}\,|\,X_{<t}]=\mb E[h(Q_{v,t})] \le \Pr[v\in\on{ex}(X_{<t})]\,h\bigg(\frac{\Pr[v\in X]}{T\Pr[v\in\on{ex}(X_{<t})]}\bigg). 
    \]
    Now, for any $0\le x\le y\le 1$ we have
    \[
        yh(x/y) = -x\log(x/y) - y(1-x/y)\log(1-x/y) \le - x\log (x/y) + x,
    \]
    (where we have used the elementary bound $-z\log z \le 1-z$, which holds for all $z\in [0,1]$). Applying this inequality with $x=\Pr[v\in X]/T$ and $y=\Pr[v\in \on{ex}(X_{<t})]$ yields
\[
         \on H[\one_{v\in X_t}\,|\,X_{<t}]\le -\frac1T\Pr[v\in X]\Big(\log \Pr[v\in X]-\log T-\log\Pr[v\in \on{ex}(X_{<t})]\Big)+\frac1T\Pr[v\in X]. 
    \]
    Substituting this into \cref{eq:entropy-setup,eq:entropy-setup-2} (and using that $\mb E|X|=\sum_{v\in V}\Pr[v\in X]$) yields the desired result.
    \end{proof}

We will also need a standard Poisson approximation lemma.
    \begin{lemma}\label{lem:moment}
        For $M, \beta>0$, there exist $A, \gamma>0$ such that the following holds. Let $\mc F$ be a sigma-algebra, let $(D_v)_{v\in V}$ be a collection of nonnegative integer valued random variables, and let $(\mu_v)_{v\in V}$ be a collection of $\mc F$-measurable random variables supported in $[0,M]$. Suppose that for all $a\in \{1,\dots,A\}$ we have
        \begin{equation*}
            \mb E\Bigg[\sum_{v\in V} \Big|\mb E[D_v(D_v-1)\dots (D_v-a+1)\,|\,\mc F] - \mu_{v}^a\Big| \Bigg]\le \gamma |V|. 
        \end{equation*}
        Then, we have 
        \[
            \mb E\Bigg[\sum_{v\in V} \big|\mb P[D_v=0\,|\,\mc F] - \exp(-\mu_{v})\big| \Bigg] \le \beta|V|. 
        \]
    \end{lemma}
    \cref{lem:moment} follows directly from the method of moments (the condition in \cref{lem:moment} says that for almost all $v$, the factorial moments of $D_v$ are well-approximated by the factorial moments of $\on{Poisson}(\mu_v)$, which implies that the distribution of $D_v$ is approximately $\on{Poisson}(\mu_v)$). For completeness, we also provide a direct proof.
    
    \begin{proof}[Proof of \cref{lem:moment}]
    For odd $A$, we have
    \[\sum_{a=0}^A(-1)^a\frac{\mb E[D_v(D_v-1)\dots(D_v-a+1)\,|\,\mc F]}{a!}\le \Pr[D_v=0\,|\,\mc F]\le \sum_{a=0}^{A-1}(-1)^a\frac{\mb E[D_v(D_v-1)\dots(D_v-a+1)\,|\,\mc F]}{a!}\]
    by the Bonferroni inequalities (i.e., the fact that the inclusion-exclusion formula alternately undercounts and overcounts).
    Also, by Taylor's theorem we have
    \[\sum_{a=0}^A(-1)^a\frac{\mu_v^a}{a!}=\exp(-\mu_v)+O\bigg(\frac{\mu_v^{A+1}}{(A+1)!}\bigg).\]
    Combining these estimates yields
    \[\big|\Pr[D_v=0\,|\,\mc F]-\exp(-\mu_v)\big|\le O_M\Bigg(\frac{M^A}{A!}+\sum_{a=1}^A\frac{\big|\mb E[D_v(D_v-1)\dots (D_v-a+1)\,|\,\mc F] - \mu_{v}^a\big|}{a!}\Bigg).\]
    Substituting the lemma assumption, summing over $V$, and averaging over $\mc F$, yields
    \[
            \mb E\Bigg[\sum_{v\in V} \big|\mb P[D_v=0\,|\,\mc F] - \exp(-\mu_{v})\big| \Bigg] \le O_M\bigg(\frac{M^A}{A!}+e\gamma\bigg)|V|.
        \]
        The desired result follows, assuming $A$ is sufficiently large in terms of $\beta,M$ and $\gamma>0$ is sufficiently small in terms of $\beta$.
    \end{proof}

Now we prove the promised special case of \cref{thm:main-upper}.

\begin{proof}[Proof of \cref{thm:main-upper} when $L=0$]
Our goal is to prove that 
\[\log \mb P[|\mc H[R]|=0]\le -\min_{\vec q\in [0,p]^V}\big(\on{D}_{\mr{KL}}(\vec{q}\,\|\,p)+\mc H(\vec q)\big)+\varepsilon pN.\]
To this end we will first rewrite $\log \mb P[|\mc H[R]|=0]$ in entropic terms (i.e., in terms of the entropy of the conditional distribution of $R$ given $|\mc H[R]|=0$). We will then apply \cref{thm:KS} to see that our conditional distribution can be written as a mixture of measures that behave like product measures from the point of view of the book hypergraphs defined in \cref{def:books}. We then apply \cref{lem:seq-ent} to each of the measures in the mixture, using \cref{lem:moment} and degree-boundedness to estimate the relevant quantities.

In addition to the parameters $\lambda,\varepsilon$ in the statement of \cref{thm:main-upper}, the proof will involve additional parameters $T,A,\gamma,\xi,\alpha,\beta,\theta,C_0$, which should be chosen as follows. First, $\alpha,\beta$ need to be small, and $T$ needs to be large, in terms of $\varepsilon,C$. Then, $\gamma$ needs to be small, and $A$ needs to be large, in terms of $\beta,C$ (this dependence will come from \cref{lem:moment}). Then, $\xi$ needs to be small in terms of $A,r,\gamma,C$, and $C_0$ needs to be large in terms of $A,C,r$. Then, $\theta$ needs to be small in terms of $A,r,\alpha,\xi,C_0$ (this dependence will come from \cref{thm:KS}).
Finally, $\lambda$ needs to be small in terms of $A,C,r,\gamma,\varepsilon,\theta$.

\medskip \noindent\textit{Step 1: Reducing to an entropic problem.} Let $\mc D$ be the collection of independent sets of $\mc H$, and let $Y$ be a sample from the conditional distribution of $R\sim \on{Bin}(V,p)$ given $R\in \mc D$. Let $\mc L(Y)$ be the law of $Y$. Then, we compute
\begin{align}
        \log \mb P[|\mc H[R]|=0]=\log \mb P[R \in \mc D] &= -\on D_{\mathrm{KL}}\big(\mc L(Y) \,\|\, \mathrm{Bin}(V,p)\big)\notag\\
        &=  \on H[Y] + \mb E\log\Bigg(\prod_{v\in V} p^{\one_{v\in Y}}(1-p)^{\one_{v\notin Y}}\Bigg)\notag\\
        &=  \on H[Y] + \sum_{v\in V} \Big(\Pr[v\in Y]\log p + \Pr[v\notin Y]\log(1-p)\Big).\label{eq:prob-entropy}
\end{align}
Apply \cref{thm:KS} (with $r_0=A(r-1)$ and our parameters $\alpha,\xi,C_0$) to obtain a set $W\subseteq V$, and for any possible outcome $U$ of $Y\cap W$ let
\[\on H^{U}[Y]=\on H[Y\,|\,Y\cap W=U],\quad q_v^U=\Pr[v\in Y\,|\,Y\cap W=U].\]
Using this notation, we can break down the formula in \cref{eq:prob-entropy} over outcomes of $Y\cap W$, to get
\begin{align}
\log \mb P[|\mc H[R]|=0]&= -\on D_{\mathrm{KL}}\big(\mc L(Y\cap W) \,\|\, \mathrm{Bin}(W,p)\big) \notag\\
&\qquad\qquad + \on H[Y\,|\,Y\cap W]+\sum_{v\in V\setminus W} \Big(\Pr[v\in Y]\log p + \Pr[v\notin Y]\log(1-p)\Big)\notag\\
& \le \mb E\Bigg[ \on H^{Y\cap W}[Y] +\sum_{v\in V\setminus W} \Big(q^{Y\cap W}_v\log p + (1-q^{Y\cap W}_v)\log(1-p)\Big)\Bigg].\label{eq:average-entropy-independent}
\end{align}
So, in order to obtain our desired upper bound on $\log \mb P[|\mc H[R]|=0]$, it suffices to prove a suitable upper bound, on average, on the quantities $\on H^{U}[Y]+\sum_{v\in V\setminus W} \big(q^{U}_v\log p + (1-q^{U}_v)\log(1-p)\big)$.

\medskip \noindent\textit{Step 2: Sequential entropy.}
Fix any outcome $U$ of $Y\cap W$; we will next apply \cref{lem:seq-ent} (with the ground set $V\setminus W$) to obtain a bound on $\on H^{U}[Y]$. Let $X$ have the conditional distribution of $Y\setminus W$ given $Y\cap W=U$, so $\on H^{U}[Y]=\on H[X]$. Let $q_v=q^U_v=\Pr[v\in X]$. The conclusion of \cref{thm:KS} tells us that, on average, the distribution of $X$ behaves like a product measure, from the point of view of our book hypergraphs.

For $S\subseteq V\setminus W$ and $v\in V\setminus W$, let $D_v(S)$ be the number of hyperedges $e\in \mc H$ such that $v\in e$ and $e\setminus v\subseteq S$. In other words, when $v\notin S$, this parameter $D_v(S)$ is (a lower bound on) the number of new hyperedges that would be introduced if we were to augment the induced subgraph $\mc H[S]$ with the additional vertex $v$. Let $\on{ex}(S)$ be the set of $v\in V\setminus W$ such that $D_v(S)=0$, let $V_1\cup\dots\cup V_T$ be a uniformly random partition of $V\setminus W$, let $X_t=X\cap V_t$ and let $b_v(t)=b_v^U(t)=\Pr[v\in \on{ex}(X_{<t})]=\Pr[D_v(X_{<t})=0]$. By \cref{lem:seq-ent}, we have
\begin{equation}
    \on H[X]\le \frac 1T\sum_{t =1}^{T}\sum_{v\in V\setminus W}q_v\big(1-\log q_v+\log b_v(t)\big).\label{eq:seq-ent-application}
\end{equation} 
Now, 
let \[\mu_v(t)=\mu_v^U(t)=\Big(\frac{t-1}T\Big)^{r-1}\sum_{\substack{e\in \mc H[V\setminus W]:\\v\in e}}\;\prod_{u\in e\setminus v}q_{u},\qquad \tilde b_v(t)=\tilde b_v^U(t)=\exp(-\mu_v(t)).\]
That is to say, if we assume that $X$ is approximately a product measure, then $\mu_v(t)$ is an estimate for the expected value of $D_v(X_{<t})$, and $\tilde b_v(t)$ is a Poisson-type estimate for $b_v(t)=\Pr[D_v(X_{<t})=0]$.

\medskip \noindent\textit{Step 3: Elementary estimates.} Now, recalling \cref{eq:average-entropy-independent}, the next step is to use \cref{eq:seq-ent-application} to estimate $\on H^U[Y] +\sum_{v\in V\setminus W} \big(q^{U}_v\log p + (1-q^{U}_v)\log(1-p)\big)=\on H[X]+\sum_{v\in V\setminus W} \big(q_v\log p + (1-q_v)\log(1-p)\big)$. We will end up with an error term depending on the quantities $|b_v(t)-\tilde b_v(t)|$ (i.e., depending on the quality of our Poisson-type estimate, which depends on the extent to which $X$ is approximately a product measure). Afterwards, the last step will be to use the conclusion of \cref{thm:KS} together with \cref{lem:moment} to show that quantities of the form $|b_v(t)-\tilde b_v(t)|$ are small, on average.

Note that $q_v\le p$ for each $v$, by Harris' inequality (\cref{lem:harris}), so
\begin{equation}
    \mu_v(t)\le \sum_{\substack{e\in \mc H[V\setminus W]:\\v\in e}}\;\prod_{u\in e\setminus v}q_u\le p^{r-1}\Delta_1(\mc H)\le C\label{eq:nondegenerate-probability-pre}
\end{equation} using our assumption that $\mc H$ is $(\lambda,C;p)$-degree-bounded.

If $b_v(t)\le \tilde b_v(t)$ then $\log b_v(t)-\log \tilde b_v(t)\le 0$. Otherwise, $\log b_v(t)-\log \tilde b_v(t)\le (b_v(t)-\tilde b_v(t))/\tilde b_v(t)\le O_C(b_v(t)-\tilde b_v(t))$ (here we are using \cref{eq:nondegenerate-probability-pre} to see that $\tilde b_v(t)=\exp(-\mu_v(t))=\Omega_C(1)$, and we are using $\on d\log x/\on dx=1/x\le 1/\tilde b_v(t)$ for $x\in [\tilde b_v(t),b_v(t)]$). In either case, we have
\[\log b_v(t)-\log \tilde b_v(t)\le O_C(|b_v(t)-\tilde b_v(t)|).\]
Recalling \cref{eq:seq-ent-application}, for each $U\subseteq W$ we obtain
\begin{align}&\on H[X] +\sum_{v\in V\setminus W} \big(q_v\log p + (1-q_v)\log(1-p)\big)\notag\\
&\qquad\qquad\le \sum_{v\in V\setminus W}\big(q_v-q_v\log q_v+q_v\log p+(1-q_v)\log(1-p)\big) \notag\\
&\qquad\qquad\qquad\qquad +\frac1T\sum_{t=1}^T\sum_{v\in V\setminus W} q_v \log \tilde b_v(t)
+ O_C\Bigg(\frac1T\sum_{t=1}^Tp\sum_{v\in V\setminus W}|b_v(t)-\tilde b_v(t)|\Bigg).\label{eq:H-expansion-independent}
\end{align}

Now, let $q_v=0$ for $v\in W$, and let $\vec q=(q_v)_{v\in V}\in [0,p]^V$. Using the elementary inequalities $q\le -(1-q)\log(1-q)+O(q^2)$ and $\log(1/(1-p))=O(p)$ we have
\begin{align}&\sum_{v\in V\setminus W}\big(q_v-q_v\log q_v+q_v\log p+(1-q_v)\log(1-p)\big)\notag\\
&\qquad \le -\sum_{v\in V}\bigg(q_v\log \frac{q_v}p+(1-q_v)\log \frac{1-q_v}{1-p}\bigg)+\sum_{v\in W}(1-q_v)\log \frac{1-q_v}{1-p}+O(p^2 N)\notag\\
&\qquad \le -\on D_{\mr{KL}}(\vec q\,\|\,p)+O(p^2 N)+O(p|W|)\le -\on D_{\mr{KL}}(\vec q\,\|\,p)+(\varepsilon/4)pN,\label{eq:H-KL-independent}\end{align}
using that $\alpha,\lambda$ are small with respect to $\varepsilon$ (recall $|W|\le \alpha N$ and $p\le \lambda$).
Also, with a Riemann sum approximation we have
\begin{align}\frac1T\sum_{t=1}^T\sum_{v\in V\setminus W} q_v \log \tilde b_v(t)&=-\frac 1T\sum_{t=1}^T\bigg(\frac{t-1}{T}\bigg)^{r-1}\cdot \sum_{v\in V}\sum_{\substack{e\in \mc H:\\v\in e}}\,\prod_{u\in e}q_{u}\notag\\
&=-\bigg(\frac1r+o_{T\to \infty}(1)\bigg)\cdot r\sum_{e\in \mc H}\prod_{u\in e}q_u\notag\\
&=-\mc H(\vec q)+o_{T\to \infty}(p^r|\mc H|)\notag\\
&=-\mc H(\vec q)+o_{T\to \infty}(CpN)\notag\\
&\le -\mc H(\vec q)+(\varepsilon/4) pN.\label{eq:H-mean-independent}\end{align}
(Here we write $o_{T\to \infty}(1)$ for a term that can be made arbitrarily small by making $T$ large, we are using that $T$ is large in terms of $C,\varepsilon$, and we are using $(\lambda,C;p)$-degree-boundedness to see that $p^r|\mc H|\le N p^r \Delta_1(\mc H)\le CpN$.)

Substituting \cref{eq:H-KL-independent,eq:H-mean-independent} into \cref{eq:H-expansion-independent}, we deduce
\begin{align}&\on H[X] +\sum_{v\in V\setminus W} \big(q_v\log p + (1-q_v)\log(1-p)\big)\notag\\
&\qquad \qquad \le -\min_{\vec q\in [0,p]^V}\big(\on{D}_{\mr{KL}}(\vec{q}\,\|\,p)+\mc H(\vec q)\big)+O_C\Bigg(\frac1T\sum_{t=1}^Tp\sum_{v\in V\setminus W}|b_v(t)-\tilde b_v(t)|\Bigg)+(\varepsilon/2)pN.\label{eq:conditional-with-random-error-independent}\end{align}

\medskip \noindent\textit{Step 4: Poisson approximation.}
Now, the final step is to use \cref{lem:moment} to show that the ``error term'' $\sum_{v\in V\setminus W}|b_v(t)-\tilde b_v(t)|$ is small, on average.

Let $\on{bk}^a(\mc H[V\setminus W];v)=\bigcup_{k}\on{bk}^a_k(\mc H[V\setminus W];v)$ and observe that these book hypergraphs were defined precisely so that
\begin{align}
\mb E[D_v(X_{<t})(D_v(X_{<t})-1)\dots(D_v(X_{<t})-a+1)]&=\sum_{e\in \on{bk}^a(\mc H[V\setminus W];v)}\mb E[\one_{e\subseteq X_{<t}}]\notag\\
&=\sum_{k=r-1}^{a(r-1)}\bigg(\frac{t-1}T\bigg)^k\sum_{e\in \on{bk}^a_k(\mc H[V\setminus W];v)}\mb E[\one_{e\subseteq X}].\label{eq:falling-factorial-independent}\end{align}
Also, recalling $q_u=\mb E[\one_{u\in X}]$, and writing $q_e(t)=((t-1)/T)^{|e|}\prod_{u\in e}q_u$, we have
\begin{align*}&\sum_{k=r-1}^{a(r-1)}\bigg(\frac{t-1}T\bigg)^k\sum_{e\in \on{bk}^a_k(\mc H[V\setminus W];v)}\;\prod_{u\in e}\mb E[\one_{u\in  X}]\\
&\qquad =\left(\sum_{e\owns v}q_{e\setminus v}(t)\right)^{a}\pm \sum_{\substack{e_{1},\dots,e_{a}\owns v\\\text{some }e_i=e_j}}q_{e_{1}\setminus v}(t)\dots  q_{e_{a}\setminus v}(t)\pm
 \sum_{\substack{e_{1},\dots,e_{a}\owns v\text{ distinct}\\|(e_1\cup \dots\cup e_a)\setminus v|<a(r-1)}}q_{(e_{1}\cup\dots\cup e_{a})\setminus v}(t),
\end{align*}
where the sums are over (sequences of) hyperedges in $\mc H[V\setminus W]$. Note that $\sum_{e\owns v}q_{e\setminus v}(t)=\mu_v(t)$ and
\[ \sum_{\substack{e_{1},\dots,e_{a}\owns v\text{ distinct}\\|(e_1\cup \dots\cup e_a)\setminus v|<a(r-1)}}q_{(e_{1}\cup\dots\cup e_{a})\setminus v}(t)\le \sum_{k=r-1}^{a(r-1)-1}p^{k}|\on{bk}^a_k(\mc H;v)|=O_{A,C,r}(\lambda)\le \gamma/4\]
by \cref{eq:book-degrees-goal-2} in \cref{lem:book-deg} (using that $\lambda$ is small in terms of $A,C,r,\gamma$), and
\begin{align*}
\sum_{\substack{e_{1},\dots,e_{a}\owns v\\\text{some }e_i=e_j}}q_{e_{1}\setminus v}(t)\dots  q_{e_{a}\setminus v}(t)\le \binom a2\sum_{e\owns v}q_{e\setminus v}(t)^2\left(\sum_{e\owns v}q_{e\setminus v}(t)\right)^{a-2}&\le  O_{A}\Big(\Delta_1(\mc H)p^{2(r-1)}(\Delta_1(\mc H)p^{r-1})^{a-2}\Big)\\&\le O_{A,C}(p^{r-1})\le O_{A,C}(\lambda)\le \gamma/4
\end{align*}
by $(\lambda,C;p)$-degree-boundedness, and the assumptions $p\le \lambda$ and $r\ge 2$. So, we have
\begin{equation}
   \sum_{k=r-1}^{a(r-1)}\bigg(\frac{t-1}T\bigg)^k\sum_{e\in \on{bk}^a_k(\mc H[V\setminus W];v)}\prod_{u\in e}\mb E[\one_{u\in  X}]=\mu_v(t)^a\pm \gamma/2. \label{eq:product-independent}
\end{equation}
Now, the conclusion of \cref{thm:KS} (recalling \cref{lem:book-deg} to verify the condition \cref{eq:KS-degree-condition}, using that $\lambda$ is small in terms of $\theta$) implies that 
\[
\mb E\sum_{v\in V\setminus W} \Bigg|\sum_{e\in \on{bk}^a_k(\mc H[V\setminus W];v)}\bigg(\mb E[\one_{e\subseteq Y}\,|\,Y\cap W]-\prod_{u\in e}\mb E[\one_{u\in Y}\,|\,Y\cap W]\bigg)\Bigg| 
\le  \xi N.
\]
Taking an appropriate weighted sum over $k$ then yields
\[
\mb E\sum_{v\in V\setminus W} \Bigg|
\sum_{k=r-1}^{a(r-1)}\bigg(\frac{t-1}T\bigg)^k\sum_{e\in \on{bk}^a_k(\mc H[V\setminus W];v)}\bigg(\mb E[\one_{e\subseteq Y}\,|\,Y\cap W]-\prod_{u\in e}\mb E[\one_{u\in Y}\,|\,Y\cap W]\bigg)\Bigg|
\le  a(r-1)\xi N
\]
which is at most $(\gamma/2)N$
(recalling that $\xi$ is small in terms of $\gamma,A$).
Then for all $t$ let $Y_t=(Y\setminus W)\cap V_t$ and $Y_{<t}=Y_1\cup \dots\cup Y_{t-1}$, so  $X_{<t}$ has the distribution of $Y_{<t}$ given $Y\cap W=U$. Substituting \cref{eq:falling-factorial-independent,eq:product-independent} into the above inequality, we obtain
\[
\mb E\sum_{v\in V\setminus W} \Big|\mb E\big[D_v(Y_{<t})(D_v(Y_{<t})-1)\dots(D_v(Y_{<t})-a+1)\,\big|\,Y\cap W\big]-\mu_v^{Y\cap W}(t)^a\Big| \le  \gamma N,
\]
so, recalling \cref{eq:nondegenerate-probability-pre}, by \cref{lem:moment} (with $\mc F$ being the $\sigma$-algebra generated by $Y\cap W$) we have
\[
   \mb E\sum_{v\in V\setminus W}|b_v^{Y\cap W}(t)-\tilde b_v^{Y\cap W}(t)|=\mb E\sum_{v\in V\setminus W}\Big|\Pr[D_v(Y_{<t})=0\,|\,Y\cap W]-\exp(-\mu_v^{Y\cap W}(t))\Big|\le \beta N.
\]
The desired result then follows by combining this with \cref{eq:conditional-with-random-error-independent,eq:average-entropy-independent} (recalling that $\beta$ is small in terms of $\varepsilon,C$).
\end{proof}

\subsection{General lower tails in binomial random subsets}
We now prove the full statement of \cref{thm:main-upper}, guided by the special case proved in the last subsection. In fact, we prove a slight strengthening of \cref{thm:main-upper} in which we simultaneously consider the lower tail of the numbers of vertices and hyperedges. This strengthened version of \cref{thm:main-upper} will be used in our proof of \cref{thm:main-m-upper}. 
\begin{theorem}
\label{thm:main-upper-general}For any $r\in \mb N_{\ge 2}$ and $C,\varepsilon>0$, there are $\lambda>0$ and $N_0\in \mb N$ such that the following holds. Consider $p\in (0,\lambda)$ and let $\mathcal{H}$ be an $N$-vertex $(\lambda,C;p)$-degree-bounded $r$-uniform hypergraph, with $N\ge N_0$. Then, for $R\sim \on{Bin}(V,p)$, any $L\ge 0$ and any $m\in \mb N\cup \{\infty\}$,
\[
\log \Pr\big[|\mc H[R]|\le L\text{ and }|R|\le m\big]\le -\min_{\vec{q}}\Big(\on{D}_{\mr{KL}}(\vec{q}\,\|\,p)+\Phi_L(\mc H(\vec q))\Big)+\varepsilon pN,
\]
where the minimum is over all $\vec q\in [0,p]^V$ whose sum of entries is at most $m$.
\end{theorem}

The main difference between the $L=0$ case and the $L>0$ case is that if $L>0$ the constraint $|\mc H[R]|\le L$ cannot be broken down into ``local'' constraints at each vertex, and therefore \cref{lem:seq-ent} is not the right tool. A first thought might be to instead consider the Laplace transform $\mb E \exp(-\tau |\mc H[R]|)$ (also known as the moment generating function), which is in some sense a more ``local'' object: if we are building $R$ in a sequential fashion, we can accumulate a factor of $\exp(-\tau)$ for each hyperedge we encounter. The Chernoff bound provides a standard way to estimate large deviation probabilities in terms of the Laplace transform, namely
\[\Pr[Z\le L]\le \inf_{\tau\ge 0}\frac{\mb E[\exp(-\tau Z)]}{\exp(-\tau L)},\]
and this is well-known to produce essentially optimal bounds in the case where $Z$ is Poisson.

It is possible to describe the Laplace transform in entropic terms (via the so-called \emph{Gibbs variational formula}), but this involves a certain tilted measure, and would require some changes to the Kozma--Samotij product decomposition lemma. Instead, we consider some direct entropic manipulations motivated by the Chernoff bound (but not, strictly speaking, an instance of it).

The first step is a ``tilted'' version of \cref{lem:seq-ent}.

\begin{lemma}\label{lem:seq-ent-tilted}
Let $\mc H$ be an $r$-uniform hypergraph on the vertex set $V$, and let $X\subseteq V$ be a random subset of $V$. For $v\in V$ and $S\subseteq V$, let $D_v(S)$ be the number of hyperedges $e\in \mc H$ containing $v$ such that $e\setminus v\subseteq S$.

Let $T\in \mb N$. Independently from $X$, consider a uniformly random partition of $V$ into $T$ sets $V_1,\dots,V_T$. Let $X_t=X\cap V_t$ and $X_{<t}=X_1\cup\dots\cup X_{t-1}$. Then, for every $\tau\ge 0$,
\begin{equation}
    \on H[X]-\tau\mb E|\mc H[X]|
    \le\frac1T
    \sum_{t=1}^T\sum_{v\in V}
    \Pr[v\in X]
    \Big(
        1-\log \Pr[v\in X]
        +\log \mb E \exp(-\tau D_v(X_{<t}))
    \Big).
    \label{eq:seq-ent-tilted}
\end{equation}
\end{lemma}

\begin{proof}
As in the proof of \cref{lem:seq-ent}, by the chain rule and
subadditivity we have
\begin{equation}
    \on H[X]+\mb E|X|\log T
    \le \sum_{t=1}^T\sum_{v\in V}\mb E h(Q_{v,t}),
    \label{eq:tilted-chain-rule}
\end{equation}
where $Q_{v,t}=\Pr[v\in X_t\mid X_{<t}]$.

Now, note that 
\[|\mc H[X]|\ge \sum_{t=1}^T\sum_{v\in V}
     \one_{\{v\in X_t\}}D_v(X_{<t}).\]
Indeed, this sum counts hyperedges of $\mc H[X]$ which have a single vertex in $X_t$, and all other vertices in $X_{<t}$, for some $t$. So, recalling \cref{eq:tilted-chain-rule}, and noting that $\mb E\big[\one_{\{v\in X_t\}}D_v(X_{<t})\big]=\mb E\big[Q_{v,t}D_v(X_{<t})\big]$, we have
\begin{equation}
    \on H[X]-\tau\mb E|\mc H[X]|\le
\sum_{t=1}^T\sum_{v\in V}\mb E\big[h(Q_{v,t})-\tau D_v(X_{<t})Q_{v,t}\big]
-\mb E|X|\log T.\label{eq:tilted-entropy-intermediate}
\end{equation}
We next need an upper bound on $\mb E\big[h(Q_{v,t})-\tau D_v(X_{<t})Q_{v,t}\big]$, using the fact that $\mb EQ_{v,t}=\mb P[v\in X_t]=\Pr[v\in X]/T$. We do this by introducing a Lagrange multiplier. Indeed, letting $a_v=\Pr[v\in X]/T$, for any $\theta\in \mb R$ we have
\begin{align}
    \mb E\big[h(Q_{v,t})-\tau D_v(X_{<t})Q_{v,t}\big]&= \mb E\big[h(Q_{v,t})-\tau D_v(X_{<t})Q_{v,t}\big]+\theta(\mb EQ_{v,t}-a_v)\notag\\
    &=-\theta a_v+\mb E\Big[h(Q_{v,t})+\big(\theta-\tau D_v(X_{<t})\big)Q_{v,t}\Big]\notag\\
    &\le  -\theta a_v+\mb E\Big[\log\Big(1+\exp\big(\theta-\tau D_v(X_{<t})\big)\Big)\Big]\notag\\
    &\le  -\theta a_v+\mb E\big[\exp\big(\theta-\tau D_v(X_{<t})\big)\big]\notag\\
    &=-\theta a_v+e^\theta\mb E\exp(-\tau D_v(X_{<t})),\label{eq:lagrange}
\end{align}
where we have used that $\sup_{q\in [0,1]}(h(q)+sq)=h(q_s)+s q_s=\log(1+e^{s})$ for $q_s=e^s/(1+e^s)$ (which is a simple calculus exercise), and the elementary inequality $\log(1+x)\le x$.

Now, we take\footnote{This is not well-defined if $a_v=\Pr[v\in X]=0$, but then we can simply remove $v$ from $V$ (it is irrelevant for the lemma).} $\theta=\log a_v-\log \mb E\exp(-\tau D_v(X_{<t}))$ (this is the choice of $\theta$ that minimises the expression in \cref{eq:lagrange}), and obtain
\[\mb E\big[h(Q_{v,t})-\tau D_v(X_{<t})Q_{v,t}\big]\le a_v\Big(1-\log a_v+\log\mb E\exp(-\tau D_v(X_{<t}))\Big).\]
We then substitute this into \cref{eq:tilted-entropy-intermediate}. Recalling that $a_v=\Pr[v\in X]/T$ and $\mb E|X|=\sum_{v\in V}\Pr[v\in X]$, the desired result follows.
\end{proof}

We will also need a variant of \cref{lem:moment}, where instead of estimating the zero-probability we estimate the Laplace transform.

\begin{lemma}
\label{lem:moment-laplace}
        For $M, \beta>0$, there exist $A, \gamma>0$ such that the following holds. Let $\mc F$ be a sigma-algebra, let $(D_v)_{v\in V}$ be a collection of nonnegative integer valued random variables, and let $(\mu_v)_{v\in V}$ be a collection of $\mc F$-measurable random variables supported in $[0,M]$. Suppose that for all $a\in \{1,\dots,A\}$ we have
        \[
            \mb E\Bigg[\sum_{v\in V} \Big|\mb E[D_v(D_v-1)\dots (D_v-a+1)\,|\,\mc F] - \mu_{v}^a\Big| \Bigg]\le \gamma |V|. 
        \]
        Then, for any nonnegative $\mc F$-measurable random variable $\tau$,
        \[
            \mb E\Bigg[\sum_{v\in V} \Big|\mb E[\exp(-\tau D_v)\,|\,\mc F] - \exp(-(1-e^{-\tau})\mu_{v})\Big| \Bigg] \le \beta|V|. 
        \]
    \end{lemma}
    \cref{lem:moment-laplace} is basically just the fact that the Laplace transform of a random variable is its moment-generating function (so its value can be approximated in terms of moments). It can also be proved as a consequence of \cref{lem:moment} via a random subsampling argument (this technique is sometimes called \emph{R\'enyi thinning}), as follows.
    \begin{proof}[Proof of \cref{lem:moment-laplace}]
    Let $\rho=1-e^{-\tau}$. We think of $D_v$ as being the size of a random set, and define $D_v'$ to be the size of a $\rho$-random subset of this random set. 
    Formally, conditionally on $\mc F$ and the random variables $(D_v)_{v\in V}$, we let $D_v'\sim \on{Bin}(D_v,\rho)$.
    
    We compute
    \[ \Pr[D_v'=0\mid\mc F] = \mb E[(1-\rho)^{D_v}\,|\,\mc F] = \mb E[e^{-\tau D_v}\,|\,\mc F], \]
    and
    \[ \mb E[D_v'(D_v'-1)\dots(D_v'-a+1)\,|\,\mc F] = \rho^a\mb E[D_v(D_v-1)\dots(D_v-a+1)\,|\,\mc F].\]
    Since $0\le\rho \mu_v\le M$, the desired result then follows by applying \cref{lem:moment} to the variables $D_v'$.
    \end{proof}

Now we prove \cref{thm:main-upper-general}.

\begin{proof}[Proof of \cref{thm:main-upper-general}]
We proceed in a similar way to the proof in \cref{subsec:independent-example}, so we will be brief with details that we have already seen there. We will need auxiliary parameters $T,A,\gamma,\xi,\alpha,\beta,\theta,C_0$, chosen in exactly the same order as in \cref{subsec:independent-example}.

\medskip
\noindent
\textit{Step 1: Reducing to an entropic problem.}
The first step is basically the same as in \cref{subsec:independent-example}. Let $\mc D_{L,m}=\{S\subseteq V:|\mc H[S]|\le L\text{ and }|S|\le m\}$, and let $Y$ be a sample from the conditional distribution of $R\sim\on{Bin}(V,p)$ given $R\in\mc D_{L,m}$. 
Apply \cref{thm:KS} with the parameters $r_0=A(r-1)$ and $\alpha,\xi,C_0$, to obtain a set $W\subseteq V$. For any possible outcome $U$ of $Y\cap W$ we let
\[\on H^{U}[Y]=\on H[Y\,|\,Y\cap W=U],\quad q_v^U=\Pr[v\in Y\,|\,Y\cap W=U],\]
and we obtain
\begin{align}
\log \mb P[|\mc H[R]|\le L\text{ and }|R|\le m]&= -\on D_{\mathrm{KL}}\big(\mc L(Y\cap W) \,\|\, \mathrm{Bin}(W,p)\big) \notag\\
&\qquad\qquad + \on H[Y\,|\,Y\cap W]+\sum_{v\in V\setminus W} \Big(\Pr[v\in Y]\log p + \Pr[v\notin Y]\log(1-p)\Big)\notag\\
& \le \mb E\Bigg[ \on H^{Y\cap W}[Y] +\sum_{v\in V\setminus W} \Big(q^{Y\cap W}_v\log p + (1-q^{Y\cap W}_v)\log(1-p)\Big)\Bigg].\label{eq:average-entropy-lower-tail}
\end{align}

\medskip\noindent\textit{Step 2: Tilted sequential entropy.}
As before, fix any outcome $U$ of $Y\cap W$, and let $X$ have the conditional distribution of $Y\setminus W$ given $Y\cap W=U$, so $\on H^{U}[Y]=\on H[X]$. Let $q_v=q^U_v=\Pr[v\in X]$. For $S\subseteq V\setminus W$ and $v\in V\setminus W$, let $D_v(S)$ be the number of hyperedges $e\in \mc H$ such that $v\in e$ and $e\setminus v\subseteq S$. Let $V_1\cup\dots\cup V_T$ be a uniformly random partition of $V\setminus W$, and let $X_t=X\cap V_t$.

We also need to make a new definition to incorporate the Laplace transform: let $d_{v,\tau}(t)=d_{v,\tau}^U(t)=\mb E\exp(-\tau D_v(X_{<t}))$.

By the definition of $X$ (via $Y$) we always have $|\mc H[X]|\le L$, so for any $\tau\ge 0$ we have $\on H[X]\le\on H[X]-\tau\mb E|\mc H[X]|+\tau L$. So, \cref{eq:seq-ent-tilted} gives
\begin{align}
    \on H[X]\le\frac1T\sum_{t=1}^T\sum_{v\in V\setminus W}q_v\big(1-\log q_v+\log d_{v,\tau}(t)\big)+\tau L.
    \label{eq:seq-ent-tilted-application}
\end{align}
Then, let \[\mu_v(t)=\mu_v^U(t)=\Big(\frac{t-1}T\Big)^{r-1}\sum_{\substack{e\in \mc H[V\setminus W]:\\v\in e}}\;\prod_{u\in e\setminus v}q_{u},\qquad \tilde d_{v,\tau}(t)=\tilde d_{v,\tau}^U(t)=\exp\big(-(1-e^{-\tau})\mu_v(t)\big).\]
\medskip\noindent\textit{Step 3: Elementary estimates and optimisation over the tilt.}
First note that 
$\mu_v(t)\le p^{r-1}\Delta_1(\mc H)\le C$
by $(\lambda,C;p)$-boundedness and Harris' inequality. So, $\tilde d_{v,\tau}(t)\ge e^{-C}=\Omega_C(1)$. As in \cref{subsec:independent-example}, considering separately the cases $d_{v,\tau}(t)\le \tilde d_{v,\tau}(t)$ and $d_{v,\tau}(t)> \tilde d_{v,\tau}(t)$, we see that
\[\log d_{v,\tau}(t)-\log\tilde d_{v,\tau}(t)\le O_C(|d_{v,\tau}(t)-\tilde d_{v,\tau}(t)|),\]
so we deduce from \cref{eq:seq-ent-tilted-application} that
\begin{align}&\on H[X] +\sum_{v\in V\setminus W} \big(q_v\log p + (1-q_v)\log(1-p)\big)\notag\\
&\qquad\qquad\le \sum_{v\in V\setminus W}\big(q_v-q_v\log q_v+q_v\log p+(1-q_v)\log(1-p)\big) \notag\\
&\qquad\qquad\qquad\qquad +\frac1T\sum_{t=1}^T\sum_{v\in V\setminus W} q_v \log \tilde d_{v,\tau}(t)+\tau L
+ O_C\Bigg(\frac1T\sum_{t=1}^Tp\sum_{v\in V\setminus W}|d_{v,\tau}(t)-\tilde d_{v,\tau}(t)|\Bigg).\label{eq:H-expansion-lower-tail}
\end{align}
Let $q_v=0$ for $v\in W$, and let $\vec q=(q_v)_{v\in V}\in [0,p]^V$. Note that $\sum_v q_v\le \mb E |X|\le m$ by the definition of $X$ via $Y$. Exactly as in \cref{eq:H-KL-independent}, we have
\begin{align}&\sum_{v\in V\setminus W}\big(q_v-q_v\log q_v+q_v\log p+(1-q_v)\log(1-p)\big)\notag\\
&\qquad \le -\on D_{\mr{KL}}(\vec q\,\|\,p)+O(p^2 N)+O(p|W|)\le -\on D_{\mr{KL}}(\vec q\,\|\,p)+(\varepsilon/4)pN,\label{eq:H-KL-lower-tail}\end{align} and as in \cref{eq:H-mean-independent} we have
\begin{equation}\frac1T\sum_{t=1}^T\sum_{v\in V\setminus W} q_v \log \tilde d_{v,\tau}(t)\le -(1-e^{-\tau})\mc H(\vec q)+(\varepsilon/4) pN.\label{eq:H-mean-lower-tail}\end{equation}
Now, we set $\tau=\max(0,\log(\mc H(\vec q)/L))$, as the value\footnote{If $L=0$ we need to send $\tau\to \infty$, or in this case we can just use the $L=0$ approach in the previous subsection.} of $\tau\ge 0$ that minimises $-(1-e^{-\tau})\mc H(\vec q)+\tau L$, so that 
\[-(1-e^{-\tau})\mc H(\vec q)+\tau L=-\Phi_L(\mc H(\vec q))\]
(this is essentially just the Chernoff bound calculation for Poisson distributions). Combining this with \cref{eq:H-KL-lower-tail,eq:H-mean-lower-tail,eq:H-expansion-lower-tail}, we deduce
\begin{align}&\on H[X] +\sum_{v\in V\setminus W} \big(q_v\log p + (1-q_v)\log(1-p)\big)\notag\\
&\qquad \qquad \le -\min_{\vec q}\big(\on{D}_{\mr{KL}}(\vec{q}\,\|\,p)+\Phi_L(\mc H(\vec q))\big)+O_C\Bigg(\frac1T\sum_{t=1}^Tp\sum_{v\in V\setminus W}|d_{v,\tau}(t)-\tilde d_{v,\tau}(t)|\Bigg)+(\varepsilon/2)pN,\label{eq:conditional-with-random-error-lower-tail}\end{align}
where the minimum is over all $\vec q\in [0,p]^V$ whose sum of entries is at most $m$.

\medskip\noindent\textit{Step 4: Poisson approximation for the exponential moments.}
It remains to control the final term in
\cref{eq:conditional-with-random-error-lower-tail}. We proceed in exactly the same way as in \cref{subsec:independent-example}, using \cref{lem:moment-laplace} instead of \cref{lem:moment}. Indeed, the calculations in \cref{subsec:independent-example} yield
\[
\mb E\sum_{v\in V\setminus W} \Big|\mb E\big[D_v(Y_{<t})(D_v(Y_{<t})-1)\dots(D_v(Y_{<t})-a+1)\,\big|\,Y\cap W\big]-\mu_v^{Y\cap W}(t)^a\Big| \le  \gamma N,
\]
so by \cref{lem:moment-laplace} we have
\[
   \mb E\sum_{v\in V\setminus W}|d_{v,\tau}^{Y\cap W}(t)-\tilde d_{v,\tau}^{Y\cap W}(t)|=\mb E\sum_{v\in V\setminus W}\Big|\mb E\big[\exp(-\tau D_v(Y_{<t}))\,\big|\,Y\cap W\big]-\exp(-(1-e^{-\tau})\mu_v^{Y\cap W}(t))\Big|\le \beta N.
\]
The desired result follows by combining this with \cref{eq:conditional-with-random-error-lower-tail,eq:average-entropy-lower-tail}.
\end{proof}

\subsection{Upper bounds on lower tails in uniform random sets}\label{subsec:uniform-random-upper}
In this subsection we prove \cref{thm:main-m-upper}, concerning uniform random subsets $R\sim \on{Unif}(V,m)$. Actually, we are able to deduce this from \cref{thm:main-upper-general} (this was the reason to include a constraint on $|R|$ in the statement of \cref{thm:main-upper-general}).

Before starting, we record some elementary estimates about the function $\Phi_L$.
\begin{fact}\label{fact:Phi}
For all $L\ge 0$ and $M>0$, we have
\begin{enumerate}
    \item $(1-e^{-1})M\le L+\Phi_L(M)$,
\item $\Phi_L'(M)=\max(0,1-L/M)\in[0,1]$.
\end{enumerate}
\end{fact}
Both parts of \cref{fact:Phi} are simple calculus exercises. The first part is actually a special case of the fact that $-\Phi_L(M)$ is the infimum of $-(1-e^{-\tau})M+\tau L$ over all $\tau\ge 0$ (which motivated our choice of $\tau$ in the proof of \cref{thm:main-upper-general}). The second part says that $\Phi_L$ is an increasing 1-Lipschitz function.

We also need a simple consequence of Stirling's inequality.
\begin{fact}\label{fact:LO}
    Consider any $p\in [0,1]$ and $N\in\mb N_{\ge 0}$ such that $pN\in \mb N_{\ge 0}$, and let $B_{p,N}\sim\on{Bin}(N,p)$. Then $\Pr[B_{p,N}=pN]=\Omega(1/\sqrt{1+pN})$.
\end{fact}

Now we prove \cref{thm:main-m-upper}.

\begin{proof}[Proof of \cref{thm:main-m-upper}]
    First, note that for any $p\in (0,1)$ and $R_{p}\sim \on{Bin}(V,p)$ we have
\[\Pr\big[|\mc H[R]|\le L]=\frac{\Pr\big[|\mc H[R_p]|\le L\text{ and }|R_p|=m\big]}{\Pr\big[|R_p|=m\big]}\le \frac{\Pr\big[|\mathcal{H}[R_p]|\le L\text{ and }|R_p|\le m\big]}{\Pr\big[|R_p|=m\big]}.\]
The key idea is to apply this inequality with $p$ much larger than $m/N$, so there is a huge entropy penalty for the event that $|R_p|$ is significantly less than $m$. This means that the conditional distribution of $R_p$ given $|R_p|\le m$ is very similar to the distribution of $R$.

The proof will involve an auxiliary parameter $B>0$, which must be large in terms of $\varepsilon,r,C$. We will take $\lambda$ to be small in terms of $B$.

\medskip
\noindent\textit{Step 1: Obtaining a clean variational problem.}
If $p=Bm/N$ then $\mc H$ is $(\lambda,O_{C,B,r,\varepsilon}(1);p)$-degree-bounded, so \cref{thm:main-upper-general} (with error parameter $\varepsilon/(4B)$) shows that
\[\log\Pr\big[|\mathcal{H}[R_p]|\le L\text{ and }|R_p|\le m\big]\le -\min_{\vec{q}}\Big(\on{D}_{\mr{KL}}(\vec{q}\,\|\,p)+\Phi_L(\mc H(\vec q))\Big)+(\varepsilon/4) m,\]
where the minimum is over all $\vec q\in [0,p]^V$ whose sum of entries $|\vec q|$ is at most $m$.

Then, letting $B_{m/N,N}\sim \on{Bin}(N,m/N)$ and $a_B=\log B+\log(1-m/N)-\log(1-Bm/N)$, we compute
\[
\on{D}_{\mr{KL}}(\vec{q}\,\|\,p)+\log\Pr\big[|R_p|=m\big] 
  =\on{D}_{\mr{KL}}(\vec{q}\,\|\,m/N)+a_B(m-|\vec{q}|)+\log \Pr[B_{m/N,N}=m].\]
  By \cref{fact:LO} we have $\log \Pr[B_{m/N,N}=m]\ge -O(\log m)\ge -(\varepsilon/4)m$ (using that $m\ge m_0$ is large).
So,
\[\log\Pr\big[|\mathcal{H}[R]|\le L\big]\le-\min_{\vec{q}:|\vec{q}|\le m}I_B(\vec{q})+(\varepsilon/2)m,\]
where
\[
I_B(\vec{q})=\on{D}_{\mr{KL}}(\vec{q}\,\|\,m/N)+\Phi_L(\mathcal{H}(\vec{q}))+a_B(m-|\vec{q}|).
\]
To prove the theorem, it suffices to show that
\[
\min_{\substack{\vec{q}\in[0,p]^{V}:|\vec{q}|\le m}
}I_B(\vec{q})\ge\min_{\vec{q}\in [0,1]^V:|\vec{q}|=m}I_B(\vec{q})-(\varepsilon/2)m.
\]

\medskip
\noindent\textit{Step 2: First reductions.}
Consider any $\vec{q}=(q_v)_{v\in V}\in[0,p]^{V}$ satisfying $|\vec{q}|\le m$;
it suffices to find $\vec{s}\in [0,1]^V$ satisfying $|\vec{s}|=m$ such that
$I_B(\vec{q})\ge I_B(\vec{s})-(\varepsilon/2)m$. There are two obvious candidates for $\vec s$: namely, the constant $(m/N)$-valued weighting $(m/N)\vec 1\in [0,p]^V$, and the rescaled weighting $\theta\vec{q}$, where $\theta=m/|\vec{q}|$.

So, suppose that the constant weighting $(m/N)\vec 1$ does not do the job: namely, suppose 
\[
  I_B(\vec{q})<I_B((m/N)\vec 1)-(\varepsilon/2)m<I_B((m/N)\vec 1)=\Phi_L(\mc H(m/N)).
  \]
  This implies that $L<\mathcal{H}(m/N)$ (otherwise we would have $\Phi_L(\mc H(m/N))=0$), and it implies that $I_B(\vec{q})\le \Phi_L(\mathcal{H}(m/N))\le \mathcal{H}(m/N)$. By $(\lambda,C;m/N)$-degree-boundedness we have $\mathcal{H}(m/N)\le O_C(m)$, so
  \begin{align}
      L&\le O_C(m),\label{eq:L-bound}\\
      I_B(\vec q)&\le O_C(m).\label{eq:I(q)-bound}
  \end{align}
Using that $I_B(\vec{q})\ge a_B(m-|\vec{q}|)=\Omega((m-|\vec q|)\log B)$ and $I_B(\vec{q})\ge \Phi_L(\mc H(\vec q))$, we deduce from \cref{eq:I(q)-bound} that
\begin{align}
m-|\vec{q}| &\le O_{C}(m/\log B),\label{eq:m-1}\\
\Phi_L(\mathcal{H}(\vec{q})) & \le O_{C}(m).\label{eq:m-2-pre}    
\end{align}
Then, we can use \cref{fact:Phi}(1) with \cref{eq:L-bound,eq:m-2-pre} to see that
\begin{equation}
\mathcal{H}(\vec{q}) \le O_{C}(m).\label{eq:m-2}\end{equation}

Now, our objective for the rest of the proof is to show that $I_B(\vec{q})\ge I_B(\theta \vec{q})-(\varepsilon/2)m$, using \cref{eq:m-1,eq:m-2}.

\medskip
\noindent\textit{Step 3: Continuity of $I_B$.}
By \cref{eq:m-1}, for each $v\in V$ we have $\theta q_v\le \theta p\le(1+O_{C}(1/\log B))\cdot B\cdot m/N$.
Writing $i_s(t)=t\log(t/s)+(1-t)\log((1-t)/(1-s))$, we compute
\[
    i_{s}'(t)=\log\left(\frac{t(1-s)}{s(1-t)}\right),
\]
so, recalling that $m/N\le\lambda$, and $\lambda$ is small in terms of $B$, and $a_B=\log B+\log(1-m/N)-\log(1-Bm/N)$, we have
\begin{align*}
i_{m/N}(\theta q_{v})-i_{m/N}(q_{v}) &\le(\theta q_{v}-q_{v})\log\left(\frac{\theta q_v(1-m/N)}{(m/N)(1-\theta q_v)}\right)\\
& \le(\theta q_{v}-q_{v})(a_B+O(1/\log B+B\lambda)).
\end{align*}
Recalling that $\lambda$ is small in terms of $B,\varepsilon$, and $B$ is large in terms of $\varepsilon$, it follows that
\[
\on{D}_{\mr{KL}}(\theta \vec{q}\,\|\,m/N)-\on{D}_{\mr{KL}}(\vec{q}\,\|\,m/N)\le(m-|\vec{q}|)a_B+(\varepsilon/4)m.
\]
Also, by \cref{eq:m-2,eq:m-1} and \cref{fact:Phi}(2), we have
\[
\Phi_L(\mathcal{H}(\theta \vec{q}))-\Phi_L(\mathcal{H}(\vec q))\le \mathcal{H}(\theta \vec{q})-\mathcal{H}(\vec{q})=(\theta^{r}-1)\mathcal{H}(\vec{q})\le O_{r}(1/\log B)\cdot O_{C}(m),
\]
so 
\[
I_B(\theta \vec{q})-I_B(\vec{q})\le (\varepsilon/4)m+O_{C,r}(m/\log B)\le(\varepsilon/2)m
\]
as desired.
\end{proof}

\begin{remark}\label{rem:structure}
The above proof shows how to study lower tail probabilities with respect to a random vertex subset constrained to have exactly $m$ vertices. It would be a simple matter to generalise this proof to allow for different constraints on the number of vertices in given sets (e.g., we could demand that in one vertex subset we take exactly $m_1$ vertices, and in a second vertex subset we take exactly $m_2$ vertices). By studying how such constraints affect the lower-tail probabilities, one can make conclusions about the typical structure of a sample from the unconstrained lower-tail event.
\end{remark}

\section{Lower bounds on lower tails}\label{sec:lower}
In this section we prove \cref{thm:main-m-lower,thm:main-lower}.
For both these proofs, the main tool we will need is the following lower bound on the number of independent sets of a given size in a hypergraph.
Importantly, we can restrict our attention to independent sets which are ``typical'' with respect to certain test functions. 
\begin{theorem}\label{thm:GW-lower-bound}
Fix constants $r\in \mb N_{\ge 2}$ and $Q\in \mb N$ and $C,\nu>0$. Consider a parameter $\alpha\in (N^{\nu-1},N^{-\nu})$, an $N$-vertex $r$-uniform hypergraph  $\mc G$, and uniform ``test hypergraphs'' $\mc F_1,\dots,\mc F_Q$ on the same vertex set $V$ as $\mc G$.
Suppose
\begin{itemize}
    \item $\mc G$ is $(N^{-\nu},C;\alpha)$-degree-bounded,
    \item Writing $k_i$ for the uniformity of $\mc F_i$, we have $k_i=O(1)$ and $|\mc F_i|\alpha^{k_i}\to \infty$, and $\Delta_s(\mc F_i)\le o(|\mc F_i|\alpha^{s})$ for $s\in \{1,\dots,k_i-1\}$.
    \item No hyperedge of any $\mc F_i$ contains a hyperedge of $\mc G$.
\end{itemize}
Then there is $m\in \mb N$ with $m=(1+o(1))\alpha N$, such that there are at least 
\[\binom{N}m\exp\Big(-\mc G(\alpha)-o(\alpha N)\Big)\]
size-$m$ sets $X\subseteq V$ satisfying $|\mc G[X]|=0$ and $|\mc F_i[X]|=(1+o(1))\alpha^{k_i}|\mc F_i|$ for each $i\in\{1,\dots,Q\}$.
\end{theorem}
\cref{thm:GW-lower-bound} can be deduced from work of Guo and Warnke~\cite{GW}, in which they analyse a random process that greedily builds an independent set. Unfortunately, their paper was not written with enumeration in mind\footnote{Actually, \cite{GW} is an unpublished draft of a paper that is in the process of being finalised. In the final version of their paper, the authors intend to include additional theorems on enumeration, which will hopefully make the deduction of \cref{thm:GW-lower-bound} much more transparent.}, so it takes quite some explanation to see how to deduce the statement of \cref{thm:GW-lower-bound} from statements and estimates in \cite{GW}. We defer this explanation to \cref{app:semi-random}.

If we apply \cref{thm:GW-lower-bound} after randomly subsampling the hyperedges of a hypergraph, we can obtain a similar lower bound on the number of sets inducing at most $L$ hyperedges (for general $L$, not just $L=0$). The following theorem encapsulates such a subsampling argument. (Recall the definition of $\phi$ from \cref{eq:phi}.)
\begin{theorem}\label{thm:lower-bound-lower-tail-uniform}
In the setting of \cref{thm:GW-lower-bound}, for each $y\in [0,1]$ there is $m\in \mb N$ with $m=(1+o(1))\alpha N$, such that there are at least 
\[\binom{N}m\exp\Big(-\phi(y)\mc G(\alpha)-o(\alpha N)\Big)\]
size-$m$ sets $X\subseteq V$ satisfying $|\mc G[X]|\le y\mc G(\alpha)+o(\alpha N)$ and $|\mc F_i[X]|=(1+o(1))\alpha^{k_i}|\mc F_i|$ for each $i\in\{1,\dots,Q\}$.

\end{theorem}
\begin{proof}
We may assume that $y=\Omega(1)$, as the $y=o(1)$ case is implied by \cref{thm:GW-lower-bound} (since $\phi(o(1))=1-o(1)$).

    Let $\mc G'$ be a random subhypergraph of $\mc G$ obtained by including each hyperedge with probability $y$ independently, and let $\mc G''=\mc G\setminus \mc G'$. Then (e.g., by a Chernoff bound), with probability $1-o(1)$ we have $\mc G'(\alpha)=y\mc G(\alpha)+o(\alpha N)$ and $(\mc G\setminus \mc G')(\alpha)=(1-y)\mc G(\alpha)+o(\alpha N)$. Say such an outcome of $\mc G'$ is ``good''.

    Now, we apply \cref{thm:GW-lower-bound} to $\mc G\setminus \mc G'$, using the test hypergraph\footnote{Using the fact that $\mc G$ is $(N^{-\nu},C;\alpha)$-degree-bounded, it is easy to verify that $\mc G'$ satisfies the requirements for being a test hypergraph, unless $|\mc G'(\alpha)|=O(1)$. In this case, add a tiny number of ``fake hyperedges'' to $\mc G'$ (this will not affect any of the following estimates).} $\mc G'$ in addition to $\mc F_1,\dots,\mc F_Q$. This tells us that if $\mc G'$ is good then there is $m(\mc G')=(1+o(1))\alpha N$ and a family $\mc A(\mc G')$ of at least 
    \[\binom{N}{m(\mc G')}\exp\Big(-(1-y)\mc G(\alpha)-o(\alpha N)\Big)\]
    size-$m(\mc G')$ sets $X\subseteq V$ satisfying $|(\mc G\setminus \mc G')[X]|=0$ and $|\mc F_i[X]|=(1+o(1))\alpha^{k_i}|\mc F_i|$ for each $i\in\{1,\dots,Q\}$, and 
    $|\mc G'[X]|=\alpha^{r}|\mc G'|+o(\alpha N)=y\mc G(\alpha)+o(\alpha N)$. In particular, all such sets satisfy $|\mc G[X]|\le y\mc G(\alpha)+o(\alpha N)$.

    Since there are only $N=\exp(o(\alpha N))$ possibilities for $m$, there must be some value of $m$ such that 
    \begin{equation}
        \mb E|\mc A_m(\mc G')|\ge \binom{N}{m}\exp\Big(-(1-y)\mc G(\alpha)-o(\alpha N)\Big),\label{eq:red-blue-pigeonhole}
    \end{equation}
    where the expected value is over the randomness of $\mc G'$, and
    \[\mc A_m(\mc G')=\begin{cases}
        \mc A(\mc G')&\text{if }m(\mc G')=m\text{ and }\mc G'\text{ is good},\\
        \emptyset&\text{if }m(\mc G')\ne m\text{ or }\mc G'\text{ is not good}.
     \end{cases}\]
     On the other hand, let $\mc A_m$ be the collection of \emph{all} size-$m$ subsets $X\subseteq V$ satisfying $|\mc G[X]|= y\mc G(\alpha)\pm \varepsilon\alpha N$ and $|\mc F_i[X]|=(1\pm \varepsilon)\alpha^{k_i}|\mc F_i|$ for each $i\in\{1,\dots,Q\}$, where $\varepsilon=o(1)$ is such that $\mc A_m(\mc G')\subseteq \mc A_m$ for all $\mc G'$. We have
     \begin{equation}
         \mb E|\mc A_m (\mc G')|\le \sum_{X\in \mc A_m}y^{|\mc G[X]|}=|\mc A_m|\exp\Big(y\mc G(\alpha)\log y+o(\alpha N)\Big).\label{eq:lutz-subsample}
     \end{equation}
     The desired lower bound on $|\mc A_m|$ then follows from \cref{eq:red-blue-pigeonhole,eq:lutz-subsample}.
\end{proof}
Next, for a given partition of the vertex set, we show how to deduce similar bounds for sets with prescribed intersection sizes with the parts of our partition. We do this with a double-counting argument, applying \cref{thm:lower-bound-lower-tail-uniform} after appropriately subsampling the vertices of each part. We will soon see that this result can be interpreted as a ``coarse grained'' version of \cref{thm:main-m-lower}.
\begin{theorem}\label{thm:prescribed-parts-lower}
For any $r\in \mb N_{\ge 2}$ and $\delta,C,\eta,\nu>0$, the following holds for sufficiently large $N$. Let $\alpha\in (N^{\nu-1},N^{-\nu})$ and let $\mc H$ be an $N$-vertex $(N^{-\nu},C;\alpha)$-degree-bounded $r$-uniform hypergraph. Let $\mc P$ be a partition of the vertex set into parts of size at least $\delta N$, and let $z\in [0,1]$.

Consider any $(m_W)_{W\in \mc P}\in \mb N_{\ge 0}^{\mc P}$ with $m_W\le (\alpha/2) |W|$ for each $W\in \mc P$, such that the sum of the $m_W$ is at least $\delta \alpha N$. Let $\vec x=(x_v)_{v\in V}\in [0,\alpha/2]^V$ be the vertex-weighting defined by $x_v=m_W/|W|$ for all $v\in W\in \mc P$. Then, there are at least
\[\left(\prod_{W\in \mc P}\binom{|W|}{m_W}\right)\exp\Big(-\phi(z)\mc H(\vec x)-\eta \alpha N\Big),\]
subsets $X\subseteq V$ satisfying $|\mc H[X]|\le z\mc H(\vec x)$ which have exactly $m_W$ vertices in each part $W\in \mc P$.
\end{theorem}
\begin{proof}
In this proof we view $\delta,C,\nu$ as constants (i.e., when we write asymptotic notation, the implicit constants may depend on $\delta,C,\nu$).

    Consider some small $\eps>0$ (we will need it to be small in terms of $\delta,C,r,\eta$), and
    for $W\in\mc P$ let
    \[u_W=\left\lfloor\frac{m_W}{\alpha}+\varepsilon|W|\right\rfloor\le |W|.\]
    Let $\mc U$ be the set of $U\subseteq V$ such that $|U\cap W|=u_W$ for all $W\in \mc P$. Using $(N^{-\nu},C;\alpha)$-degree-boundedness, we can compute that the average value of $\alpha^{r}|\mc H[U]|$, over all $U\in \mc U$, is $\mc H(\vec x)+O(\varepsilon\alpha N)=O(\alpha N)$, so by Markov's inequality there is a set $\mc U_{\mr{typ}}\subseteq \mc U$ with $|\mc U_{\mr{typ}}|\ge \Omega(\varepsilon)|\mc U|$ such that $\alpha^{r}|\mc H[U]|\le \mc H(\vec x)+O(\varepsilon\alpha N)$ for all $U\in \mc U_{\mr{typ}}$.
    
    Now, for all $U\in \mc U_{\mr{typ}}$, we apply either \cref{thm:GW-lower-bound} or \cref{thm:lower-bound-lower-tail-uniform} to $\mc H[U]$ (with adjusted values of $C,\nu$), with a 1-uniform test hypergraph for each part $W\in \mc P$. Specifically:
    \begin{itemize}
        \item If $z>  2\varepsilon^{1/4}$ and $\mc H(\vec x)\ge  \varepsilon^{1/4}\alpha N$, apply \cref{thm:lower-bound-lower-tail-uniform} with $y=z-\eps^{1/4}$. Note that if $S\subseteq U$ satisfies $|\mc H[S]|\le y\alpha^r|\mc H[U]|+o(\alpha N)$, then it satisfies $|\mc H[S]|\le y\mc H(\vec x)+O(\varepsilon\alpha N)\le z\mc H(\vec x)$ (for small $\varepsilon$). 
        \item If $z\le 2\varepsilon^{1/4}$ or $\mc H(\vec x)\le \varepsilon^{1/4}\alpha N$, apply \cref{thm:GW-lower-bound}, noting that an independent set $S\subseteq U$ certainly satisfies $|\mc H[S]|\le z\mc H(\vec x)$. In this case, set $y=0$, so $\phi(y)=1$.
    \end{itemize}
    In either case, we obtain some $m(U)=\alpha|U|+o(\alpha N)$ and some $(m_W(U))_{W\in \mc P}$ summing to $m(U)$ satisfying $m_W(U)=(1+o(1))\alpha|W\cap U|=(1+o(1))\alpha u_W$ (which in particular means $m_W(U)=m_W\pm 2\varepsilon \alpha N$ and $m_W(U)\ge m_W$) for all $W\in \mc P$, such that the number of subsets $S\subseteq U$ satisfying $|\mc H[S]|\le z\mc H(\vec x)$, with exactly $m_W(U)$ vertices in each part $W\in \mc P$, is at least
    \begin{equation}
        N^{-1/\delta}\binom {|U|}{m(U)}\exp(-\phi(y)\mc H[U](\alpha)-o(\alpha N))\ge \prod_{W\in \mc P}\binom{u_W}{m_W(U)}\exp\Big(-\phi(z)\mc H(\vec x)-(\eta/2)\alpha N\Big).\label{eq:lower-to-sum}
    \end{equation}
    Here we have used that $\phi$ is continuous on $[0,1]$, that $\eps>0$ is small in terms of $\eta$, and that $\binom {n_1+\dots+n_J}{k_1+\dots+k_J}\ge \binom {n_1}{k_1}\cdot\dots \cdot\binom {n_J}{k_J}$. We have also used that $|\mc P|\le 1/\delta$, so there are at most $N^{1/\delta}$ possibilities of $(m_W(U))_{W\in \mc P}$ satisfying $m_W(U)=\alpha|W\cap U|+o(\alpha N)$ for each $W\in \mc P$.

    There are at most $N^{1/\delta}=\exp(o(\alpha N))$ possibilities for $(m_W(U))_{W\in \mc P}$, so there is $\mc U^*\subseteq \mc U_{\mr{typ}}$ with $|\mc U^*|\ge \exp(-o(\alpha N))|\mc U_{\mr{typ}}|\ge\exp(-o(\alpha N))|\mc U| $ such that $(m_W(U))_{W\in \mc P}$ takes a common value $\vec m^*=(m_W^*)_{W\in \mc P}$ for all $U\in \mc U^*$.
    Then, let $\mc T$ be the set of triples $(U,S,X)$ such that
    \begin{itemize}
    \item $U\in \mc U^*$,
    \item $X\subseteq S\subseteq U$,
    \item $|\mc H[S]|\le z\mc H(\vec x)$,
    \item $|U\cap W|=u_W$ and $|S\cap W|=m^*_W$ and $|X\cap W|=m_W$.
    \end{itemize}
    Summing \cref{eq:lower-to-sum} over all suitable $U$ and $S$, we have
        \begin{align*}|\mc T|&\ge\left(\prod_{W\in \mc P}\binom{m_W^*}{m_W}\binom{u_W}{m_W^*}\right)|\mc U^*|\exp\Big(-\phi(z)\mc H(\vec x)-(\eta/2)\alpha N\Big)\\
        &\ge \left(\prod_{W\in \mc P}\binom{|W|}{u_W}\binom{m_W^*}{m_W}\binom{u_W}{m_W^*}\right)\exp\Big(-\phi(z)\mc H(\vec x)-\eta\alpha N\Big).\end{align*}
    On the other hand, for any $X$ with $|X\cap W|=m_W$ for all $W\in \mc P$, the number of pairs $(U,S)$ for which $(U,S,X)\in \mc T$ is at most
    \[\prod_{W\in \mc P}\binom{|W|-m_W}{m^*_W-m_W}\binom{|W|-m^*_W}{u_W-m^*_W}.\]
    Using the identity
    \[\binom{u_W}{m^*_W}\binom{m^*_W}{m_W}\binom{|W|}{u_W}=\binom{|W|}{m_W}\binom{|W|-m_W}{m^*_W-m_W}\binom{|W|-m^*_W}{u_W-m^*_W},\]
we see that the number of $X$ which feature in some $(U,S,X)\in \mc T$ (which is a lower bound on our desired number of subsets) is at least
\[\left(\prod_{W\in \mc P}\binom{|W|}{m_W}\right)\exp\Big(-\phi(z)\mc H(\vec x)-\eta \alpha N\Big),\]
as desired.
\end{proof}

For the proof of \cref{thm:main-m-lower} we also need an analytic lemma, showing that we can assume no entries of $\vec q$ are much larger than $m/N$.
\begin{lemma}\label{lem:truncation-lower}
    In the setting of \cref{thm:main-m-lower}, there is $A=O_{C,r}(1)$ such that the following holds. Let $\vec q$ be a minimiser of $\on{D}_{\mr{KL}}(\vec{q}\,\|\,m/N)+\Phi_L(\mc H(\vec q))$, over all $\vec q=(q_v)_{v\in V}\in [0,1]^V$ with sum of entries exactly $m$. Then $q_v\le Am/N$ for all $v\in V$.
\end{lemma}
\begin{proof}
First, note that we can assume $L< \mc H(m/N)$. Indeed, if $L\ge \mc H(m/N)$ it is easy to see that the unique minimiser is the constant weighting $(m/N)\vec 1\in [0,1]^V$ (by strict convexity of $\on{D}_{\mr{KL}}(\vec{q}\,\|\,m/N)$, and the fact that $\Phi_L(M)=0$ for $L\ge M$). So, by $(N^{-\nu},C;m/N)$-degree-boundedness we have
 \begin{equation}
      L\le O_C(m).\label{eq:L-bound-2}
  \end{equation}

Now, let $p=m/N$. The dependence of $A$ on $C,r$ will be moderated by an additional parameter $B$ (we need $B$ to be large in terms of $C,r$, and $A$ to be large in terms of $B,C,r$).

The plan is to show that if there were $v$ such that $q_v>Ap$, then we would be able to transfer some weight from $v$ to some other $w$ to decrease the value of $\on{D}_{\mr{KL}}(\vec{q}\,\|\,m/N)+\Phi_L(\mc H(\vec q))$ (and since we are assuming $\vec q$ is a minimiser, this is impossible).

First, we need to find a suitable vertex $w$. Since we are assuming that $\vec q$ is a minimiser, we have \[\on{D}_{\mr{KL}}(\vec{q}\,\|\,p)+\Phi_L(\mc H(\vec q))\le \on{D}_{\mr{KL}}(p\vec 1\,\|\,p)+\Phi_L(\mc H(p\vec 1))=O_{C,r}(pN),\]where we write $p\vec 1$ for the vertex-weighting that gives every vertex a weight of exactly $p=m/N$, and we used $(N^{-\nu},C;p)$-degree-boundedness. In particular, we have 
\begin{align}
    \on{D}_{\mr{KL}}(\vec{q}\,\|\,p)&\le O_{C,r}(pN),\label{eq:Dq-bound}\\
    \Phi_L(\mc H(\vec q))&\le O_{C,r}(pN).\label{eq:Hq-bound-pre}
\end{align}
Then, we can use \cref{fact:Phi}(1) with \cref{eq:L-bound-2,eq:Hq-bound-pre} to see that
\begin{equation}
        \mc H(\vec q)\le O_{C,r}(pN).\label{eq:Hq-bound}
\end{equation}
Let $i_p(q)=q\log(q/p)+(1-q)\log((1-q)/(1-p))$, and note that if $q\ge Bp$ then $i_p(q)\ge q(\log B-1)$. So, using \cref{eq:Dq-bound} we have
\begin{equation}
    \sum_{v:q_v\ge Bp}q_v\le \sum_{v:q_v\ge Bp}\frac{i_p(q_v)}{\log B-1}\le \frac{\on{D}_{\mr{KL}}(\vec{q}\,\|\,p)}{\log B-1}\le \frac m2\label{eq:truncation-few}
\end{equation} if $B$ is sufficiently large in terms of $C,r$.

Then, let
\[h_v=\frac{\partial}{\partial q_v}\mc H(\vec q)=\sum_{\substack{e\in \mc H:\\e\owns v}}\prod_{u\in e\setminus v}q_u.\]
Using \cref{eq:truncation-few} and \cref{eq:Hq-bound}, we have
\[\left(\sum_{v:q_v< Bp}q_vh_v\right)/\left(\sum_{v:q_v< Bp}q_v\right)\le \frac 2m\sum_{v\in V}q_vh_v=\frac{2r}m\mc H(\vec q)=O_{C,r}(1).\]
(We can assume the left-hand side is not dividing by zero, or we are already done).

So, there is $w\in V$ such that $q_w<Bp$ and $h_w=O_{C,r}(1)$. Now, suppose for the purpose of contradiction there is $v$ such that $q_v>Ap$. If we were to move some weight from $v$ to $w$, the derivative of the effect on $\on{D}_{\mr{KL}}(\vec{q}\,\|\,p)+\Phi_L(\mc H(\vec q))$ would be
\begin{align*}
    &\bigg(\frac{\partial}{\partial q_w}i_p(q_w)+\theta h_w\bigg)-\bigg(\frac{\partial}{\partial q_v}i_p(q_v)+\theta h_v\bigg)\\&\qquad=\big(\log q_w-\log(1-q_w)\big)-\big(\log q_v-\log(1-q_v)\big)+ \theta(h_w-h_v)\\&\qquad\le \log B-\log A+O_C(B)
    \end{align*}
where $\theta=\Phi_L'(\mc H(\vec q))\in [0,1]$. This derivative is negative for large $A$, which contradicts the minimality of $\vec q$.
\end{proof}
Now, we deduce \cref{thm:main-lower,thm:main-m-lower}. Actually, it is convenient to deduce \cref{thm:main-lower} from \cref{thm:main-m-lower}.
\begin{proof}[Proof of \cref{thm:main-m-lower}]
Let $A$ be as in \cref{lem:truncation-lower} and let $\vec q$ be a minimiser of $\on{D}_{\mr{KL}}(\vec{q}\,\|\,m/N)+\Phi_L(\mc H(\vec q))$, over all $\vec q=(q_v)_{v\in V}\in [0,Am/N]^V$ summing to $m$. Let $\eta>0$ be small (in terms of $A,C,r$), and let $\mc P^0$ be a partition of $V$ into (at most) $\lceil A/\eta\rceil+1$ parts, where we group vertices $v\in V$ according to which integer multiple of $\eta m/N$ is closest to $q_v$ (breaking ties arbitrarily). Letting $\delta=\eta^2/2$, there are at most $A\eta N$ vertices in parts of size less than $\delta N$. Let $\mc P$ be obtained by merging these vertices into a single part, and then merging this new part with some part of size at least $\delta N$. So, all parts of $\mc P$ have size at least $\delta N$, and the number of parts is at most $1/\delta$.

Let $\vec x^{\,0}=(x_v)_{v\in V}\in [0,Am/N]^V$ be obtained by averaging $q_v$ over each part $W\in \mc P$. By the choice of $\mc P$ (via $\mc P^0$) we have $\|\vec x^{\,0}-\vec q\|_1\le \eta m+(Am/N)(A\eta N)=O_A(\eta m)$. Let $\vec x$ be obtained by slightly altering the entries of $\vec x^{\,0}$ to ensure that the sum of the entries of $\vec x$ over each part $W\in \mc P$ is an integer $m_W\in \mb N_{\ge 0}$ (while still ensuring that the sum of $\vec x$ is exactly $m$, and that $\vec x$ is constant on every part). We can do this in such a way that $\|\vec x^{\,0}-\vec x\|_1\le 1/\delta$. By \cref{thm:prescribed-parts-lower} (with say $\alpha=3Am/N$, and adjusted values of $C,\nu$) we have
\[\log \Pr\big[|\mc H[R]|\le L\big]\ge\log\left(\frac{\prod_{W\in \mc P}\binom{|W|}{m_W}}{\binom N m}\right)-\Phi_L(\mc H (\vec x))
-\eta \alpha N.\]
Using \cref{fact:LO} (and using the notation therein), we compute
\begin{align*}\log\left(\frac{\prod_{W\in \mc P}\binom{|W|}{m_W}}{\binom N m}\right)&=-\on{D}_{\mr{KL}}(\vec{x}\,\|\,m/N)+\sum_{W\in\mc P}\log\Pr[B_{m_W/|W|,|W|}=m_W]-\log\Pr[B_{m/N,N}=m]\\
&=-\on{D}_{\mr{KL}}(\vec{x}\,\|\,m/N)+O_\delta(\log N)=-\on{D}_{\mr{KL}}(\vec{x}^{\,0}\,\|\,m/N)+O_\delta(\log N)\\
&\ge -\on{D}_{\mr{KL}}(\vec{x}^{\,0}\,\|\,m/N)-(\varepsilon/2)m,\end{align*}
and by convexity we have $\on{D}_{\mr{KL}}(\vec{x}^{\,0}\,\|\,m/N)\le \on{D}_{\mr{KL}}(\vec{q}\,\|\,m/N)$. Then, using $(N^{-\nu},C;m/N)$-degree-boundedness we have $\mc H(\vec x)=\mc H(\vec q)+O_{A}(\eta m)$, so by \cref{fact:Phi}(2) we have $\Phi_L(\mc H(\vec x))=\Phi_L(\mc H(\vec q))+O_{A}(\eta m)$. The desired result follows, taking $\eta$ sufficiently small in terms of $\varepsilon$.
\end{proof}

\begin{proof}[Proof of \cref{thm:main-lower}]
Let $\vec q\in [0,p]^V$ be a minimiser of $\on{D}_{\mr{KL}}(\vec{q}\,\|\,p)+\Phi_L(\mc H(\vec q))$ and recall the definition of $i_p$ from the proof of \cref{lem:truncation-lower}. First, we consider the (easy) case where $\vec q$ has sum of entries $m_{\vec q}\le \delta pN$, for some $\delta>0$ small in terms of $\varepsilon$. Note that
\[\log \Pr\big[|\mc H[R]|\le L\big]\ge \log\Pr[R=\emptyset]=N\log(1-p)=-pN+o(pN),\]
so by convexity we have
\[\on{D}_{\mr{KL}}(\vec{q}\,\|\,p)\ge Ni_p(m_{\vec q}/N)\ge N\cdot \big(\delta p\log \delta+(1-\delta)p+O(p^2)\big)\ge -\log \Pr\big[|\mc H[R]|\le L\big]-\varepsilon p N.\] 

So for the rest of the proof we assume $\vec q$ has sum of entries at least $\delta pN$.
Let $\vec q_{\mr U}$ be obtained by slightly rescaling the entries of $\vec q$ (changing the total weight by at most 1) in such a way that the sum of entries of $\vec q_{\mr U}$ is an integer $m\ge \lfloor\delta pN\rfloor$. Recalling the notation in the proof of \cref{lem:truncation-lower}, and recalling \cref{fact:LO} (and the notation therein), we compute
\[\log \Pr\big[|R|=m\big]=-Ni_p(m/N)+\log \Pr\big[B_{m/N,N}=m\big]=-Ni_p(m/N)+O(\log N).\]
Note that the conditional distribution of $R$, given $|R|=m$, is precisely $\on{Unif}(V,m)$. So, by \cref{thm:main-m-lower} (adjusting the parameters $C,\nu,\varepsilon$) we have
\begin{align*}\log \Pr\big[|\mc H[R]|\le L\big]&\ge \log \Pr\big[|R|=m\big]+\log  \Pr\big[|\mc H[R]|\le L\,\big|\,|R|=m\big]\\
&\ge -Ni_p(m/N)-\on{D}_{\mr{KL}}(\vec{q}_{\mr U}\,\|\,m/N)-\Phi_L(\mc H(\vec q_{\mr U}))-(\varepsilon/2) m\\
&=-\on{D}_{\mr{KL}}(\vec{q}_{\mr U}\,\|\,p)-\Phi_L(\mc H(\vec q_{\mr U}))-(\varepsilon/2) m.\end{align*}
To conclude, we note that $\on{D}_{\mr{KL}}(\vec{q}\,\|\,p)=\on{D}_{\mr{KL}}(\vec{q}_{\mr U}\,\|\,p)+O(\log N)$ and $\mc H(\vec q)=\mc H(\vec q_{\mr U})+O_{C,r}(1)$ (using $(N^{-\nu},C;m/N)$-degree-boundedness).
\end{proof}

\section{Basic reductions for subgraph counts in random graphs}\label{sec:reductions}

In this section we explain how to deduce \cref{thm:main-graphs} and \cref{cor:main-graphs} from \cref{thm:main-upper,thm:main-m-upper,thm:main-lower,thm:main-m-lower}. Everything here is quite standard, so we are brief with details.
This section will involve quite a few integrals, so it is convenient to use shorthand notation of the form $\int W=\int_0^1\int_0^1W(x,y)\on d y\on dx$.

First, the deduction of \cref{thm:main-graphs} from \cref{thm:main-upper,thm:main-m-upper,thm:main-lower,thm:main-m-lower} is immediate, given the following proposition. See e.g.\ \cite[Lemma~7.1]{JPPSnew} for a proof.
\begin{proposition}
\label{prop:H-free}
For a strictly 2-balanced graph $H$ which is not the 2-edge path, let $\mathcal H_H$ be the $e(H)$-uniform hypergraph on $N=\binom n2$ vertices, which has a vertex for each edge of the complete graph $K_n$, and a hyperedge for every copy of $H$ in $K_n$.

Fix $\delta>0$ and suppose that $\delta n^{-1/d_{2}(H)}\le p\le \delta^{-1}n^{-1/d_{2}(H)}$. Then, there is $\nu_H>0$ and $C_{\delta,H}>0$ such that $\mc H_H$ is $(N^{-\nu_H},C_{\delta,H};p)$-degree-bounded.
\end{proposition}

Next we deduce \cref{cor:main-graphs} from \cref{thm:main-graphs}. For this, we need to show that the objective functions in \cref{thm:main-graphs} have approximately the same minimum values as the objective functions in \cref{cor:main-graphs}. We will see this via quite straightforward correspondences between $\mc Q$ and $\mc W$ and between $\mc Q_m$ and $\tilde{\mc W}$.

\begin{lemma}\label{lem:approximate-weighting-with-graphon}
Recall the notation from \cref{thm:main-graphs,cor:main-graphs}, and let $N=\binom n2$.
\begin{enumerate}
    \item In the setting of \cref{cor:main-graphs}(1): for any $\vec q\in \mc Q$ there is $W\in \mc W$ such that
\begin{align*}
        &p\binom n2\cdot \bigg(
        \int(W\log W-W+1)+\frac{2c^{e(H)-1}}{|\on{Aut}(H)|}\psi_z(t_H(W))\bigg) \\
        &\qquad\qquad=   \on{D}_{\mr{KL}}(\vec{q}\,\|\,p)+\Phi_L(X_H(\vec q))+o(pN).
    \end{align*}
    \item In the setting of \cref{cor:main-graphs}(2): for any $\vec q\in \mc Q_m$ satisfying $\|\vec q\|_\infty =O(m/N)$ there is $W\in \tilde{\mc W}$ such that
\[
        m\cdot \bigg(
        \int W\log W+\frac{2c^{e(H)-1}}{|\on{Aut}(H)|}\psi_z(t_H(W))\bigg)= \on{D}_{\mr{KL}}(\vec{q}\,\|\,m/N)+\Phi_L(X_H(\vec q))+o(m).
    \]
\end{enumerate}
\end{lemma}

\begin{lemma}\label{lem:approximate-graphon-with-weighting}
Recall the notation from \cref{thm:main-graphs,cor:main-graphs}, and let $N=\binom n2$.
\begin{enumerate}
    \item In the setting of \cref{cor:main-graphs}(1): for any $W\in \mc W$ there is $\vec q\in \mc Q$ such that
\begin{align*}
        &\on{D}_{\mr{KL}}(\vec{q}\,\|\,p)+\Phi_L(X_H(\vec q)) \\
        &\qquad\qquad=   p\binom n2\cdot \bigg(
        \int(W\log W-W+1)+\frac{2c^{e(H)-1}}{|\on{Aut}(H)|}\psi_z(t_H(W))\bigg)+o(pN).
    \end{align*}
    \item In the setting of \cref{cor:main-graphs}(2): for any $W\in \tilde{\mc W}$ satisfying $\int W\log W<\infty$ and $t_H(W)<\infty$,
    there is $\vec q\in \mc Q_m$ such that
\[
        \on{D}_{\mr{KL}}(\vec{q}\,\|\,m/N)+\Phi_L(X_H(\vec q))= m\cdot \bigg(
        \int W\log W+\frac{2c^{e(H)-1}}{|\on{Aut}(H)|}\psi_z(t_H(W))\bigg)+o_W(m).
    \]
\end{enumerate}
\end{lemma}
(In \cref{lem:approximate-graphon-with-weighting}(2), the notation $o_W(1)$ indicates a term that tends to zero at a rate which may depend on $W$.)

Before we proceed with the proofs of \cref{lem:approximate-graphon-with-weighting,lem:approximate-weighting-with-graphon}, note that these lemmas, together with \cref{thm:main-graphs}, directly imply \cref{cor:main-graphs}, using \cref{lem:truncation-lower} to see that any optimiser of the $\mb G(n,m)$ variational problem in \cref{thm:main-graphs} must satisfy $\|\vec q\|_\infty \le O(m/N)$.

\begin{proof}[Proof of \cref{lem:approximate-weighting-with-graphon}]
Split $[0,1]$ into $n$ equal intervals $I_{1},\dots,I_{n}$.

    For (1), we take
\[W(x,y)=\begin{cases}
    q_{ij}/p&\text{if }(x,y)\in I_i\times I_j\text{ for }i\ne j,\\
    0&\text{otherwise.}
\end{cases}\]

Using the estimate $(1-q_{ij})\log ((1-q_{ij})/(1-p))=-q_{ij}+p+o(p)$ (which holds whenever $q_{ij}\le p=o(1)$), it is easy to see that 
\[p\binom n2\int(W\log W-W+1)= \on{D}_{\mr{KL}}(\vec{q}\,\|\,p)+o(pN)+O(pN/n)\]
(note that zeroing out the diagonal blocks affects the integral by at most $1/n$). Also, note that
\[p\binom n2\cdot \frac{2c^{e(H)-1}}{|\on{Aut}(H)|}=\mb EX_H(G)+o(pN),\quad t_H(W)=X_H(\vec q)/\mb EX_H(G)+O(1/n).\]
Recalling that $\psi_z(t)=t\phi(z/t)$, the desired result then follows by the continuity of $\phi$.

Then, for (2) we need to be slightly more careful with the definition of $W$ to ensure that its integral is exactly 1:
\[W(x,y)=\begin{cases}
    q_{ij}\cdot (n^2/(2m))&\text{if }(x,y)\in I_i\times I_j\text{ for }i\ne j,\\
    0&\text{otherwise.}
\end{cases}\]
Then, we can proceed using basically the same estimates as for (1), with $m/N$ in place of $p$. We need the assumption $\|\vec q\|_\infty =O(m/N)$ to ensure that $(1-q_{ij})\log ((1-q_{ij})/(1-p))=-q_{ij}+p+o(p)$ for all $ij\in K_n$.
\end{proof}

\begin{proof}[Proof of \cref{lem:approximate-graphon-with-weighting}]
    Let $X_{1},\dots,X_{n}$ be i.i.d.\ samples from the uniform
distribution on $[0,1]$.

For (1), we consider the \emph{random} weighting $\vec q=(q_{ij})_{ij\in K_n}\in \mc Q$ defined by $q_{ij}=pW(X_{i},X_{j})$. 
Using similar estimates as for \cref{lem:approximate-weighting-with-graphon}, we can compute
\[\mb E\on{D}_{\mr{KL}}(\vec{q}\,\|\,p)=p\binom n2\int(W\log W-W+1)+o(pN)\]
and
\[\mb EX_H(\vec{q})=\mb EX_H(G)\cdot t_H(W)=p\binom n2\cdot \frac{2c^{e(H)-1}}{|\on{Aut}(H)|}\cdot t_H(W)+o(pN).\]
We then observe that $\on{D}_{\mr{KL}}(\vec{q}\,\|\,p)$ and $X_H(\vec{q})$ are both concentrated around their expected values: with probability $1-o(1)$ they both differ from their expected value by $o(pN)$. For example, one way to see this is using McDiarmid's bounded-difference inequality, noting that changing a single $X_i$ can affect $\on{D}_{\mr{KL}}(\vec{q}\,\|\,p)$ and $X_H(\vec{q})$ by at most $O(np)$ (this uses that $p$ is in the critical regime, and $W\le 1$).

So, we can fix an outcome of $\vec q$ simultaneously satisfying $\on{D}_{\mr{KL}}(\vec{q}\,\|\,p)=p\binom n2\int(W\log W-W+1)+o(pN)$ and $X_H(\vec{q})=\mb EX_H(G)\cdot t_H(W)+o(pN)$. The desired result then follows by the continuity of $\phi$.

For (2), we cannot \emph{a priori} assume that $W$ is bounded. Instead, we truncate $W$: define $W_{M}:[0,1]^{2}\to[0,M]$ by $W_{M}(\vec{x})=\min(M,W(\vec{x}))$.

We also need to make sure $\vec q$ is normalised appropriately. For this, it is convenient to introduce some notation. For a function $f:K_n\to \mb R$, we introduce the notation $\langle f\rangle =N^{-1}\sum_{ij\in K_n}f(ij)$. Let $w_0:[n]^2\to [0,M]$ be the random function defined by $w_0(i,j)=W_M(X_i,X_j)$ for $i\ne j$, and $w_0(i,i)=0$ for all $i$. Let $w_1=\langle w_0\rangle^{-1} w_0$ be a normalised version of $w_0$, and let $\vec q=(q_{ij})_{ij\in K_n}$ be defined by $q_{ij}=(m/N)w_1(i,j)$, so $\vec q$ has sum of entries exactly $m$.

Using a concentration argument as in (1), sending $n\to \infty$ for fixed $M$, we can prove that with probability $1-o(1)$ we have
\[\langle w_0\rangle=\int W_M+o_M(1),\quad \langle w_0\log w_0\rangle=\int W_M\log W_M+o_M(1)\]and, writing $\langle f(x_1,\dots,x_k)\rangle$ for the average of $f(x_1,\dots,x_k)$ over distinct $x_1,\dots,x_k\in [n]$,
\[\bigg\langle\prod_{ij\in H}w_0(x_i,x_j)\bigg\rangle = t_H(W_M)+o_M(1).\]
On the other hand, we have
\[\int W_{M}\to 1,\quad \int W_{M}\log W_M\to  \int W\log W,\quad t_H(W_M)\to t_H(W),\] as $M\to \infty$, by the dominated convergence theorem.
So, taking $M\to \infty$ very slowly as $n\to \infty$, we see that there is an outcome of $w_0$ satisfying
\[\langle w_0\rangle=1+o_W(1),\quad \langle w_1\log w_1\rangle=\langle w_0\rangle^{-1} \langle w_0\log w_0\rangle-\log\langle w_0\rangle=\int W_{M}\log W_M+o_W(1),\]
and
\[\bigg\langle\prod_{ij\in H}w_1(x_i,x_j)\bigg\rangle=\langle w_0\rangle^{-e(H)} \bigg\langle\prod_{ij\in H}w_0(x_i,x_j)\bigg\rangle=t_H(W)+o_W(1).\]
Then, it is not hard to see that the resulting weighting $\vec q$ satisfies the desired result.
\end{proof}

\section{Solving variational problems for subgraph counts in random graphs}\label{sec:variational}
In this section we prove \cref{prop:sidorenko} and \cref{thm:triangles,thm:H-early,thm:H-late}. 
Throughout the section, we use shorthand notation of the form $\int_{x,y}W(x,y)=\int_0^1\int_0^1W(x,y)\on d y\on dx$, or, even more concisely, of the form
$\int W=\int_0^1\int_0^1W(x,y)\on d y\on dx$.
Also, we let $g:[0,1]\to \mb R$ be the convex function $t\mapsto t\log t-t+1$.

\subsection{Graphs with the Sidorenko property}
First we give the (easy) proof of \cref{prop:sidorenko}.

\begin{proof}[Proof of \cref{prop:sidorenko}]
     Let $\gamma={2c^{e(H)-1}}/{|\on{Aut}(H)|}$. 
For any $W\in \mc W$ with mean value $\bar w=\int W
$, we have $\int g(W)
\ge g(\bar w)$ by convexity, and $t_H(W)\ge t_H(\bar w)$ by the Sidorenko property (where we abusively identify $\bar w$ with the constant graphon taking value $\bar w$). So, $\int g(W)+\gamma t_H(W)\ge g(\bar w)+\gamma t_H(\bar w)$. This immediately implies the desired formula for $\Psi_{\mr{Bin}}(H,c,0)$.

For $\Psi_{\mr{Unif}}(H,c,0)$, we need a truncation and rescaling argument, since the Sidorenko property only (directly) considers functions $[0,1]^2\to [0,1]$. Consider any $W\in \tilde {\mc W}$ and (as in \cref{sec:reductions}) let $W_M=\min(M,W)$. Then
\begin{align*}t_H(W)=\lim_{M\to \infty}t_H(W_M)&=\lim_{M\to \infty}M^{e(H)}t_H(W_M/M)\\&\ge \lim_{M\to \infty}M^{e(H)}\Bigg(\int W_M/M\Bigg)^{e(H)}=\lim_{M\to \infty}\Bigg(\int W_M\Bigg)^{e(H)}=1,\end{align*} and $\int W\log W=\int g(W)\ge g(1)=0$ by Jensen's inequality. The desired formula for $\Psi_{\mr{Unif}}(H,c,0)$ follows.
\end{proof}

\subsection{General graphs in the early critical regime}
In this subsection we prove \cref{thm:H-early}. Fix a graph $H$ with $v$ vertices, $h$ edges and $a$ automorphisms. For $W\in \mc W$ and $\gamma\ge 0$, let
\[
    \Psi_{\gamma}(W) = \int_{x_1,x_2}g(W(x_1,x_2))+\gamma\int_{x_1,\dots,x_{v(H)}} \prod_{ij\in H}W(x_i,x_j).
\]
So, recalling the definitions in \cref{cor:main-graphs} and the definition of $g$ at the start of this section, we have
\[
    \Psi_{\mathrm{Bin}}(H,c,0) = \min_{W\in \mc W} \Psi_{2c^{h-1}/a}(W).
\]
Let \[\Psi_\gamma^{\on{opt}}=\min_{s\in [0,1]}({s\log s-s+1+\gamma s^{h}}),\]
so our goal is to show that for $\gamma\in (0,e/h]$ we have $\min_{W\in \mc W} \Psi_{\gamma}(W)\ge \Psi_\gamma^{\on{opt}}$.

The starting point in our proof of \cref{thm:H-early} is as follows. For $W\in \mc W$, let $(X_1,\dots,X_{v(H)})$ be a uniformly random point in $[0,1]^{v(H)}$, let $i_1j_1,\dots,i_hj_h$ be an enumeration of the edges of $H$, and let $W_{\ell}=W(X_{i_{\ell}},X_{j_\ell})$. So,
\begin{align*}\Psi_{\gamma}(W)&=\mb E[g(W(X_1,X_2))+\gamma \prod_{ij\in H}W(X_i,X_j)]\\
&=\mb E\bigg[\frac{g(W_1)+\dots+g(W_h)}{h}+\gamma W_1\cdot \dots\cdot W_h\bigg].
\end{align*}
So, to prove \cref{thm:H-early}, it suffices to prove the following proposition.
\begin{proposition}\label{prop:calculus}
    For any $0< \gamma \le e/h$ and $w_1,\dots,w_h\in [0,1]$, we have
    \[
        \frac{g(w_1)+\dots+g(w_h)}{h}+\gamma w_1\cdot \dots\cdot w_h\ge \Psi_{\gamma}^{\mr{opt}}.
    \]
\end{proposition}
\begin{proof}
Let $F(w_1,\dots,w_h)=(g(w_1)+\dots+g(w_h))/{h}+\gamma w_1\cdot \dots\cdot w_h$. Our task is to minimise $F(w_1,\dots,w_h)$ over all $w_1,\dots,w_h\in[0,1]$.
We will show that $F$ has a unique minimiser of the form $(w,\dots,w)$, which implies that the minimum value of $F$ is at least $\Psi_{\gamma}^{\mr{opt}}$.

First, we rule out minimisers on the boundary. If $w_i = 0$ then increasing $w_i$ by $\varepsilon$ decreases $g(w_i)$ by $\varepsilon\log (1/\varepsilon)+O(\varepsilon)$, while increasing $ w_1\cdot \dots\cdot w_h$ by $O(\varepsilon)$. Then, assuming each $w_j>0$: if $w_i=1$ then decreasing $w_i$ by $\varepsilon$ increases $g(w_i)$ by $O(\varepsilon^2)$ while decreasing $w_1\cdot \dots\cdot w_h$ by $\Omega(\varepsilon)$.

We deduce that every minimiser $(w_1,\dots,w_h)$ of $F$ lies in $(0,1)^h$ and is a stationary point of $F$, i.e., it satisfies 
\[\nabla F(w_1,\dots,w_h)=\bigg(\frac{\log w_i}h+\gamma \prod_{j\ne i}w_j\bigg)_{\!i}=\vec 0.\]

Now, fix a particular stationary point $(w_1^*,\dots,w_h^*)\in (0,1)^h$. We will show that if $(w_1^*,\dots,w_h^*)$ minimises $F$, then $w_1^*=\dots=w_h^*$. To this end, fix any two distinct indices $i,j$ and let $Q=h\gamma\prod_{\ell\notin\{i,j\}}w_\ell^*\le h\gamma\le e$ (using our assumption on $\gamma$). Stationarity tells us that
\[\log w_i^*=-Qw_j^*,\qquad \log w_j^*=-Qw_i^*.\]
Substituting these equations into each other, we see that both $w_i^*$ and $w_j^*$ are zeroes of the function 
\[f_Q(x)=\log x+Qe^{-Qx}.\]
Then, observe that for all $x>0$ we have \[f'_Q(x)=\frac{1-Q(Qx)e^{-Qx}}{x}\ge \frac{1-Q/e}{x}\ge 0,\] with equality only at $x=1/e$. That is to say, $f_Q$ is strictly increasing, so has at most one zero, meaning $w_i^*=w_j^*$. Since $i,j$ were arbitrary, this shows that $w_1^*=\dots=w_h^*$, as desired.
\end{proof}

\subsection{Triangles}
In this subsection we explain how to prove \cref{thm:triangles}.
For $W\in \mc W$ and $\gamma\ge 0$, let
\[
    \Psi_{\gamma}(W) = \int_{x,y}g(W(x,y))+\gamma\int_{x,y,z} W(x,y)W(y,z)W(x,z).
\]
So, recalling the definitions in \cref{cor:main-graphs} and the definition of $g$ at the start of this section, we have
\[
    \Psi_{\mathrm{Bin}}(K_3,c,0) = \min_{W\in \mc W} \Psi_{c^2/3}(W).
\]

Let \[\Psi_\gamma^{\on{opt}}=\min_{r,t\in [0,1]}\Bigg(\frac{g(r)+g(t)}2+\gamma\bigg(\frac14r^3+\frac34rt^2\bigg)\Bigg),\]
so our goal is to show that $\min_{W\in \mc W} \Psi_{\gamma}(W)\ge \Psi_\gamma^{\on{opt}}$ for all $\gamma\ge 0$.

Just as for \cref{thm:H-early}, the starting point in our proof of \cref{thm:triangles} is as follows. For $W\in \mc W$, let $(X,Y,Z)$ be a uniformly random point in the cube $[0,1]^3$, and let $(A,B,C)=(W(X,Y),W(X,Z),W(Y,Z))$. So, 
\begin{align*}\Psi_{\gamma}(W)&=\mb E[g(A)+\gamma ABC]\\
&=\mb E\bigg[\frac{g(A)+g(B)+g(C)}{3}+\gamma ABC\bigg].
\end{align*}
So, if we could prove that 
\begin{equation}
    \min_{(a,b,c)\in [0,1]^3}\bigg(\frac{g(a)+g(b)+g(c)}3+\gamma abc\bigg)\ge \Psi_{\gamma}^{\mr{opt}},\label{eq:cheap-variational-relaxation} 
\end{equation}
we would be done.

As we saw in \cref{prop:calculus}, this inequality does hold for sufficiently small $\gamma$ (we proved that it holds for $\gamma\le e/3$, but really it holds as long as $\gamma$ is less than about 1.529, which corresponds to $c$ being less than about 2.142). For larger $\gamma$ we need to use a bit more information about $A,B,C$. Namely, we use the ``positive semidefiniteness'' property\footnote{This is a ``rank 1'' positive semidefiniteness property. We could have considered similar inequalities involving more than one function $G$, but this turned out not to be necessary.} that for any bounded measurable function $G:[0,1]\to \mb R$, we have

\begin{equation}
    \mb E [G(A)G(B)] = \mb E G(W(X,Y))G(W(X,Z)) = \mb E\Big[\mb E\big[ G(W(X,Y))\,\big|\,X\big]^2\Big] \ge 0.\label{eq:variational-new-constraint}
\end{equation}
(and similarly $\mb E [G(A)G(C)]\ge 0$ and $\mb E [G(B)G(C)]\ge 0$). With this in mind, it suffices to prove the following proposition.
\begin{proposition}\label{prop:numerical}
    For all $\gamma> 0$, there exists a bounded measurable function $G_{\gamma}:[0,1]\to \mathbb{R}$ such that
\[
    \inf_{(a,b,c)\in [0,1]^3}\bigg(\frac{g(a)+g(b)+g(c)}3+\gamma abc-G_{\gamma}(a)G_{\gamma}(b)-G_{\gamma}(a)G_{\gamma}(c)-G_{\gamma}(b)G_{\gamma}(c)\bigg)\ge \Psi_{\gamma}^{\mr{opt}}.
\]
\end{proposition}
One way to prove \cref{prop:numerical} is as follows. First, guess the form of the ``certificate function'' $G_\gamma$ in some kind of parameterised way (informed by the places we know the inequality must be sharp), and then perform some kind of numerical optimisation procedure to optimise the parameters to obtain a suitable guess for $G_\gamma$. Then, to prove that this $G_\gamma$ actually satisfies the desired inequality for all $a,b,c\in [0,1]$ (which we can write as $\mc R(a,b,c)\ge 0$ for an appropriate $\mc R$), one can proceed by a computational grid search using interval arithmetic. Namely, one can rigorously compute a tiny interval containing $\mc R(a,b,c)$ for a large finite number of points $(a,b,c)\in [0,1]^3$, and use a crude bound on the gradient $\nabla\mc R(a,b,c)$ to see that $\mc R(a,b,c)\ge 0$ except possibly in a small neighbourhood close to those (finitely many) $(a,b,c)$ for which the inequality is sharp. In each such neighbourhood one can compute a Hessian to see that these sharp points are local minima. We remark that this type of computational verification of an inequality is definitely not novel to this paper; it is the domain of a field called \emph{validated numerics}, which is important in various parts of mathematics and engineering (see for example \cite{Tuc11} for an introductory textbook on the subject). For example, much more sophisticated versions of this strategy appeared in Hales' proof of Kepler's conjecture~\cite{Hal05,Kepler2} and Schwarz' resolution of the 5-Electron case of Thomson's Problem~\cite{Sch13}.

As discussed in \cref{sec:AI}, we used AI tools to execute the above strategy for us. Unfortunately, the implementation is rather complicated and, not being experts in validated numerics, we were not able to extract much insight out of it. So, we say nothing more about the proof of \cref{prop:numerical} here, and defer to a Lean certificate\footnote{As mentioned in \cref{sec:AI}, Lean is not particularly well-optimised for this type of thing. So in order for our Lean proof to be reasonably short and run in a reasonable amount of time on a standard personal computer, the actual Lean proof is more sophisticated than the approach described above (it does some branch-and-bound to partition space efficiently, instead of a na\"ive grid search).} which can be found at \url{https://palomar-registry.org/entry?id=PALOMAR-2026-10-19-000018}. The definitions of the functions $G_\gamma$ can be found in \cref{app:numerical}.
\begin{remark}
In some sense we ``got lucky'' that constraints of the type in \cref{eq:variational-new-constraint} were enough to prove \cref{thm:triangles} for all $c>0$; we are not aware of any conceptual reason why this should have been the case. It would be interesting to investigate this further: is there some kind of systematic computational framework which is (reasonably feasible in practice and) always enough to solve the variational problems in \cref{cor:main-graphs}? Of course, it would be even more desirable to find a way to avoid heavy computational work altogether!
\end{remark}

\subsection{Edge-critical graphs in the late critical regime}\label{subsec:H-late}
In this subsection we prove \cref{thm:H-late}. Fix a strictly 2-balanced edge-critical graph $H$ with chromatic number $k+1$, and for $W\in \mc W$ and $\gamma\ge 0$, let
\[
    \Psi_{\gamma}(W) = \int_{x_1,x_2}g(W(x_1,x_2))+\gamma t_H(W).
\]
So, recalling the definitions in \cref{cor:main-graphs} and the definition of $g$ at the start of this section, we have
\[
    \Psi_{\mathrm{Bin}}(H,c,0) = \min_{W\in \mc W} \Psi_{\gamma}(W),
\]
where $\gamma=2c^{e(H)-1}/|\on{Aut}(H)|$.
Our goal is to show that if $\gamma$ is sufficiently large then $\Psi_{\gamma}$ is minimised by a function of the form $W_{r,t}^{(k)}$.

Just for this subsection, it will be convenient to use the language of graphons, and some of the corresponding theory (see for example \cite{Lov12}). The plan is as follows. First we use the Erd\H os--Stone--Simonovits theorem to show that if $W\in\mc W$ minimises $\Psi_\gamma$, then $W$ is close in the cut metric to the balanced complete $k$-partite graphon. Then we use local minimality of $W$ to show that in fact $W$ is close in $L_\infty$-norm to a complete $k$-partite graphon. Finally, we finish the proof with a convexity argument.

The following lemma encapsulates the first of these steps. Denote by $\on d_\square(G,G')$ the cut distance between two graphons $G,G'$, and let $G_k$ be the balanced $k$-partite graphon.

\begin{lemma}\label{lem:weak-stability}
    For any $\delta>0$ there is $\gamma(\delta)>0$ such that for $\gamma\ge \gamma(\delta)$, any minimiser $W_\gamma$ of $\Psi_\gamma$ satisfies $\on d_\square(W_\gamma,G_k)\le \delta$.
\end{lemma}
\begin{proof}
Suppose the conclusion of the lemma were false. Then, there would be a sequence of values of $\gamma$, tending to $\infty$, and a corresponding sequence of $\Psi_\gamma$-minimisers $W_\gamma$, for which $\on d_\square(W_\gamma,G_k)> \delta$. By compactness, there would be an infinite subsequence along which $W_\gamma$ converges (in the cut metric) to a limit $W_*$ satisfying $\on d_\square (W_*,G_k)\ge \delta$. We will show this is impossible: for any sequence of values of $\gamma$ tending to $\infty$, and any corresponding sequence of $\Psi_\gamma$-minimisers $W_\gamma$ tending to a limit $W_*$, we must have $\on d_\square(W_*,G_k)=0$.

To see this, first note that since $W_\gamma$ minimises $\Psi_\gamma$, we must have 
\[\int g(W_\gamma)+\gamma t_H(W_\gamma)= \Psi_\gamma(W_\gamma)\le \Psi_\gamma(G_k)=1/k. \]
This implies that $t_H(W_\gamma)\le 1/(k\gamma)\to 0$, so $t_H(W_*)=0$. It also implies that 
\begin{equation}
    \int g(W_*)\le 1/k\label{eq:gW*}
\end{equation}
(by lower-semicontinuity of $W\mapsto \int g(W)$).

Since $t_H(W_*)=0$, the stability version of the Erd\H os--Stone--Simonovits theorem~\cite{Sim68} tells us that the support graphon $\one_{W_*>0}$ satisfies
\[\int (1-g(W_*))\le \int \one_{W_*>0}\le 1-1/k,\] with equality only when $W_*=\one_{W_*>0}$ almost everywhere\footnote{i.e., for all but a measure-zero set} and $\on d_\square(W_*,G_k)=0$. Recalling \cref{eq:gW*}, this means that $\int g(W_*)= 1/k$ and $\on d_\square(W_*,G_k)=0$.
\end{proof}

Next, we use the local optimality of any $\Psi_{\gamma}$-minimiser $W$,
to boost the cut metric approximation from \cref{lem:weak-stability} to a strong $L_{\infty}$ approximation.

\begin{lemma}
\label{lem:strong-stability}For every $\Psi_{\gamma}$-minimiser
$W$ there is a partition $I_1\cup \dots\cup I_k$ such that:
\begin{enumerate}
\item each part $I_i$ has measure $1/k+o(1/k)$.
    \item $W(x,y)\le \exp(-\Omega_H(\gamma))$
    for almost every $x,y$ in the same part $I_i$,
    \item $W(x,y)\ge 1-\exp(-\Omega_H(\gamma))$ for almost every $x,y$ in two different parts $i,j$,
\end{enumerate}
where asymptotics are as $\gamma\to \infty$.
\end{lemma}

To prove \cref{lem:strong-stability} we need some further lemmas. First, we need the basic
fact that a $\Psi_\gamma$-minimiser is a ``stationary point'' of
$\Psi_\gamma$. To state this, we need some notation: for distinct $a,b\in \on V(H)$ let
\[R_{a,b}^W(x,y)=\int_{x_i:i\in \on V(H)\setminus \{a,b\}}\,\prod_{ij\in H\setminus ab}W(x_i,x_j),\quad \text{where }x_a=x\text{ and }x_b=y.\]

\begin{lemma}\label{lem:EL}
If $W$ is a minimiser of $\Psi_\gamma$ then for almost all $x,y$ we have
\[
   \log W(x,y)+\frac\gamma2\sum_{ab\in H} (R_{a,b}^W(x,y)+R_{b,a}^W(x,y))=0.
\]
\end{lemma}
\begin{proof}
Let $\Delta_H W=\log W+(\gamma/2)\sum_{ab\in H} (R_{a,b}^W+R_{b,a}^W)$. Noting that $g'(w)=\log w$ and that $W$ is constrained to be symmetric, this is the functional derivative of $\Psi_\gamma$ at $W$; in other words, the first variation of $\Psi_\gamma(W)$ with respect to a symmetric perturbation $U:[0,1]^2\to \mb R$ is $\int (\Delta_H W)U$. For an interior minimiser $W$, the functional derivative $\Delta_H W$ must be zero almost everywhere, so to prove the desired result, we just need to rule out minimisers on the boundary.

To this end, first, note that if $W = 0$ on a set of positive measure, then increasing $W$ by $\varepsilon$ on this set decreases $g(W)$ by $\varepsilon\log (1/\varepsilon)+O(\varepsilon)$, while increasing $t_H(W)$ by $O(\varepsilon)$. So, for any $\Psi_\gamma$-minimiser $W$ we have $W>0$ almost everywhere. Second, if $W=1$ on a set of positive measure, then decreasing $W$ by $\varepsilon$ on this set increases $g(W)$ by $O(\varepsilon^2)$ while decreasing $t_H(W)$ by $\Omega(\varepsilon)$.
\end{proof}

Now we prove \cref{lem:strong-stability}.

\begin{proof}[Proof of \cref{lem:strong-stability}]
In this proof, asymptotic notation is as $\gamma\to \infty$. For a function $f:[0,1]\to \mb R$, we write ``$f=o(1)$'' to indicate that $\|f\|_1=o(1)$.

Let $W$ be a minimiser of $\Psi_\gamma$. By \cref{lem:weak-stability}, we have $\on d_\square(W,G_k)=o(1)$. So, there is a realisation of the balanced complete $k$-partite graphon $G_k$ via some partition $I_1^0\cup \dots\cup I_k^0=[0,1]$ (where each $I_j^0$ has measure $1/k$) such that
\[\|W-G_k\|_1=\sum_{i,j}\Bigg|\int_{I_i^0\times I_j^0}(W(x,y)-G_k(x,y))\on dx\on dy\Bigg|=o(1).\]

We also record some further inequalities: by continuity of $g$ we have $\int g(W)=\int g(G_k)+o(1)=1/k+o(1)$, and by minimality of $W$ we have $\Psi_\gamma(W)\le \Psi_\gamma(G_k)$ so \begin{equation}
    \gamma t_H(W)=\Psi_\gamma(W)-\int g(W)\le \Psi_\gamma(G_k)-\frac1k+o(1)=o(1).\label{eq:small-hom}
\end{equation}

\medskip \noindent\textit{Step 1: Edge-criticality.}
We now explain how to use the edge-criticality of $H$. Since $H$ is edge-critical, there is an edge $ab\in H$ such that $H\setminus ab$ can be embedded in $G_k$ with $a,b$ in the same part $I_j^0$. So, for $x,y$ in the same part $I_j^0$, we have $R_{a,b}^{G_k}(x,y)=\Omega_H(1)$.

Since $\|W-G_k\|_1=o(1)$, we can obtain a similar conclusion with respect to $W$, as long as $x,y$ ``resemble two points in the same part of $G_k$''. Specifically, for any $x,y\in [0,1]$, if there is some $j$ and $\beta=\Omega_H(1)$ such that for all $\ell\ne j$ we have
\[\int_{I_\ell^0} \one_{W(x,z)\ge \beta}\on dz+\int_{I_\ell^0} \one_{W(y,z)\ge \beta}\on dz\ge 1/k+\Omega_H(1),\]
then for each $\ell\ne j$ there is significant overlap in the set of $z$ for which $W(x,z)$ is non-negligible and the set of $z$ for which $W(y,z)$ is non-negligible, and it is not hard to see that $R_{a,b}^{W}(x,y)=\Omega_H(1)$.
Recalling \cref{lem:EL}, we deduce that for such $x,y$ we have
\[
W(x,y)\le \exp\bigg(-\frac\gamma 2R_{a,b}^W(x,y)\bigg)=\exp(-\Omega_H(\gamma)).
\]

\medskip \noindent\textit{Step 2: Bounding the within-block densities.}
Since $\|W-G_k\|_1=o(1)$, 
there are ``typical'' sets $I_1^1\subseteq I_1^0,\dots,I_k^1\subseteq I_k^0$ of measure $1/k-o(1)$ such that 
\[
    W(x,\cdot)=G_k(x,\cdot)+o(1)\quad \text{for all }x\in I_j^1.
\]
By the considerations in Step 1, it is easy to see that $W(x,y)\le \exp(-\Omega_H(\gamma))$ when $x,y$ lie in the same typical set $I_j^1$. However, it is a bit more delicate to 
deal with the few exceptional $x$ which do not lie in any $I_j^1$. For this, we use a different type of local minimality property: minimality under the operation of ``removing vertices'' (in a certain continuum sense).

Specifically, letting say $\alpha_k=1/(10k)$ and $\beta_k=\alpha_k^2$, we claim that for almost all $x\in I_j^0$ we have
\begin{equation}
    \int_y \one_{W(x,y)>\beta_k}\ge 1-1/k-\alpha_k.\label{eq:zykov-claim}
\end{equation}
To see this, suppose for the purpose of contradiction that there is a measure-$\varepsilon$ set $S\subseteq [0,1]$ such that \cref{eq:zykov-claim} fails for all $x\in S$. Then, let $W_{\bar S}$ be obtained by ``removing $S$ and rescaling the domain accordingly''\footnote{For almost all our different types of techniques to solve variational problems, throughout this section, it doesn't really make any difference whether we consider the ``discrete'' variational problems in \cref{thm:main-graphs} or their continuum variants in \cref{cor:main-graphs}. This rescaling argument seems to be the main exception (in the discrete world, the natural approach would be some kind of \emph{Zykov symmetrisation}, see e.g.\ \cite[Lemma~1.2.7]{Zha23}).}: for example, if (without loss of generality) we had $S=[1-\varepsilon,1]$, we would take
$W_{\bar S}(x,y)=W((1-\varepsilon)x,(1-\varepsilon)y)$. Now, for $x\in S$ we have $\int_y g(W(x,y))\ge (1/k+\alpha_k)g(\beta_k)\ge 1/k+\Omega_k(1)$, so $\int g(W_{\bar S})\le (1-\varepsilon)^{-2}\big(\int g(W)-2\varepsilon(1/k+\Omega_k(1))+O(\varepsilon^2)\big)=\int g(W)-\Omega_k(\varepsilon)$ for small $\varepsilon$. On the other hand, $\gamma t_H(W_{\bar S})\le (1-\varepsilon)^{-v(H)}\gamma t_H(W)=\gamma t_H(W)+o(\varepsilon)$ by \cref{eq:small-hom}. For small $\varepsilon$, we therefore have $\Psi_\gamma(W_{\bar S})<\Psi_\gamma(W)$, contradicting the minimality of $W$.

Now, say that a vertex is ``$j$-small'' if $\int_{I_j^1}\one_{W(x,y)>\beta_k}\on dy< \alpha_k$. So, \cref{eq:zykov-claim} implies that almost all $x$ are $j$-small for at most one $j$.

In fact, \cref{eq:zykov-claim} and the considerations in Step~1 together imply that almost every $x$ is $j$-small for \emph{exactly} one $j$. Indeed, if $x$ is non-$j$-small for every $j$ then for any typical $y\in I_1^1\cup\dots\cup I_k^1$ we have 
$W(x,y)\le \exp(-\Omega_H(\gamma))=o(1)$ for any typical $y$, by the considerations in Step~1. Recalling that all but an $o(1)$-fraction of $y$ are typical, this contradicts \cref{eq:zykov-claim}.

Now, let $I_j$ be the set of $x$ which are $j$-small. We have established that $I_1,\dots,I_k$ form a partition of $[0,1]$ (ignoring measure-zero sets). Note that $I_j^1\subseteq I_j$ for each $j$, so each $I_j$ has measure $1/k+o(1)$. By \cref{eq:zykov-claim}, for almost every $x\in I_j$ we have $\int_{I_\ell}\one_{W(x,y)>\beta_k}\on dy\ge 1/k-2\alpha_k-o(1)$ for all $\ell\ne j$.
By the considerations in Step~1 we deduce
\begin{equation}
    W(x,y)\le \exp\bigg(-\frac\gamma 2R_{a,b}^W(x,y)\bigg)=\exp(-\Omega_H(\gamma))\quad \text{for $x,y$ in the same part $I_j$}.\label{eq:within-part-bound}
\end{equation}
\medskip \noindent\textit{Step 3: Bounding the between-block densities.}
Recall that $H$ has chromatic number $k+1$. So, for any $x,y$ in different parts $I_i^0,I_j^0$, we have $R_{a,b}^{G_k}(x,y)=0$ for all $a,b\in \on V(H)$. Correspondingly, using \cref{eq:within-part-bound}, for all $x,y$ in different parts $I_i,I_j$ we have $R_{a,b}^{W}(x,y)\le \exp(-\Omega_H(\gamma))$, so by \cref{lem:EL} we have
\[|\log W(x,y)|\le O_H(\exp(-\Omega_H(\gamma))),\]which implies that
\[W(x,y)\ge 1-\exp(-\Omega_H(\gamma)).\qedhere\]
\end{proof}

Now, to conclude the proof of \cref{thm:H-late}, the final step is to show that $\Psi_\gamma$ has suitable convexity properties in the region implied by \cref{lem:strong-stability}.

To this end, we parameterise the class of functions $W$ satisfying the conditions in \cref{lem:strong-stability}. For parameters $\eta,\alpha>0$, let $\mc Q$ be the set of all vectors of the form $(\mbf p,\mbf Q)$, where:
\begin{itemize}
    \item $\mbf p=(p_1,\dots,p_k)\in (0,1)^k$ is a vector such that $p_1+\dots+p_k=1$ and $p_i=1/k\pm \eta$ for each $i$,
    \item $\mbf Q=(Q_{ij})_{i,j=1}^k$ is a vector of $k^2$ measurable functions $Q_{ij}:[0,1]^2\to [0,1]$, such that $Q_{ij}(x,y)=Q_{ji}(y,x)$ for each $i,j,x,y$ (in particular, this means that each $Q_{ii}$ is symmetric), and such that $Q_{ii}(x,y)\le e^{-\alpha\gamma}$ and $Q_{ij}(x,y)\ge 1-e^{-\alpha\gamma}$ for all $i\ne j$ and $x,y\in [0,1]$.
\end{itemize}
Then, for each $(\mbf p,\mbf Q)$, let $W_{\mbf p,\mbf Q}\in \mc W$ be a function $[0,1]^2\to[0,1]$ of the following form: partition $[0,1]$ into parts $I_1,\dots,I_k$ with measures $p_1,\dots,p_k$, and on each set $I_i\times I_j$, define $W_{\mbf p,\mbf Q}$ to be a rescaled version of $Q_{ij}$.

We choose say $\eta=o(1)$ and $\alpha=\Omega_H(1)$ such that \cref{lem:strong-stability} guarantees that every minimiser of $\Psi_\gamma$ can be expressed in the form $W_{\mbf p,\mbf Q}$ for some $(\mbf p,\mbf Q)\in \mc Q$.
\begin{lemma}\label{lem:convex}
    If $\gamma$ is sufficiently large, then on $\mathcal{Q}$, the map $\mf F_\gamma:(\mathbf{p},\mathbf{Q}) \mapsto \Psi_\gamma(W_{\mathbf{p},\mathbf{Q}})$ is convex.
\end{lemma}
\begin{proof}
    Consider any $(\mbf p,\mbf Q)\in \mc Q$, any $\mbf u=(u_1,\dots,u_k)\in \mb R^k$ with $u_1+\dots+u_k=0$, and any vector $\mbf V=(V_{ij})_{i,j=1}^k$ of functions $V_{ij}:[0,1]^2\to \mb R$ such that $V_{ij}(x,y)=V_{ji}(y,x)$ for each $i,j,x,y$. Our objective is to show that the second variation of $\mf F_\gamma$ at $(\mbf p,\mbf Q)$ with respect to the perturbation $(\mbf u,\mbf V)$ is nonnegative. That is to say, we want to show that the map
    $F_\gamma:\delta\mapsto \mf F_\gamma(\mbf p+\delta\mbf u,\mbf Q+\delta\mbf V)$
    satisfies $F_\gamma''(0)\ge 0$. This is a fairly routine matter of estimating the relevant Hessian, but the calculations are a little involved.

    First, we write \[\mf F_\gamma(\mbf p,\mbf Q)=\mf G(\mbf p,\mbf Q)+\gamma \mf T(\mbf p,\mbf Q),\qquad G(\delta)=\mf G(\mbf p+\delta\mbf u,\mbf Q+\delta\mbf V),\qquad T(\delta)=\mf T(\mbf p+\delta\mbf u,\mbf Q+\delta\mbf V),\] where
    \begin{align*}\mf G(\mbf p,\mbf Q)&=\sum_{i,j=1}^kp_ip_j\int g(Q_{ij}),\\\mf T(\mbf p,\mbf Q)&=
\sum_{\phi:\on V(H)\to[k]}
\left(\prod_{a\in V(H)}p_{\phi(a)}\right)
\int_{x_1,\dots,x_{v(H)}}
\prod_{ab\in E(H)}
Q_{\phi(a)\phi(b)}(x_a,x_b)\end{align*}
(here we write $[k]=\{1,\dots,k\}$).

\medskip \noindent\textit{Step 1: The entropy term.}
We compute
\[G''(0)= \sum_{i, j} p_ip_j \int \frac{V_{ij}^2}{Q_{ij}} \;\;+ 2\sum_{i, j} u_iu_j \int g(Q_{ij}) \;\;+ 4\sum_{i, j}p_i u_j \int V_{ij} \log Q_{ij}\]
(where all sums are over $i,j\in \{1,\dots,k\}$). Since $(\mathbf{p}, \mathbf{Q}) \in \mathcal{Q}$, we have 
    \[
        \sum_{i, j=1}^k u_iu_j \int g(Q_{ij}) \ge g\big(e^{-\alpha\gamma})\sum_i u_i^2-g(1-e^{-\alpha \gamma})\sum_{i,j}|u_iu_j|\ge \frac12 \sum_i u_i^2, 
    \]
    and 
    \[
        \sum_{i, j=1}^k p_ip_j \int \frac{V_{ij}^2}{Q_{ij}} \ge \bigg(\frac{1}{k}- \eta\bigg)^2\Bigg(\sum_{i\ne j} \int V_{ij}^2 +  \sum_{i} \int \frac{V_{ii}^2}{Q_{ii}}\Bigg)
        \ge \frac1{2k^2}\sum_{i\ne j} \int V_{ij}^2 +  \frac1{2k^2}\sum_{i} \int \frac{V_{ii}^2}{Q_{ii}}. 
    \]
    Also,
        \begin{align*}
        \Bigg |\sum_{i, j}p_i u_j \int V_{ij} \log Q_{ij} \Bigg|&\le \frac{2}{k} \sum_{i\ne j} |\log(1-e^{-\alpha\gamma})|\,|u_j|\int |V_{ij}| + \frac{2}{k} \sum_i |u_i| \int \Big(|V_{ii}|\, |\log Q_{ii}|\Big)\\
        &\le  \sum_{i\ne j} e^{-\alpha\gamma/2} |u_j| \sqrt{\int |V_{ij}|^2} +  \sum_i |u_i| \sqrt{\int \frac{V_{ii}^2}{Q_{ii}} \int \Big(Q_{ii}\, |\log Q_{ii}|^2\Big)}\\
        &\le  e^{-\alpha\gamma/2}\sum_{i\ne j}  \sqrt{u_j^2\int |V_{ij}|^2} +  e^{-\alpha\gamma/3} \sum_i \sqrt{u_i^2\int \frac{V_{ii}^2}{Q_{ii}}}\\
        &\le e^{-\alpha\gamma/2}\sum_{i\ne j}  \Bigg(u_j^2+\int V_{ij}^2\Bigg)  + e^{-\alpha\gamma/3} \sum_i \Bigg(u_i^2+ \int \frac{V_{ii}^2}{Q_{ii}}\Bigg),
    \end{align*}
    where in the second line we have used the Cauchy--Schwarz inequality and in the last line we have used the arithmetic-geometric mean inequality.
    
    Combining the above estimates, and noting that $(1/(2k^2))Q_{ii}^{-1}\ge e^{\alpha \gamma/2}$, we have
    \begin{equation}
        G''(0)\ge \frac13 \sum_i u_i^2+\frac1{3k^2}\sum_{i\ne j}\int V_{ij}^2+ e^{\alpha \gamma/2}\sum_{i} \int V_{ii}^2.\label{eq:G''}
    \end{equation}
    That is to say, $G''(0)$ has quite a large positive value; it remains to show that $\gamma T''(0)$ cannot be so negative as to cancel this out.

\medskip \noindent\textit{Step 2: The $H$-counting term.}
We compute
\[T''(0)=D_{\mr{ee}}+2D_{\mr{ev}}+D_{\mr{vv}}\]
where for each $\sigma\in \{\mr{ee},\mr{ev},\mr{vv}\}$ we have
\[D_\sigma=\sum_{\phi:\on V(H)\to[k]}
\Bigg(\prod_{a\in V(H)}p_{\phi(a)}\Bigg)
\int_{x_1,\dots,x_{v(H)}}
\Bigg(\prod_{ab\in E(H)}
Q_{\phi(a)\phi(b)}(x_a,x_b)\Bigg)d_\sigma(x_1,\dots,x_{v(H)})\]
with
\begin{align*}d_{\mr{ee}}(x_1,\dots,x_{v(H)})&=\sum_{\substack{a_1b_1,a_2b_2\in H\\\text{distinct}}} \frac{V_{\phi(a_1)\phi(b_1)}(x_{a_1},x_{b_1})V_{\phi(a_2)\phi(b_2)}(x_{a_2},x_{b_2})}{Q_{\phi(a_1)\phi(b_1)}(x_{a_1},x_{b_1})Q_{\phi(a_2)\phi(b_2)}(x_{a_2},x_{b_2})},\\
d_{\mr{ev}}(x_1,\dots,x_{v(H)})&=\sum_{v\in \on V(H),ab\in H} \frac{u_{\phi(v)}V_{\phi(a)\phi(b)}(x_{a},x_{b})}{p_{\phi(v)}Q_{\phi(a)\phi(b)}(x_{a},x_{b})},\\
d_{\mr{vv}}(x_1,\dots,x_{v(H)}) &=\sum_{\substack{v_1,v_2\in \on V(H)\\\text{distinct}}} \frac{u_{\phi(v_1)}u_{\phi(v_2)}}{p_{\phi(v_1)}p_{\phi(v_2)}}.
\end{align*}
Since $H$ has chromatic number $k+1$, for every $\phi:\on V(H)\to [k]$ there is an edge $ab\in H$ for which $\phi(a)=\phi(b)$, which implies that $Q_{\phi(a)\phi(b)}\le e^{-\alpha \gamma}$ (recalling the definition of $\mc Q$). This will be our main mechanism to bound the various terms in $T''(0)$.

If $\phi(a_1)\ne \phi(b_1)$ and $\phi(a_2)\ne \phi(b_2)$ then the term in $D_{\mr{ee}}$ corresponding to $a_1b_1$ and $a_2b_2$ has absolute value at most
\begin{align*}
    &O_{H}\Bigg(e^{-\alpha \gamma}\int_{x_v:v\in\{a_1,b_1,a_2,b_2\}} |V_{\phi(a_1)\phi(b_1)}(x_{a_1},x_{b_1})|\,|V_{\phi(a_2)\phi(b_2)}(x_{a_2},x_{b_2})|\Bigg)\\
    &\qquad \le O_{H}\Bigg(e^{-\alpha \gamma}\sqrt{\int V_{\phi(a_1)\phi(b_1)}^2}\cdot \sqrt{\int V_{\phi(a_2)\phi(b_2)}^2}\Bigg),
\end{align*}
using the Cauchy--Schwarz inequality. If $\phi(a_1)=\phi(b_1)$, the corresponding contribution to $D_{\mr{ee}}$ has absolute value at most
\begin{align*}
    &O_{H}\Bigg(\int_{x_v:v\in\{a_1,b_1,a_2,b_2\}} |V_{\phi(a_1)\phi(a_1)}(x_{a_1},x_{b_1})|\,|V_{\phi(a_2)\phi(b_2)}(x_{a_2},x_{b_2})|\Bigg)\\
    &\qquad\le O_{H}\Bigg( \sqrt{e^{\alpha\gamma/4}\int V_{\phi(a_1)\phi(a_1)}^2}\cdot \sqrt{e^{-\alpha\gamma/4}\int V^2_{\phi(a_2)\phi(b_2)}}\Bigg)\\
    &\qquad\le O_H\Bigg(e^{\alpha\gamma/4}\int V_{\phi(a_1)\phi(a_1)}^2+e^{-\alpha\gamma/4}\int V_{\phi(a_2),\phi(b_2)}^2\Bigg),
\end{align*}
using the Cauchy--Schwarz inequality and the arithmetic-geometric mean inequality. A similar estimate holds if $\phi(a_2)=\phi(b_2)$. We deduce that
\[|D_{\mr{ee}}|\le O_{H}\Bigg(e^{-\alpha \gamma/4}\sum_{i,j}\int V_{i,j}^2+e^{\alpha\gamma/4}\sum_{i}\int V_{ii}^2\Bigg).\]
Similarly,
\begin{align*}|D_{\mr{ev}}|&\le O_H\Bigg(   \sqrt{\bigg(\sum_i u_i^2\bigg)\bigg(\sum_i \int V_{ii}^2 + e^{-\alpha \gamma} \sum_{i,j} \int V_{ij}^2\bigg)}\Bigg)\\&\le O_H\Bigg(e^{-\alpha\gamma/4}\sum_i u_i^2+e^{\alpha\gamma/4}\sum_i\int V_{ii}^2+e^{-3\alpha \gamma/4} \sum_{i,j} \int V_{ij}^2\Bigg),\end{align*}
and
    \[
        |D_{\mr{vv}}|\le O_H\Bigg(e^{-\alpha \gamma} \sum_i u_i^2\Bigg).
    \]
    Combining all these estimates, we obtain
    \[|\gamma T''(0)|\le e^{-\alpha \gamma/4} \sum_i u_i^2+e^{-\alpha \gamma/5} \sum_i u_i^2+e^{\alpha \gamma/3}\sum_i\int  V_{ii}^2+e^{-\alpha \gamma/5} \sum_{i,j}\int V_{ij}^2,\]
    and combining this with \cref{eq:G''} we obtain $F''_\gamma(0)\ge 0$, as desired.
\end{proof}

Now it is a simple matter to deduce \cref{thm:H-late}.
\begin{proof}[Proof of \cref{thm:H-late}]
    Recall that our objective is to show that $\Psi_\gamma(W)$ is minimised by a function of the form $W_{r,t}^{(k)}$. By \cref{lem:strong-stability}, it suffices to show that $\mf F_\gamma$ is minimised by some $(\mbf p,\mbf Q)$ for which $p_1=\dots=p_k=1/k$ and for which each $Q_{ij}$ is a constant-valued function, where the constant value of $Q_{ij}$ only depends on whether $i=j$ or not.

    By \cref{lem:convex}, $\mf F_\gamma$ is convex on its domain $\mc Q$. Note that $\mf F_\gamma(\mbf p,\mbf Q)$ is symmetric under permutations of the index set $\{1,\dots,k\}$. By convexity, averaging over such permutations cannot increase the value of $\mf F_\gamma$, so there is a minimiser such that $p_1=\dots=p_k$, and $Q_{ii}$ is the same for each $i$, and $Q_{ij}$ is the same for each $i\ne j$. Also, $\mf F_\gamma(\mbf p,\mbf Q)$ is symmetric with respect to measure-preserving relabellings $\sigma_1,\dots,\sigma_k:[0,1]\to [0,1]$ (where we relabel each $Q_{ij}(x,y)$ as $Q_{ij}(\sigma_i(x),\sigma_j(y))$), so a further averaging argument shows that there is a minimiser such that each $Q_{ij}$ is a constant-valued function.
\end{proof}

\section{Steiner triple systems and Latin squares}\label{sec:STS}

In this section we discuss how to prove \cref{thm:STS,thm:latin}. Recall that the lower bounds were already obtained in previous work~\cite{KSSS23,KSSS24}, so our task is to prove the upper bounds. First we reduce to a variational problem, via a variant of \cref{thm:main-m-upper}, and then we solve this variational problem using ideas related to Sidorenko's conjecture.

Even if the reader is only interested in \cref{thm:STS,thm:latin}, we recommend that they first familiarise themselves with the proof of the $L=0$ case of \cref{thm:main-upper} in \cref{subsec:independent-example}, and the deduction of \cref{thm:main-m-upper} from \cref{thm:main-upper-general}. The ideas in those (quite short) proofs will be our starting point for this section.

First, we make some definitions which allow us to talk about Steiner triple systems and Latin squares using similar notation to \cref{thm:main-m-upper}. 

\begin{definition}
\label{def:STS-H}
Define the graph $G$, the 3-uniform hypergraph $K$ and the
integer $m\in\mb N$ in one of the following two ways:
\begin{enumerate}
\item let $G=K_n$ be the complete graph and let $K=K_{n}^{(3)}$ be the complete 3-uniform hypergraph on the same $n$-vertex set (so $|G|=\binom n2$ and $|K|=\binom n3$), and let\footnote{Here we need to assume $\binom n2$ is divisible by 3, or else no Steiner triple system of order $n$ exists.} $m=\binom{n}{2}/3=|G|/3$, or
\item let $G=K_{n,n,n}$ be the complete 3-partite graph and let $K=K_{n,n,n}^{(3)}$ be the complete 3-partite 3-uniform hypergraph on the same $3n$-vertex set (so $|G|=3n^2$ and $|K|=n^3$), and let $m=n^2=|G|/3$.
\end{enumerate}
In both cases, we let $N=|K|$ and define hypergraphs\footnote{Note that vertices of $\mc H$ are triples of vertices of $K$. This is slightly confusing, and we will therefore make sure to be consistent with our notation: vertices of $\mc H$ will always be denoted with the letters $u,v$, while vertices of $K$ will be denoted with the letters $x,y,z$.} $\mathcal{H}_{4},\mathcal{H}_{2}$
on the vertex set $K$ (i.e., the vertices of $\mathcal{H}_{4},\mathcal{H}_{2}$
are triangles in $G$), as follows. For $\mathcal{H}_{4}$,
put a hyperedge for each Pasch configuration, and for $\mathcal{H}_{2}$,
put a hyperedge for each pair of triples intersecting in a pair of
vertices. Then, let $\mathcal{H}=\mathcal{H}_{4}\cup\mathcal{H}_{2}$
(this is a mixed-uniformity hypergraph, with some hyperedges of size
4 and some hyperedges of size 2).
\end{definition}

Setting (1) will be used for \cref{thm:STS}: in this case, an independent
set of size $m$ in $\mathcal{H}$ corresponds to an
anti-Pasch Steiner triple system of order $n$. Setting (2)
will be used for \cref{thm:latin}: in this case, an independent set of size $m$ in $\mathcal{H}$ corresponds to an order-$n$ Latin square with property
$N_{2}$.

Unfortunately, we cannot apply \cref{thm:main-m-upper} to the hypergraphs in \cref{def:STS-H}, for several
reasons. Most obviously, \cref{thm:main-m-upper} is not stated for mixed-uniformity hypergraphs. But, also, while it is easy to see that $\mc H_4$ is $(O(1/n),O(1);1/n)$-degree-bounded, it turns out that $\mathcal{H}_2$ does not satisfy the necessary degree-boundedness condition. For example, in the first case in \cref{def:STS-H}, if $u,v\in \on V(\mc H)$
correspond to two triples in $K=K_{n}^{(3)}$ sharing more than one
vertex, then $\on{pair-deg}_{\mc H_2}(u,v)=2+(n-4)$,
meaning that if $p$ has order of magnitude $1/n$ and $r=2$ then
$\Gamma(\mc H_2)\ge\Omega(p^{1-r})$. As a consequence,
it turns out that the book hypergraphs $\on{bk}_{k}^{a}(\mathcal{H})$
(as defined in \cref{def:books}) do not satisfy the necessary degree conditions for the proof of \cref{thm:main-upper,thm:main-m-upper}. This is not just a technicality; it turns out that counting independent sets in $\mc H_2$ leads to totally different formulas than counting independent sets in $\mc H_4$ (roughly speaking, the former leads to ``Bernoulli-type'' statistics instead of ``Poisson-type'' statistics).

Fortunately, our methods are robust enough to incorporate different types of statistics (and to incorporate the interaction between different uniformities). We need to adapt the definitions of our
book hypergraphs in two ways: first, we need to generalise the definition to consider multiple uniformities simultaneously, and second, we consider books with at most one ``$\mc H_2$-page'' for each $g\in G$. (Note that if we were to consider books including more than one such page, then the page vertices themselves would not form an independent set of $\mc H_2$; as we are interested in counting independent sets we never need to consider such books). 
\begin{definition}\label{def:book-STS}
Let $G,\mc H_4,\mc H_2,\mc H$ be as in \cref{def:STS-H}. Fix an ordering on the pairs of vertices of $G$, so for each $v=xyz\in K=\on V(\mc H)$ we can write $v[1]<v[2]<v[3]$ for the three pairs $xy,xz,yz\in G$, labelled according to our fixed ordering.

For $a\in \mb N_{\ge 0}$ and $Q\subseteq \{1,2,3\}$, an \emph{$(a,Q)$-book} is a spine vertex $v\in \on V(\mc H)=K$, a sequence of distinct
hyperedges $(h_{1},\dots,h_{a})$ of $\mathcal{H}_{4}$ containing $v$, and a choice of a triple $u_i\in K\setminus v$ containing $v[i]$ for each $i\in Q$. The \emph{page vertices} of the $(a,Q)$-book are the vertices $u_i\in \on V(\mc H)$ for $i\in Q$, together with the vertices in $(h_1\cup \dots \cup h_a)\setminus v$ (note that these two types of page vertices are always disjoint).
 \begin{itemize}
\item For $a,k\in \mb N_{\ge 0}$ and $Q\subseteq \{1,2,3\}$ and $v\in V$, let $\on{bk}^{a,Q}_k(\mc H;v)$ be the $k$-uniform multihypergraph defined as follows. For each $(a,Q)$-book which has spine $v$ and exactly $k$ page vertices, put this set of page vertices as a hyperedge of $\on{bk}^{a,Q}_k(\mc H;v)$.
    \item Denote by $\on{bk}^{a,Q}_k(\mc H)$ the $k$-uniform multihypergraph obtained as the multiset union of $\on{bk}^{a,Q}_k(\mc H;v)$ for each $v\in V$.
    \end{itemize}
\end{definition}

When we define book hypergraphs in this way, the necessary analogue of \cref{lem:book-deg} holds. 
\begin{lemma}\label{lem:book-degrees-STS}
Let $\mc H_4,\mc H_2,\mc H$ be as in \cref{def:STS-H}. Then, for any $A\in\mb N$, any $a\le A$, any $Q\subseteq\{1,2,3\}$ and any $k\le3a+|Q|$,
we have
\begin{align}
    \Delta_s(\on{bk}^{a,Q}_k(\mc H))&\le O_{A}(1)\cdot (o(1/n))^{s-1}\cdot 
    n^{k}\text{ for }s\in \{1,\dots,k\},\label{eq:book-degrees-goal-3-STS}
    \\
    |\on{bk}^{a,Q}_k(\mc H;v)|&\le O_{A}(1)\cdot n^{k}\text{ for all }v\in V.\label{eq:book-degrees-goal-1-STS}
    \end{align}
    Moreover, for $k<3a+|Q|$ we have
    \begin{equation}
        |\on{bk}^{a,Q}_k(\mc H;v)|\le O_{A}(1)\cdot o(1)\cdot n^{k}\text{ for all }v\in V.\label{eq:book-degrees-goal-2-STS}
    \end{equation}
\end{lemma}
\begin{proof}
This proof proceeds along similar lines to the proof of \cref{lem:book-deg}; it involves a fairly mechanical consideration of the ways in which pages of both types can intersect a set of fixed vertices and each other. So, we will be brief with details.

First, note that for any $(a,Q)$-book, the $\mc H_2$-page vertices are distinct from each other and disjoint from the $\mc H_4$-page vertices. For a fixed spine $v=xyz$, there are at most $O(n)$ choices for each $\mc H_2$-page vertex, so \cref{eq:book-degrees-goal-1-STS,eq:book-degrees-goal-2-STS} follow directly from \cref{lem:book-deg} (using that $\mc H_4$ is $(o(1),O(1);1/n)$-degree-bounded, which can be straightforwardly checked in an analogous way to \cref{prop:H-free}).

For \cref{eq:book-degrees-goal-3-STS}, fix $v_1,\dots,v_s\in \on V(\mc H)$; our goal is to upper-bound $\deg_{\on{bk}_k^{a,Q}(\mc H)}(v_1,\dots,v_s)$. First, specify which of our fixed vertices should be $\mc H_2$-page vertices and which should be $\mc H_4$-page vertices (there are $O_A(1)$ ways to make this specification). Let the number of fixed $\mc H_2$-page vertices be $t\le |Q|$. We distinguish various cases\footnote{There are various different ways to break down into cases, each of which involves slightly different combinations of estimates. However, none of the ways we tried seemed to yield a really clean proof.} for the values of $t,s$.
\begin{itemize}
    \item If $t=0$ or $s-t\ge 2$, the desired bound follows from \cref{lem:book-deg} (using that there are $O(n)$ choices for each unfixed $\mc H_2$-page vertex, given the $\mc H_4$-pages).
    \item If $t\ge 2$, note that there are only $O(1)$ possible spines. Indeed, if we write $x_1y_1z_1$ and $x_2y_2z_2$ for two fixed $\mc H_2$ page vertices (where at least two of $x_1,x_2,y_1,y_2,z_1,z_2$ must coincide), then the spine $xyz$ must satisfy $\{x,y,z\}\subset\{x_1,x_2,y_1,y_2,z_1,z_2\}$. For each such spine $v=xyz$, the contribution to $\deg_{\on{bk}_k^{a,Q}(\mc H)}(v_1,\dots,v_s)$ is at most $O_A(n^{b-(s-t)})\cdot O(n^{|Q|-t})=O_A(1)\cdot(o(1/n))^{s-1}n^k$, where $b=k-|Q|$ is the number of $\mc H_4$-page vertices, and we have used a basic overlap-statistics argument as in \cref{eq:book-calculation-for-later} (using $\Delta_\ell(\mc H_4)\le O(n^{4-\ell})$) to count the number of choices for the $\mc H_4$-pages with a given spine and $s-t$ fixed $\mc H_4$-page vertices.
    \item It remains to consider the case where $s\in \{1,2\}$ and $t=1$. There are $O(n)$ choices for the spine, and we can use a basic overlap-statistics argument as in \cref{eq:book-calculation-for-later} to see that the contribution to $\deg_{\on{bk}_k^{a,Q}(\mc H)}(v_1,\dots,v_s)$ is at most $O_A(n^{b-(s-t)})\cdot O(n^{|Q|-t})=O_A(n^{k-s})$, so we obtain the bound $\deg_{\on{bk}_k^{a,Q}(\mc H)}(v_1,\dots,v_s)\le O_A(n^{k-s+1})$. This is already enough if $s=1$. To improve this bound for $s=2$ (meaning $t=s-t=1$), reorder the pages so that the first $\mc H_4$-page contains the unique fixed $\mc H_4$-page vertex. After we make a choice for the spine there are at most $\Delta_2(\mathcal H_4)=o(n^2)$ choices for this first $\mc H_4$-page (this saves an \(o(1)\)-factor over the basic overlap-statistics bound).
    \qedhere
\end{itemize}
\end{proof}

\subsection{Reducing to a variational problem}

In this subsection we discuss how to adapt the proof of \cref{thm:main-m-upper} to prove the following lemma.

\begin{lemma}\label{lem:STS-reduction-m}
For any $\varepsilon>0$, the following holds for sufficiently large $n$. Let $\mc H,\mc H_4,m$ be as in \cref{def:STS-H}. Then, for $R\sim\on{Unif}(V,m)$,
\[
\log \Pr\big[|\mc H[R]|=0\big]\le  -\min_{\vec{q}}\Big(\on{D}_{\mr{KL}}(\vec{q}\,\|\,m/N)+\mc H_4(\vec q)\Big)-3m+\varepsilon m,
\]
where the minimum is over all $\vec q\in [0,1]^{\on V(\mc H)}$ whose sum of entries is exactly $m$.
\end{lemma}

Just as for \cref{thm:main-m-upper}, we need to start with a similar theorem for binomial random subsets (i.e., we need an analogue of the $L=0$ case of \cref{thm:main-upper-general}). For this, we need some additional notation.

\begin{definition}
    Let $G,K$ be as in \cref{def:STS-H}. For $g\in G$ and a weighting $\vec q=(q_v)_{v\in K}$, write
\[
    \ell_g(\vec q)=\sum_{v\in K:g\subseteq v}q_v.
\]
Also, let $\chi(x)=x+(1-x)\log(1-x)$.
\end{definition}

\begin{lemma}\label{lem:STS-reduction-p}
For any $\varepsilon,B>0$, there is $N_0\in \mb N$ such that the following holds. Let $N\ge N_0$ and let $\mc H,\mc H_4,G,m$ be as in \cref{def:STS-H}. Then, for\footnote{For this choice of $p$, it is easy to see that $\mc H_4$ is $(o(1),O_B(1);p)$-bounded.} $p=Bm/N$ and $R\sim\on{Bin}(V,p)$,
\[
\log \Pr\big[|\mc H[R]|=0\text{ and }|R|\le m\big]\le  -\min_{\vec{q}\in \mc Q^*}\Bigg(\on{D}_{\mr{KL}}(\vec{q}\,\|\,p)+\mc H_4(\vec q)+\sum_{g\in G}\chi(\ell_g(\vec q))\Bigg)+\varepsilon Np,
\]
where $\mc Q^*$ is the set of all $\vec q\in [0,p]^{\on V(\mc H)}$ whose sum of entries is at most $m$, and which satisfy $\ell_g(\vec q)\in [0,1]$ for each $g\in G$.
\end{lemma}

\begin{remark}
    The constraint $\ell_g(\vec q)\le 1$ corresponds to the fact that in an independent set $X\subseteq \on V(\mc H)$, each pair $g\in G$ can appear in at most one of the triples of $X$. It seems that with a bit more work one can also add a similar constraint $\ell_g(\vec q)=1$ to the statement of \cref{lem:STS-reduction-m}, but this turns out not to be necessary to prove \cref{thm:latin,thm:STS}.
\end{remark}

We prove \cref{lem:STS-reduction-p} in nearly the same way as the $L=0$ case of \cref{thm:main-upper} (proved in \cref{subsec:independent-example}). The only new ingredient we need is a generalisation of \cref{lem:moment}, which is essentially the method of moments with the product of a Poisson distribution and three Bernoulli distributions.

\begin{lemma}\label{lem:moment-STS}
For every $M,\beta>0$, there are $A\in\mb N$ and $\gamma>0$ such that the
following holds.  Let $\mc F$ be a sigma-algebra.  For every $v\in V$, let
$D_v$ be a nonnegative integer-valued random variable, let
$Z_{v,1},Z_{v,2},Z_{v,3}\in\{0,1\}$, and let
$\mu_v,\zeta_{v,1},\zeta_{v,2},\zeta_{v,3}$ be $\mc F$-measurable random
variables satisfying $0\le\mu_v\le M$ and $0\le\zeta_{v,i}\le1$.
Suppose that, for every $a\in\{0,\dots,A\}$ and every $Q\subseteq\{1,2,3\}$, we have
\[
\mb E\sum_{v\in V}
\left|
 \mb E\left[(D_v)(D_v-1)\dots(D_v-a+1)\prod_{i\in Q}Z_{v,i}\,\middle|\,\mc F\right]
 -\mu_v^a\prod_{i\in Q}\zeta_{v,i}
\right|
\le\gamma |V|.
\]
Then
\[
\mb E\sum_{v\in V}
\left|
 \Pr[D_v=Z_{v,1}=Z_{v,2}=Z_{v,3}=0\mid\mc F]
 -\exp(-\mu_v)\prod_{i=1}^3(1-\zeta_{v,i})
\right|
\le\beta |V|.
\]
\end{lemma}

\begin{proof}
Let $\zeta_{v,Q}=\prod_{i\in Q}\zeta_{v,i}$ and $Z_{v,Q}=\prod_{i\in Q}Z_{v,i}$, and let 
\[P_{v,Q}=\mb E[\one_{\{D_v=0\}}Z_{v,Q}\,|\,\mc F]=\Pr[D_v=0\text{ and }Z_{v,i}=1\text{ for }i\in Q\,|\,\mc F].\]
Using the Bonferroni inequalities and Taylor's theorem as in the proof of \cref{lem:moment}, we have
    \[|P_{v,Q}-\exp(-\mu_v)\zeta_{v,Q}|\le O_M\Bigg(\frac{M^{A}}{A!}+\sum_{a=0}^A\frac{\big|\mb E[D_v(D_v-1)\dots (D_v-a+1)Z_{v,Q}\,|\,\mc F] - \mu_{v}^a\zeta_{v,Q}\big|}{a!}\Bigg).\]

    Then, by the inclusion-exclusion formula we have
    \[\Pr[D_v=Z_{v,1}=Z_{v,2}=Z_{v,3}=0\,|\,\mc F]=\sum_{Q\subseteq\{1,2,3\}}(-1)^{|Q|}P_{v,Q}.\]
    Note that
    \[
\sum_{Q\subseteq\{1,2,3\}}(-1)^{|Q|}
\exp(-\mu_v)\zeta_{v,Q}=\exp(-\mu_v)\prod_{i=1}^3(1-\zeta_{v,i}),
\]
so, summing over $v$ and averaging over $\mc F$, we obtain
\[
\mb E\sum_{v\in V}
\left|
 \Pr[D_v=Z_{v,1}=Z_{v,2}=Z_{v,3}=0\mid\mc F]
 -\exp(-\mu_v)\prod_{i=1}^3(1-\zeta_{v,i})
\right|\le  O_M\bigg(\frac{M^{A}}{A!}+e\gamma\bigg)|V|.
\]
The desired result follows, assuming $A$ is sufficiently large in terms of $\beta,M$ and $\gamma>0$ is sufficiently small in terms of $\beta$.
\end{proof}

\begin{proof}[Proof of \cref{lem:STS-reduction-p}]
Throughout this proof we view $B$ as a fixed constant (so all asymptotic notation and parameter choices may implicitly depend on $B$).

We proceed in a similar way to \cref{subsec:independent-example}, so we will be brief with details that we have already seen there. We will need auxiliary parameters $T,A,\gamma,\xi,\alpha,\beta,\theta,C_0$, chosen in exactly the same order as in \cref{subsec:independent-example}. In addition to the dependencies listed there, we also need $\beta$ to be small in terms of $T$. 

\medskip\noindent
\textit{Step 1: Reducing to an entropic problem.}
Recall that $\on V(\mc H)=K$. Let $\mc D=\{S\subseteq K:|\mc H[S]|\le 0\text{ and }|S|\le m\}$, and let $Y$ be a sample from the conditional distribution of $R\sim\on{Bin}(K,p)$ given $R\in\mc D$. Apply \cref{thm:KS} with parameters $r_0=3A+3$ and $\alpha,\xi,C_0$ to obtain a set $W\subseteq K$. For any possible outcome $U$ of $Y\cap W$ we let
\[\on H^{U}[Y]=\on H[Y\,|\,Y\cap W=U],\quad q_v^U=\Pr[v\in Y\,|\,Y\cap W=U],\]
and we obtain
\begin{align}
\log \mb P[|\mc H[R]|=0\text{ and }|R|\le m]&= -\on D_{\mathrm{KL}}\big(\mc L(Y\cap W) \,\|\, \mathrm{Bin}(W,p)\big) \notag\\
&\qquad\qquad + \on H[Y\,|\,Y\cap W]+\sum_{v\in K\setminus W} \Big(\Pr[v\in Y]\log p + \Pr[v\notin Y]\log(1-p)\Big)\notag\\
& \le \mb E\Bigg[ \on H^{Y\cap W}[Y] +\sum_{v\in K\setminus W} \Big(q^{Y\cap W}_v\log p + (1-q^{Y\cap W}_v)\log(1-p)\Big)\Bigg].\label{eq:average-entropy-STS}
\end{align}

\medskip\noindent
\textit{Step 2: Sequential entropy.}
Fix any outcome $U$ of $Y\cap W$, and let $X$ have the conditional
distribution of $Y\setminus W$ given $Y\cap W=U$, so $\on H^{U}[Y]=\on H[X]$. Let $q_v=q^U_v=\Pr[v\in X]$ for $v\notin W$, and $q_v=0$ for $v\in W$.

We now need some new parameters to describe the different types of ``threats'' from hyperedges of $\mc H_2$ or $\mc H_4$. For $S\subseteq K\setminus W$ and $v\in K\setminus W$:
\begin{itemize}
    \item Similarly to \cref{subsec:independent-example}, let $D_v(S)$ be the number of hyperedges $e\in \mc H_4$ such that $v\in e$ and $e\setminus v\subseteq S$;
    \item for $i\in\{1,2,3\}$ define $Z_{v,i}(S)\in \{0,1\}$ by letting $Z_{v,i}(S)=1$ if $v[i]\subseteq u$ for some $u\in S\setminus v$.
\end{itemize}
Note that if $v\notin S$ and at least one of $D_v(S),Z_{v,1}(S),Z_{v,2}(S),Z_{v,3}(S)$ is nonzero, then augmenting $\mc H[S]$ with $v$ would introduce a hyperedge (of $\mc H_4$ or $\mc H_2$, respectively). Let $\on{ex}(S)$ be the set of $v\in K\setminus W$ such that $D_v(S)=Z_{v,1}(S)=Z_{v,2}(S)=Z_{v,3}(S)=0$, let $V_1\cup\dots\cup V_T$ be a uniformly random partition of $K\setminus W$, let $X_t=X\cap V_t$ and let \[b_v(t)=b_v^U(t)=\Pr[v\in \on{ex}(X_{<t})]=\Pr[D_v(X_{<t})=Z_{v,1}(X_{<t})=Z_{v,2}(X_{<t})=Z_{v,3}(X_{<t})=0].\] By \cref{lem:seq-ent}, we have
\begin{equation}
    \on H[X]\le \frac 1T\sum_{t =1}^{T}\sum_{v\in K\setminus W}q_v\big(1-\log q_v+\log b_v(t)\big).\label{eq:seq-ent-application-STS}
\end{equation} 
Now, 
let \[\mu_v(t)=\mu_v^U(t)=\Big(\frac{t-1}T\Big)^{3}\sum_{\substack{e\in \mc H_4[K\setminus W]:\\v\in e}}\;\prod_{u\in e\setminus v}q_{u},\qquad \zeta_{v,i}(t)=\zeta_{v,i}^U(t)=\frac{t-1}T\cdot \ell_{v[i]}(\vec q)\]
and
\[\tilde  b_v(t)=\tilde b_v^U(t)=\exp(-\mu_v(t))\prod_{i=1}^3(1-\zeta_{v,i}(t)).\]

\medskip \noindent\textit{Step 3: Elementary estimates.}
As in \cref{subsec:independent-example}, by Harris' inequality and $(o(1),O_B(1);p)$-degree-boundedness of $\mc H_4$, we have $\mu_v(t)\le p^{3}\Delta_1(\mc H_4)=O_B(1)$. Also note that $(t-1)/T\le1-1/T$ and
$\ell_g(\vec q)=\mb E|\{v\in X:g\subseteq v\}|\le1$ for all $g\in G$, so
\begin{equation}
    \tilde b_v(t)\ge \exp(-O(1))/T^3=\Omega_{T}(1).\label{eq:nondegenerate-probability-STS}
\end{equation}
As in
\cref{subsec:independent-example}, considering separately the cases
$b_v(t)\le\tilde b_v(t)$ and
$b_v(t)>\tilde b_v(t)$ yields
\[
 \log b_v(t)-\log\tilde b_v(t)
 \le O_{T}(|b_v(t)-\tilde b_v(t)|),
\]
so we deduce from \cref{eq:seq-ent-application-STS} that
\begin{align}&\on H[X] +\sum_{v\in K\setminus W} \big(q_v\log p + (1-q_v)\log(1-p)\big)\notag\\
&\qquad\qquad\le \sum_{v\in K\setminus W}\big(q_v-q_v\log q_v+q_v\log p+(1-q_v)\log(1-p)\big) \notag\\
&\qquad\qquad\qquad\qquad +\frac1T\sum_{t=1}^T\sum_{v\in K\setminus W} q_v \log \tilde b_{v}(t)
+ O_T\Bigg(\sum_{t=1}^Tp\sum_{v\in K\setminus W}|b_{v}(t)-\tilde b_{v}(t)|\Bigg).\label{eq:H-expansion-STS}
\end{align}
Let $\vec q=(q_v)_{v\in K}\in [0,p]^K$. Note that $\sum_v q_v\le \mb E |X|\le m$ and $\ell_g(\vec q)=\mb E|\{v\in X:g\subseteq v\}|\le1$ by the definition of $X$ via $Y$, so $\vec q\in \mc Q^*$.
Exactly as in \cref{eq:H-KL-independent}, we have
\begin{align}&\sum_{v\in K\setminus W}\big(q_v-q_v\log q_v+q_v\log p+(1-q_v)\log(1-p)\big)\notag\\
&\qquad \le -\on D_{\mr{KL}}(\vec q\,\|\,p)+O(p^2 N)+O(p|W|)\le -\on D_{\mr{KL}}(\vec q\,\|\,p)+(\varepsilon/5)pN,\label{eq:H-KL-STS}\end{align}
and as in \cref{eq:H-mean-independent} we have
\begin{equation}\frac1T\sum_{t=1}^T\sum_{v\in K\setminus W} q_v \cdot (-\mu_{v}(t))\le -\mc H_4(\vec q)+(\varepsilon/5) pN.\label{eq:H-mean-STS}\end{equation}
Compared to \cref{subsec:independent-example}, we also need to consider some additional terms, due to the slightly more complicated definition of $\tilde b_v(t)$. Indeed, note that
\[
\frac1T\sum_{t=1}^T\sum_{v\in K}q_v
 \sum_{i=1}^3\log\big(1-\zeta_{v,i}(t)\big)
=\sum_{g\in G} \frac1T\sum_{t=1}^T
 \ell_g(\vec q)\log\Big(1-\frac{t-1}{T}\cdot \ell_g(\vec q)\Big).
 \]
 By a Riemann sum approximation we have 
 \[\frac1T\sum_{t=1}^T
 \ell_g(\vec q)\log\Big(1-\frac{t-1}{T}\cdot \ell_g(\vec q)\Big)=\int_0^1\ell_g(\vec q)\log(1-z\ell_g(\vec q))\,dz+o_{T\to\infty}(1)=-\chi(\ell_g(\vec q))+o_{T\to\infty}(1),\]
 so
 \begin{equation}
     \frac1T\sum_{t=1}^T\sum_{v\in K}q_v
 \sum_{i=1}^3\log\big(1-\zeta_{v,i}(t)\big)=-\sum_{g\in G}\chi(\ell_g(\vec q))+o_{T\to \infty}(m).\label{eq:pair-cost-STS}
 \end{equation}
 Combining
\cref{eq:seq-ent-application-STS,eq:H-KL-STS,eq:H-mean-STS,eq:pair-cost-STS,eq:H-expansion-STS}, we deduce
\begin{align}&\on H[X] +\sum_{v\in K\setminus W} \big(q_v\log p + (1-q_v)\log(1-p)\big)\notag\\
& \le -\min_{\vec{q}\in \mc Q^*}\Bigg(\on{D}_{\mr{KL}}(\vec{q}\,\|\,p)+\mc H_4(\vec q)+\sum_{g\in G}\chi(\ell_g(\vec q))\Bigg)+O_T\Bigg(\sum_{t=1}^Tp\sum_{v\in K\setminus W}|b_{v}(t)-\tilde b_{v}(t)|\Bigg)+(\varepsilon/2)pN.\label{eq:conditional-with-random-error-STS}\end{align}

\medskip\noindent
\textit{Step 4: Joint Bernoulli/Poisson approximation.}
It remains to control the final term in
\cref{eq:conditional-with-random-error-STS}. We proceed similarly to \cref{subsec:independent-example}, using \cref{lem:moment-STS} instead of \cref{lem:moment}. This requires a few additional estimates.

Let $\on{bk}^{a,Q}(\mc H[K\setminus W];v)=\bigcup_{k}\on{bk}^{a,Q}_k(\mc H[K\setminus W];v)$ and observe that
\begin{align}
&\mb E\Bigg[D_v(X_{<t})(D_v(X_{<t})-1)\dots(D_v(X_{<t})-a+1)\prod_{i\in Q}Z_{v,i}(X_{<t})\Bigg]\notag\\
&\qquad=\sum_{e\in \on{bk}^{a,Q}(\mc H[K\setminus W];v)}\mb E[\one_{e\subseteq X_{<t}}]\notag\\
&\qquad=\sum_{k=0}^{3a+|Q|}\bigg(\frac{t-1}T\bigg)^k\sum_{e\in \on{bk}^{a,Q}_k(\mc H[K\setminus W];v)}\mb E[\one_{e\subseteq X}].\label{eq:falling-factorial-STS}\end{align}

Next, let $q_e(t)=((t-1)/T)^{|e|}\prod_{u\in e}q_u$. To apply the conclusion of \cref{thm:KS}, we need to estimate
\[\sum_{k=0}^{3a+|Q|}\bigg(\frac{t-1}T\bigg)^k\sum_{e\in \on{bk}^{a,Q}_k(\mc H[K\setminus W];v)}\;\prod_{u\in e}\mb E[\one_{u\in  X}]=\sum_{e\in \on{bk}^{a,Q}(\mc H[K\setminus W];v)}q_e(t).\]
For this, note that for each $(a,Q)$-book, the page vertices corresponding to $\mc H_2$ are distinct from the page vertices corresponding to $\mc H_4$. So, we can break down the above sum over $e\in \on{bk}^{a,Q}$ into two separate sums over the two different types of pages. Indeed, recalling \cref{def:books} we have
\[\sum_{e\in \on{bk}^{a,Q}(\mc H[K\setminus W];v)}q_e(t)=\Bigg(\prod_{i\in Q}\;\sum_{\substack{u\in K\setminus v:\\u\supseteq v[i]}}q_u(t)\Bigg)\Bigg(\sum_{e\in \on{bk}^{a}(\mc H_4[K\setminus W];v)}q_e(t)\Bigg).\]
Recall from \cref{subsec:independent-example} that
\[
\sum_{e\in \on{bk}^{a}(\mc H_4[K\setminus W];v)}q_e(t)=\mu_v(t)^a\pm \gamma/2
\]
and, recalling that $q_v\le p$, note that
\[\sum_{\substack{u\in K\setminus v:\\u\supseteq v[i]}}q_u(t)=\frac{t-1}T\cdot(\ell_{v[i]}(\vec q)-q_v)= \zeta_{v,i}(t)\pm o(1).\]
Recalling that $\mu_v(t)\le O(1)$, we deduce
\begin{equation}
    \sum_{k=0}^{3a+|Q|}\bigg(\frac{t-1}T\bigg)^k\sum_{e\in \on{bk}^{a,Q}_k(\mc H[K\setminus W];v)}\;\prod_{u\in e}\mb E[\one_{u\in  X}]=\mu_v(t)^a\prod_{i\in Q}\zeta_{v,i}(t)\pm (2\gamma/3)\label{eq:product-independent-STS}
\end{equation}
Then, the conclusion of \cref{thm:KS} (recalling \cref{lem:book-degrees-STS} to verify the condition \cref{eq:KS-degree-condition}) implies that 
\[
\mb E\sum_{v\in K\setminus W} \Bigg|
\sum_{k=0}^{3a+|Q|}\bigg(\frac{t-1}T\bigg)^k\sum_{e\in \on{bk}^{a,Q}_k(\mc H[K\setminus W];v)}\bigg(\mb E[\one_{e\subseteq Y}\,|\,Y\cap W]-\prod_{u\in e}\mb E[\one_{u\in Y}\,|\,Y\cap W]\bigg)\Bigg|
\le  (3a+3)\xi N
\]
which is at most $(\gamma/3)N$ (recalling that $\xi$ is small in terms of $\gamma,A$). For all $t$ let $Y_t=(Y\setminus W)\cap V_t$ and $Y_{<t}=Y_1\cup \dots\cup Y_{t-1}$, so  $X_{<t}$ has the distribution of $Y_{<t}$ given $Y\cap W=U$. Substituting \cref{eq:falling-factorial-STS,eq:product-independent-STS} into the above inequality, we obtain
\begin{align*}
&\mb E\sum_{v\in K\setminus W} \Bigg|\mb E\Bigg[D_v(Y_{<t})(D_v(Y_{<t})-1)\dots(D_v(Y_{<t})-a+1)\prod_{i\in Q}Z_{v,i}(Y_{<t})\,\Bigg|\,Y\cap W\Bigg]\\
&\qquad\qquad \qquad -\mu_v^{Y\cap W}(t)^a\prod_{i\in Q}\zeta^{Y\cap W}_{v,i}(t)\Bigg| \le  \gamma N,
\end{align*}
so by \cref{lem:moment-STS} (with $\mc F$ being the $\sigma$-algebra generated by $Y\cap W$) we have
\begin{align*}
    &\mb E\sum_{v\in K\setminus W}|b_v^{Y\cap W}(t)-\tilde b_v^{Y\cap W}(t)|\\
    &\qquad =\mb E\sum_{v\in K\setminus W}\Bigg|\Pr[D_v(Y_{<t})=Z_{v,1}(Y_{<t})=Z_{v,2}(Y_{<t})=Z_{v,3}(Y_{<t})=0\,|\,Y\cap W]\\
    &\qquad\qquad\qquad \qquad\qquad  -\exp(-\mu_v^{Y\cap W}(t))\prod_{i=1}^3(1-\zeta^{Y\cap W}_{v,i}(t))\Bigg|\le \beta N.
\end{align*}
The desired result then follows by combining this with \cref{eq:conditional-with-random-error-STS,eq:average-entropy-STS} (and using our new assumption that $\beta$ is small with respect to $T$, in addition to $\varepsilon$).
\end{proof}

Now, we briefly discuss how to deduce \cref{lem:STS-reduction-m} from \cref{lem:STS-reduction-p}, in the same way that we deduced \cref{thm:main-m-upper} from \cref{thm:main-upper-general}.
\begin{proof}[Proof of \cref{lem:STS-reduction-m}]
As in \cref{subsec:uniform-random-upper}, let $p=Bm/N$ for some large $B$ (we need it to be large in terms of $\varepsilon$). Let $a_B=\log B+\log(1-m/N)-\log(1-p)$ and 
\[
I_B(\vec{q})=\on{D}_{\mr{KL}}(\vec{q}\,\|\,m/N)+\mathcal{H}_4(\vec{q})+a_B(m-|\vec{q}|).
\]
The same considerations as in the proof of \cref{thm:main-m-upper} (in \cref{subsec:uniform-random-upper}), using \cref{lem:STS-reduction-p} in place of \cref{thm:main-upper-general}, show that
\[\log\Pr\big[|\mathcal{H}[R]|=0\big]\le-\min_{\vec{q}\in \mc Q^*}\Bigg(I_B(\vec{q})+\sum_{g\in G}\chi(\ell_g(\vec q))\Bigg)+(\varepsilon/2)m.\]
So, to prove the theorem, it suffices to show that
\begin{equation}
\min_{\vec{q}\in \mc Q^*}\Bigg(I_B(\vec{q})+\sum_{g\in G}\chi(\ell_g(\vec q))\Bigg)\ge\min_{\vec{q}\in[0,1]^{\on V(\mc H)}:|\vec{q}|=m}I_B(\vec{q})+3m-(\varepsilon/2)m.\label{eq:goal-STS-binomial-reduction}    
\end{equation}

Consider any $\vec{q}\in \mc Q^*$;
it suffices to find $\vec{s}$ satisfying $|\vec{s}|=m$ such that
$I_B(\vec{q})+\sum_{g\in G}\chi(\ell_g(\vec q))\ge I_B(\vec{s})+3m-(\varepsilon/2)m$. Suppose that the constant weighting $(m/N)\vec 1$ does not do the job: namely, suppose 
\[
  I_B(\vec{q})+\sum_{g\in G}\chi(\ell_g(\vec q))<I_B((m/N)\vec 1)+3m=\mc H_4(m/N)+3m
  \]
  (here we are using $|G|=3m$ and $\chi(\ell_g((m/N)\vec 1))=\chi(1)=1$). Since $\mc H_4$ is $(\lambda,O(1);m/N)$-bounded, we have $\mc H_4(m/N)=O(m)$, so $a_B(m-|\vec q|)\le I_B(\vec q)\le O(m)$. Since $a_B=\Omega(\log B)$, we deduce
\begin{align}
m-|\vec{q}| &\le O(m/\log B),\label{eq:m-1-STS}\\
\mathcal{H}_4(\vec{q}) & \le O(m).\label{eq:m-2-STS}    
\end{align}
Now, let $\theta=m/|\vec q|$ (so $|\theta\vec q|=m$). 
In Step~3 of the proof of \cref{thm:main-m-upper} (in \cref{subsec:uniform-random-upper}), we showed how to use \cref{eq:m-1-STS} and \cref{eq:m-2-STS} to prove that 
\begin{equation}
    I_B(\vec{q})\ge I_B(\theta\vec q)-(\varepsilon/4)m\label{eq:previous-step-3}
\end{equation} (we proved this with ``$\varepsilon/2$'' in place of ``$\varepsilon/4$'', but the proof is the same).

We also need an additional estimate for $\sum_{g\in G}\chi(\ell_g(\vec q))$. For this, we first use \cref{eq:m-1-STS} to see that $|G|^{-1}\sum_{g\in G}\ell_g(\vec q)=|\vec q|/m=1-O(1/\log B)$. Then, note that $\chi$ is convex and continuous on $[0,1]$, so by Jensen's inequality,
\begin{equation}
    \sum_{g\in G}\chi(\ell_g(\vec q))\ge |G|\chi(|\vec q|/m)=|G|(1+o_{B\to \infty }(1))\ge 3m-(\varepsilon/4)m\label{eq:new-step-STS}
\end{equation}
for large $B$. The desired result \cref{eq:goal-STS-binomial-reduction} then follows from \cref{eq:new-step-STS,eq:previous-step-3}.
\end{proof}

\subsection{Analytic inequalities for the variational problem}
It remains to address the variational problem in \cref{lem:STS-reduction-m}, in both cases of \cref{def:STS-H}. It would be possible to first reduce to a continuum problem as in \cref{cor:main-graphs}, but this is not necessary; we can work directly with weightings $\vec q\in [0,\infty)^K$.

We start with some analytic inequalities. Throughout the rest of this section, we use notation like $\mb E_x$ to denote an average over $x\in [n]$.

First, the following lemma is essentially the well-known fact that
$K_{2,2,2}^{(3)}$ satisfies Sidorenko's conjecture (see e.g.\ \cite{Wig21}); we include a proof for completeness. This lemma is relevant
because $K_{2,2,2}^{(3)}$ can be partitioned into two hyperedge-disjoint
Pasch configurations.
\begin{lemma}
\label{lem:sidorenko-K222}For any $u:[n]^{3}\to[0,\infty)$, we have
\[
\mb E_{x_{0},y_{0},z_{0},x_{1},y_{1},z_{1}}\prod_{i,j,k\in\{0,1\}}u(x_{i},y_{j},z_{k})\ge(\mb E_{x,y,z}u(x,y,z))^{8}.
\]
\end{lemma}

\begin{proof}
The proof is just by iterated application of the Cauchy--Schwarz
inequality (over $x$, then $y$, then $z$). First,
\[
\left(\mb E_{x,y,z}u(x,y,z)\right)^{2}\le\mb E_{y,z}\left(\mb E_{x}u(x,y,z)\right)^{2}=\mb E_{x_{0},x_{1},y,z}\prod_{i\in\{0,1\}}u(x_{i},y,z).
\]
Then
\[
\left(\mb E_{x_{0},x_{1},y,z}\prod_{i\in\{0,1\}}u(x_{i},y,z)\right)^{2}\le\mb E_{x_{0},x_{1},z}\left(\mb E_{y}\prod_{i\in\{0,1\}}u(x_{i},y,z)\right)^{2}=\mb E_{x_{0},x_{1},z}\mb E_{y_{0},y_{1}}\prod_{i,j\in\{0,1\}}u(x_{i},y_{j},z).
\]
Finally,
\begin{align*}
\left(\mb E_{x_{0},x_{1},y_{0},y_{1},z}\prod_{i,j\in\{0,1\}}u(x_{i},y_{j},z)\right)^{2}&\le\mb E_{x_{0},x_{1},y_{0},y_{1}}\left(\mb E_{z}\prod_{i,j\in\{0,1\}}u(x_{i},y_{j},z)\right)^{2}\\&=\mb E_{x_{0},x_{1},y_{0},y_{1},z_{0},z_{1}}\prod_{i,j,k\in\{0,1\}}u(x_{i},y_{j},z_{k}).
\end{align*}
The desired result follows.
\end{proof}

\begin{lemma}\label{lem:STS-analytic}
Consider $w:[n]^{3}\to[0,\infty)$ satisfying $\mb E_{x,y,z}w(x,y,z)=1$.
Then
\[
\mb E_{x,y,z}\big[w(x,y,z)\log w(x,y,z)\big]+\frac{1}{4}\mb E_{*}\big[w(x_{0},y_{0},z_{0})w(x_{1},y_{1},z_{0})w(x_{1},y_{0},z_{1})w(x_{0},y_{1},z_{1})\big]\ge\frac{1}{4},
\]
where we write ``$\mb E_{*}$'' as shorthand for ``$\mb E_{x_{0},y_{0},z_{0},x_{1},y_{1},z_{1}}$''.
\end{lemma}

\begin{proof}
First, let 
\[
W=\prod_{i,j,k\in\{0,1\}}w(x_{i},y_{j},z_{k}),\quad W_{0}=\prod_{\substack{i,j,k\in\{0,1\}\\
i+j+k\text{ even}
}
}w(x_{i},y_{j},z_{k}),\quad W_{1}=\prod_{\substack{i,j,k\in\{0,1\}\\
i+j+k\text{ odd}
}
}w(x_{i},y_{j},z_{k}).
\]
By the Cauchy--Schwarz inequality and symmetry, we have
\[
\mb E_{*}\sqrt{W}=\mb E_{*}\sqrt{W_{0}W_{1}}\le\sqrt{\mb E_{*}W_{0}\mb E_{*}W_{1}}=\mb E_{*}W_{0}.
\]
So, by \cref{lem:sidorenko-K222}, we have 
\[
(\mb E_{*}W_{0})^{1/8}\ge\mb E_{x,y,z}\sqrt{w(x,y,z)}.
\]

Then, using Jensen's inequality with respect to the probability measure
$(x,y,z)\mapsto w(x,y,z)/n^{3}$ (recalling that $\mb E_{x,y,z}w(x,y,z)=1$), we have
\begin{align*}
\log\mb E_{x,y,z}\sqrt{w(x,y,z)}=\log\mb E_{x,y,z}\left[w(x,y,z)w(x,y,z)^{-1/2}\right]&\ge\mb E_{x,y,z}\left[w(x,y,z)\log(w(x,y,z)^{-1/2})\right]\\
&=-\frac{1}{2}\mb E_{x,y,z}\big[w(x,y,z)\log w(x,y,z)\big].
\end{align*}
(This is essentially monotonicity of R\'enyi divergence; see \cite{Ren60}). Combining these inequalities yields
\[
\mb E_{x,y,z}\big[w(x,y,z)\log w(x,y,z)\big]\ge-2\log\mb E_{x,y,z}\sqrt{w(x,y,z)}\ge-\frac{1}{4}\log\mb E_{*}W_{1},
\]
so
\[
\mb E_{x,y,z}\big[w(x,y,z)\log w(x,y,z)\big]+\frac{1}{4}\mb E_{*}W_{1}\ge\frac{1}{4}\left(-\log\mb E_{*}W_{1}+\mb E_{*}W_{1}\right)\ge\frac{1}{4},
\]
using that 1 is the minimum value of $-\log u+u$, over $u\in[0,\infty)$.
\end{proof}

\subsection{Concluding the proofs}
Now we conclude the proofs of \cref{thm:STS,thm:latin}.

\begin{proof}
    [Proof of \cref{thm:STS}]
    We only need to prove the upper bound, as the lower bound was already proved by Kwan, Sah, Sawhney and Simkin~\cite[Theorem~1.3]{KSSS24}. We work in setting (1) of \cref{def:STS-H}.

Let $\vec q=(q_{xyz})_{xyz\in K_n^{(3)}}$ be a weighting which minimises $\on D_{\mr{KL}}(\vec{q}\,\|\,m/N)+\mathcal{H}_{4}(\vec{q})$, over all such weightings with sum of entries exactly $m$. By \cref{lem:truncation-lower} we have $q_{xyz}\le O(m/N)=O(1/n)$ for all $xyz\in K_n^{(3)}$. We compute
\begin{align*}
\on D_{\mr{KL}}(\vec{q}\,\|\,m/N)+\mathcal{H}_{4}(\vec{q}) & =
\sum_{xyz\in K_n^{(3)}}q_{xyz}\log\left(\frac{q_{xyz}}{m/N}\right)-\sum_{xyz\in K_n^{(3)}}q_{xyz}+\sum_{xyz\in K_n^{(3)}}\frac m N\\
&\qquad+\frac{1}{24}\sum_{\substack{x_0,x_1,y_0,y_1,z_0,z_1\in \on V(K_{n}^{(3)})\\\text{distinct}}
}q_{x_0y_0z_0}q_{x_1y_1z_0}q_{x_1y_0z_1}q_{x_0y_1z_1}+o(m).
\end{align*}
Here we have used that a Pasch configuration has 24 automorphisms (we can embed a Pasch configuration as a subgraph of $K^{(3)}_{2,2,2}$, in such a way that half of the $3!\cdot2^3=48$ automorphisms of $K^{(3)}_{2,2,2}$ are also automorphisms of the Pasch configuration). We have also used that the contribution of ``degenerate'' Pasch configurations (in which two of $x_0,x_1,y_0,y_1,z_0,z_1$ coincide) is at most $O(n^{5})\cdot O(n^{-4})=o(m)$. In addition, we have used the elementary estimate $\log(1-q)=-q+O(q^2)$.

Define $w:[n]^3\to [0,\infty)$ by letting $w(x,y,z)=(n^3/(6m))q_{xyz}$ for distinct $x,y,z$, and $w(x,y,z)=0$ if two of $x,y,z$ coincide. So, $\mb E_{x,y,z}w(x,y,z)=1$. Recalling that $m=n^2/6+o(n^2)$ and $N=n^3/6+o(n^3)$, we can rewrite the above expression as
\[m\left(
\mb E_{x,y,z}\big[w(x,y,z)\log w(x,y,z)\big]+\frac{1}{4}\mb E_{*}\big[w(x_{0},y_{0},z_{0})w(x_{1},y_{1},z_{0})w(x_{1},y_{0},z_{1})w(x_{0},y_{1},z_{1})\big]\right)+o(m).\]
The desired result then follows from \cref{lem:STS-reduction-m,lem:STS-analytic}.
\end{proof}
\begin{proof}
    [Proof of \cref{thm:latin}]
    We only need to prove the upper bound, as the lower bound was already proved by Kwan, Sah, Sawhney and Simkin~\cite[Theorem~1.1]{KSSS23}. We work in setting (2) of \cref{def:STS-H}.

Let $\vec q=(q_{xyz})_{xyz\in K_{n,n,n}^{(3)}}$ be a weighting which minimises $\on D_{\mr{KL}}(\vec{q}\,\|\,m/N)+\mathcal{H}_{4}(\vec{q})$, over all such weightings with sum of entries exactly $m$. By \cref{lem:truncation-lower} we have $q_{xyz}\le O(m/N)=O(1/n)$ for all $xyz\in K_{n,n,n}^{(3)}$. The same estimates as in the previous proof yield
\begin{align*}
\on D_{\mr{KL}}(\vec{q}\,\|\,m/N)+\mathcal{H}_{4}(\vec{q}) & \ge
\sum_{x,y,z\in [n]}q_{xyz}\log\left(\frac{q_{x,y,z}}{m/N}\right)-\sum_{x,y,z\in [n]}q_{xyz}+\sum_{x,y,z\in [n]}\frac m N\\
&\qquad+\frac{1}{4}\sum_{x_0,x_1,y_0,y_1,z_0,z_1\in [n]
}q_{x_0y_0z_0}q_{x_1y_1z_0}q_{x_1y_0z_1}q_{x_0y_1z_1}+o(m),
\end{align*}
where we have used that there are $8/2=4$ automorphisms of a Pasch configuration that respect its natural tripartition.

Letting $w(x,y,z)=q_{x,y,z}/(m/N)$ (and recalling that $m=n^2$ and $N=n^3$), we can rewrite the above expression as
\[m\left(
\mb E_{x,y,z}\big[w(x,y,z)\log w(x,y,z)\big]+\frac{1}{4}\mb E_{*}\big[w(x_{0},y_{0},z_{0})w(x_{1},y_{1},z_{0})w(x_{1},y_{0},z_{1})w(x_{0},y_{1},z_{1})\big]\right)+o(m).\]
The desired result then follows from \cref{lem:STS-reduction-m,lem:STS-analytic}.
\end{proof}

\bibliographystyle{amsplain_initials_nobysame_nomr}
\bibliography{main}

\appendix
\section{Proof of the Kozma--Samotij product decomposition theorem}\label{app:KS}
In this appendix we explain how to prove \cref{thm:KS}, by making minor adaptations to the proof of \cite[Lemma~18]{KS23}\footnote{In this appendix we refer to the published version (not the arXiv version) of \cite{KS23}. There are some slight numbering differences; for example, in the arXiv version this lemma appears as Lemma~19.}. The main point is to demonstrate that $W$ does not need to depend on $\mc G$, but also we have normalised the error terms in a slightly different way.

We first recall some notation from \cite{KS23}. For a (nonnegatively) edge-weighted hypergraph $\mc G$ on the vertex set $V$, we write
\[
 e(\mc G)=\sum_{A\in\mc G}\on{wt}_{\mc G}(A),\qquad
 \deg_{\mc G}(B)=\sum_{\substack{A\in\mc G:\\B\subseteq A}}
 \on{wt}_{\mc G}(A),
\]
and let $\Delta_s(\mc G)=\max_{|B|=s}\deg_{\mc G}(B)$.  
Let $Y$ be a sample from $\on{Bin}(V,p)$ conditioned on a nontrivial decreasing property, and for a fixed set $W\subseteq V$ write
\[
 Y_v=\one_{v\in Y}, \qquad Y_A=\prod_{v\in A}Y_v,\qquad
 \mb E_W[\,\cdot\,]=\mb E[\,\cdot\mid (Y_w)_{w\in W}].
\]
Also, for $B\subseteq V\setminus W$ and $b\in B$, let
\[
 D_W(B,b)=\mb E_W[Y_B]-\mb E_W[Y_{B\setminus b}]\mb E_W[Y_b].
\]
Then, if $\mc G$ is $k$-uniform, let
\begin{equation*}
 \mc E_{\mc G}(W)=
 \mb E\left[
 \sum_{A\in\mc G[V\setminus W]}
 \sum_{\substack{B\subseteq A:\\|B|\ge2}}
 \sum_{b\in B}
 \frac{\on{wt}_{\mc G}(A)D_W(B,b)^2}
 {\binom{k}{|B|}|B|p^{2|B|}}
 \right].
\end{equation*}

It is also convenient to define $\mf G(C_0,\theta)$ to be the collection of all $k$-uniform weighted hypergraphs, with $k\le r_0$, satisfying \cref{eq:KS-degree-condition}, and to define $\mf G_{\mr{KS}}(K,\theta)$ to be the collection of all $k$-uniform weighted hypergraphs, with $k\le r_0$, satisfying the slightly different condition
    \begin{equation}
        \Delta_s(\mc G) \le K (\theta p)^{s-1} \frac{e(\mc G)}{N}\text{ for all }s\in \{1,\dots,k\},\label{eq:KS-degree-condition-original}.
    \end{equation}

The following lemma is a simple application of the Cauchy--Schwarz inequality; it appears as 
\cite[(31)--(33)]{KS23}.

\begin{lemma}\label{lem:KS-energy-controls-products}
Let $\mc G$ be an edge-weighted $k$-uniform hypergraph.  Then,
for every $W\subseteq V$,
\[
 \mb E\Bigg[
 \sum_{A\in\mc G[V\setminus W]}\on{wt}_{\mc G}(A)
 \bigg|\mb E_W[Y_A]-\prod_{v\in A}\mb E_W[Y_v]\bigg|
 \Bigg]\le
 p^k\big((k-1)e(\mc G[V\setminus W])\big)^{1/2}\mc E_{\mc G}(W)^{1/2}.
\]
\end{lemma}

\subsection{A simultaneous entropy-increment lemma}
The key ingredient in the proof of \cref{thm:KS} is a generalisation of \cite[Lemma~18]{KS23} that can handle multiple hypergraphs simultaneously. It has virtually exactly the same proof as \cite[Lemma~18]{KS23}.

\begin{lemma}\label{lem:KS-simultaneous-regularisation}
Consider $K,r_0,\alpha,\beta>0$. There exists $\theta>0$ and
$N_0\in\mb N$ such that the following holds.  Let $p\in(0,1/2)$ and $N=|V|\ge N_0$ and let $Y$ be a sample from $\on{Bin}(V,p)$, conditioned on a nonempty decreasing event $\mc D$.  Then there is $W\subseteq V$ with $|W|\le \alpha N$
such that
\[
 \mc E_{\mc G}(W)\le\beta e(\mc G)
 \qquad\text{for every }\mc G\in \mf G_{\mr{KS}}(K,\theta).
\]
\end{lemma}

\begin{proof}
We may assume that $\alpha<1/2$, $\beta<1$ and $K>1$. As in the proof of \cite[Lemma~18]{KS23}, define the parameters
\begin{equation}\label{eq:KS-parameter-choice}
 \gamma=\frac{\beta^2}{300Kr_0},
 \qquad
 \tau=\frac{\alpha\gamma}2,
 \qquad
 \theta=\frac{\tau}{2r_0},\qquad N_0=8r_0^2/\tau
\end{equation}
Let
$i_p(q)=q\log(q/p)+(1-q)\log((1-q)/(1-p))$, and for $W\subseteq V$ let
\begin{equation}\label{eq:KS-potential-small-p}
 \mc I(W)=\sum_{v\in V\setminus W}\mb E\,i_p(\mb E_WY_v).
\end{equation}
Note that, writing $R_p\sim\on{Bin}(V,p)$ (and recalling $p\le 1/2$), we have
\begin{equation}
 \mc I(W)\le
 \on{D}_{\mr{KL}}\big(\mc L(Y)\,\|\,\on{Bin}(V,p)\big)
 =-\log\Pr[R_p\in \mc D]\le -\log\Pr[R_p=\emptyset]=N\log(1/(1-p))\le 2pN. \label{eq:I(W)-KS}
\end{equation}

We will build our set $W$ in several rounds, starting with $W$ being the empty set. In each round we will use the following claim, which appears as \cite[Claim~19]{KS23}.

\begin{claim}\label{clm:KS-increment}
    Fix an edge-weighted hypergraph $\mc G\in \mf G_{\mr{KS}}(K,\theta)$. Suppose that $W\subseteq V$ satisfies $\mc E_\mc G(W)> \beta e(\mc G)$, and $|V\setminus W|\ge N_0/2$. Then there is $W'\subseteq V\setminus W$ with size at most $\tau|V|$ such that 
    \[\mc I(W\cup W')\ge \mc I(W)+\gamma p|V|.\]
\end{claim}

Now, we show how to use \cref{clm:KS-increment} to complete the proof of \cref{lem:KS-simultaneous-regularisation}. Suppose for the purpose of contradiction that there is no set $W$ with at most $\alpha N$ elements  such that $\mc E_\mc G(W)\le \beta e(\mc G)$ for every $\mc G\in \mf G_{\mr{KS}}(K,\theta)$. Let $j=\lfloor\alpha/\tau\rfloor+1$; we construct sets $W_0,\dots,W_j$ such that, for every
$i\in\{0,\dots,j\}$,
\begin{equation}\label{eq:KS-iteration-properties}
 |W_i|\le i\tau N
 \qquad\text{and}\qquad
 \mc I(W_i)\ge i\gamma pN.
\end{equation}
If such a sequence existed, we would have
\[\mc I(W_j)\ge j\cdot \gamma p|V|,\]which would contradict \cref{eq:I(W)-KS}.

Start with $W_0=\emptyset$.  Suppose that $i\le j-1$ and $W_i$ has been
constructed to satisfy the desired properties.  Then
\[
 |W_i|\le i\tau N
 \le\lfloor\alpha/\tau\rfloor\tau N
 \le\alpha N,
\]
so, by assumption, there is some $\mc G_i\in \mf G_{\mr{KS}}(K,\theta)$ such that $\mc E_{\mc G_i}(W_i)>\beta e(\mc G_i)$.  Moreover, $
 |V\setminus W_i|\ge(1-\alpha)N>N/2\ge N_0/2
$.
Let $W'\subseteq V\setminus W_i$ be the set obtained by applying \cref{clm:KS-increment} with $W_i$ and $\mc G_i$, and put
$W_{i+1}=W_i\cup W'$.  Then \cref{eq:KS-iteration-properties} continues
to hold with $i+1$ in place of $i$.  This proves the lemma.
\end{proof}

\subsection{Proof of the product decomposition theorem}

Now we complete the proof of \cref{thm:KS}. 
\begin{proof}[Proof of \cref{thm:KS}]Let $K$ be large in terms of $C_0,\xi$, and let $\beta$ be small in terms of $C_0,\xi$. Let $\theta,N_0$ be as in \cref{lem:KS-simultaneous-regularisation}.

Choose $W\subseteq V$ with $|W|\le\alpha N$, such that $\mc E_{\mc G}(W)\le\beta e(\mc G)$ for every $\mc G\in \mf G_{\mr{KS}}(K,\theta)$. We claim that the desired condition holds for this $W$.

For a weighted $k$-uniform $\mc G\in\mf G(C_0,\theta)$, call $\mc G$
\emph{relevant} if 
$p^k e(\mc G)>(\xi/2)N$. 
First, observe that for every $e\in \mc G\in \mf G(C_0,\theta)$ we have
\begin{equation*}
 \mb E\bigg|\mb E_W[Y_e]-\prod_{v\in e}\mb E_W[Y_v]\bigg|
 \le2p^{|e|}
\end{equation*}
by \cref{lem:harris}. So, for every $k$-uniform non-relevant $\mc G\in \mf G(C_0,\theta)$ we have
\[
 \mb E\Bigg[
 \sum_{e\in\mc G[V\setminus W]}\on{wt}_{\mc G}(e)
 \bigg|\mb E_W[Y_e]-\prod_{v\in e}\mb E_W[Y_v]\bigg|
 \Bigg]
 \le2p^k e(\mc G)\le\xi N.
\]
On the other hand, for any $k$-uniform relevant $\mc G\in \mf G(C_0,\theta)$, we have $\mc G\in \mf G_{\mr{KS}}(K,\theta)$ (since $K$ is large in terms of $C_0,\xi$), so we can apply the conclusion of \cref{lem:KS-simultaneous-regularisation}. Note that
\[
 k e(\mc G)=\sum_{v\in V}\deg_{\mc G}(v)
 \le N\Delta_1(\mc G)\le C_0Np^{-k},
\]
so \cref{lem:KS-energy-controls-products} and the conclusion of \cref{lem:KS-simultaneous-regularisation} give
\[
 \mb E\Bigg[
 \sum_{A\in\mc G[V\setminus W]}\on{wt}_{\mc G}(A)
 \bigg|\mb E_W[Y_A]-\prod_{v\in A}\mb E_W[Y_v]\bigg|
 \Bigg]\le p^ke(\mc G)\sqrt{(k-1)\beta}\le \frac{C_0\sqrt{(k-1)\beta}}k N\le \xi N
\]
(recalling that $\beta$ is small in terms of $C_0,\xi$). So, both non-relevant and relevant $\mc G\in \mf G(C_0,\theta)$ satisfy the desired inequality.
\end{proof}
\section{The numerical certificate function}\label{app:numerical}
For the curious reader, we describe the functions $G_\gamma$ that are used to prove \cref{prop:numerical}. We do not actually prove here that these functions $G_\gamma$ satisfy the inequality in \cref{prop:numerical} (a Lean proof, involving lots of interval arithmetic, can be found at \url{https://palomar-registry.org/entry?id=PALOMAR-2026-10-19-000018}).

$G_\gamma$ is defined in one of three different ways depending on the value of\footnote{The parameter $\gamma$ is called \texttt{\textalpha{}} in the Lean formalisation.} $\gamma$ (specifically, depending on whether $\gamma$ lies in a small window around the critical point, or is before this window, or after it). The definitions involve quite a lot of setup, much of which is for the purpose of conveniently describing the different branches of the stationary points depicted in \cref{fig:triangle-stationary}.

First, let\footnote{This function $g$ is called \texttt{H} in the Lean formalisation.} $g(w)=1-w+w\log w$, and let $s_\gamma\in(0,1]$ be the unique solution of $\log s_\gamma+3\gamma s_\gamma^2=0$. This corresponds to the blue branch in \cref{fig:triangle-stationary} (which exists for all $\gamma$).

Then, fix $h_0=939/100$ and $h_1=301/32$.
For $h\ge h_0$, let
\[
 z_h=e^{-h},\quad p_h=\frac{2hz_h}{(1-z_h)^2},\quad t_h=e^{-p_h},\quad r_h=t_hz_h,\quad
 \Delta_h=t_h-r_h,\quad \gamma(h)=\frac{2h}{3\Delta_h^2},
\]
and define $\gamma_0=\gamma(h_0)$. One can check that $\gamma(h)$ describes a bijective correspondence between $[h_0,\infty)$ and $[\gamma_0,\infty)$. Under this correspondence, $(r_h,t_h)$ corresponds to the green branch in \cref{fig:triangle-stationary}.

Next, let
\[
 \beta_h=\frac{3\gamma(h)r_h}{2},\qquad C_h=1-(1+p_h/2)t_h,\qquad
 X_h(w)=\frac{g(w)+\beta_h w^2-C_h}{\Delta_h}, \qquad y_h=X_h(s_{\gamma(h)}),
\]
and define \[\theta(h)=\begin{cases}
    (y_h^2-3y_h+1)/(y_h(1-y_h)^2)&\text{for }h_0\le h\le h_1,\\
    0&\text{for }h>h_1,
\end{cases}\qquad 
\eta_\theta(X)=3X-X^2-1+\theta X(1-X)^2,
\]
and
\[
 \hat G_h(w)=\sqrt{\frac{\Delta_h}6}\,\eta_{\theta(h)}(X_h(w)).
\]
For $0<\gamma\le \gamma_0$, we define
\[
 G_\gamma(w)=\sqrt{\frac{\gamma s_\gamma^3}{\gamma_0s_{\gamma_0}^3}}\,
 \hat G_{h_0}\left(\frac{s_{\gamma_0}}{s_\gamma}w\right).
\]
For $\gamma>\gamma_0$, let $h\ge h_0$ be such that $\gamma(h)=\gamma$ and take $G_\gamma=\hat G_h$.

Note that the definitions of $G_\gamma$ were chosen to ensure that the desired inequality is sharp at certain necessary points. Indeed, let\footnote{This function $F_\gamma$ is called \texttt{entropyF} in the Lean formalisation.}
\[F_\gamma(a,b,c)=\frac{g(a)+g(b)+g(c)}3+\gamma abc-G_{\gamma}(a)G_{\gamma}(b)-G_{\gamma}(a)G_{\gamma}(c)-G_{\gamma}(b)G_{\gamma}(c)\]
and
\[
 M_1(\gamma)=g(s_\gamma)+\gamma s_\gamma^3,
 \qquad
 \hat M_2(h)=\frac{g(r_h)+g(t_h)}2
          +\gamma(h)\bigg(\frac14r_h^3+\frac 34r_ht_h^2\bigg).
\]
The definitions of $G_\gamma$ were chosen such that:
\begin{itemize}
    \item for $\gamma\le \gamma(h_1)$, we have $F_\gamma(s_\gamma,s_\gamma,s_\gamma)=M_1(\gamma)$ and $\nabla F_\gamma(s_\gamma,s_\gamma,s_\gamma)=0$
    \item for $h\ge h_0$, we have $F_{\gamma(h)}(a,b,c)=\hat M_2(h)$ and $\nabla F_{\gamma(h)}(a,b,c)=0$ whenever $(a,b,c)=(r_h,r_h,r_h)$ or $(a,b,c)$ is a permutation of $(r_h,t_h,t_h)$.
\end{itemize}
(Note that these two conditions overlap when $h_0\le h\le h_1$.)

\section{Analysing a semi-random process}\label{app:semi-random}
In this appendix we explain how to deduce \cref{thm:GW-lower-bound} from the analysis in \cite{GW}. We note that \cite{GW} is an unpublished draft of a paper that is in the process of being finalised, and the authors (He Guo and Lutz Warnke) intend to include additional theorems in the final version of their paper which will hopefully render this appendix unnecessary.

\subsection{Description of the algorithm}
We start by briefly describing the semi-random algorithm in \cite{GW}, applied to our hypergraph $\mc G$. Recall the role of $\nu,C$ in \cref{thm:GW-lower-bound}, write $f=N^{\nu/2}$, consider any $D$ satisfying $\max(1,C)\cdot \alpha^{1-r}\le D=O(\alpha^{1-r}\log f)$, and let\footnote{We are using the same notation as in \cite{GW}, except that the parameter we have called ``$\ell$'' here is called ``$m$'' in \cite{GW}. 
Their analysis works as long as $f\ge (\log N)^c$ for some $c$ large in terms of $r,\beta$, which gives us plenty of room to spare.}
\[\beta=\nu/2,\quad \sigma=(\log f)^{-9r},\quad \rho=1/2,\quad\xi=\left(\frac \beta{37r}\right)^{1/(r-1)},\quad \ell=\lfloor\xi(\log f)^{1/(r-1)}/\sigma\rfloor.\]
For each $0\le i\le \ell$, let
\[\pi_i=\sigma+\sigma i,\quad q_i=\exp(-(\sigma i)^{r-1}),\quad p_i=\frac{\sigma
}{q_i D^{1/(r-1)}}.\]

The semi-random algorithm runs for $\ell$ rounds. It builds a sequence of independent sets $\emptyset=I_0\subseteq \dots\subseteq I_\ell$, contained inside vertex sets $\emptyset =V_0\subseteq\dots\subseteq V_\ell$ (i.e., $I_i\subseteq V_i$ for all $i$). At each round $i+1$ there is a set of ``available'' vertices $A_i$ from the previous round: this is a set of vertices which can be added to $I_i$ without introducing a hyperedge (in particular, $A_0=\on V(\mc G)$).

At the start of the round, we let $\Gamma_{i+1}\subseteq A_i$ be a random set of vertices, obtained by including each vertex in $A_i$ independently with probability $p_i$, and let $V_{i+1}=V_i\cup \Gamma_{i+1}$. Then, we let $I_{i+1}$ be a certain\footnote{In \cite{GW}, the authors describe precisely how $I_{i+1}$ should be chosen, but these details are not important for our purposes.} independent set satisfying $I_{i}\subseteq I_{i+1}\subseteq I_i\cup \Gamma_{i+1}$. (So, $V_{i+1}$ records all the vertices which we have ever attempted to add into our independent set. It turns out that we can typically include almost all the vertices of $\Gamma_{i+1}$ in $I_{i+1}$, so $V_{i+1}$ is a good approximation for $I_{i+1}$).

Then, to finish the round we need to define $A_{i+1}$, the set of available vertices for the next round. This involves deleting all vertices whose addition would cause a hyperedge together with some vertices of $V_{i+1}$, and also randomly deleting some vertices with ``stabilising coin flips'', designed to keep all statistics on a manageable trajectory. Specifically, let
\[Y_{v,j}(i)=\{W\subseteq A_0\setminus \{v\}:|W|=j\text{ and there is }e\in \mc G\text{ with }v\in e\text{ and }W\subseteq e\cap A_i\text{ and }e\setminus(\{v\}\cup W)\subseteq V_i\}\]
be the collection of all $j$-sets of alive vertices which form a hyperedge together with $v$ and $r-(j+1)$ previously chosen vertices. For convenience, we write $Y_{v}(i)=Y_{v,1}(i)$. Then, let
\[K_i=(r-1)q_i(\pi_i^{r-2}+\sigma^\rho\pi_i^{\max(r-3,0)})D^{1/(r-1)}\]
and let $S_{i+1}$ be a random subset of $A_i$ in which each vertex $v\in A_i$ is included in $S_{i+1}$ independently with probability
$\hat p_{v,i}:=1-(1-p_i)^{\max(K_i-|Y_{v}(i)|,0)}$. Then, let
\[C_{i+1}^{(j)}=\{v\in A_i:\text{there is }W\in Y_{v,j}(i)\text{ with }W\subseteq \Gamma_{i+1}\},\]
and let
\[A_{i+1}=A_i\setminus \Big(\Gamma_{i+1}\cup S_{i+1}\cup \bigcup_{j=1}^{r-1}C_{i+1}^{(j)}\Big).\]
\subsection{Typical statistics}
First, we need to know the typical size of $I_\ell$ and each $\mc F_i[I_\ell]$. This follows from the second moment method and the control over probabilities of the form $\Pr[S\subseteq I_\ell]$ provided by \cite[Theorem~5]{GW} (this is already done in \cite[Corollary~6]{GW} for $k_i\ge 2$).
\begin{lemma}\label{lem:app-typical-size}
With probability $1-o(1)$, the following holds. First, we have
\[|I_\ell|=(1+o(1)) N\xi\Big(\frac{\log f}D\Big)^{1/(r-1)},\]
and, second, for each test hypergraph $\mc F_i$ we have
\[|\mc F_i[I_\ell]|=(1+o(1)) |\mc F_i|\bigg(\xi\Big(\frac{\log f}D\Big)^{1/(r-1)}\bigg)^{k_i}.\]
\end{lemma}

We also need to know that $V_\ell=\Gamma_1\cup \dots\cup \Gamma_\ell$ is a good approximation for $I_\ell$. This is proved in \cite[Lemma~34]{GW}.
\begin{lemma}\label{lem:Gamma-approx-I}
Let $\mc E_{\mr{\ref{lem:Gamma-approx-I}}}$ be the event that $|V_\ell|-|I_\ell|\le |I_\ell|/\log f$. Then $\Pr[\mc E_{\mr{\ref{lem:Gamma-approx-I}}}]=1-o(1)$.
\end{lemma}

Also, in \cite{GW}, the authors define ``good'' events $\mc X_1,\dots,\mc X_\ell$, under which various statistics follow well-behaved trajectories.
\begin{lemma}\label{lem:GW-good-events}
    There are events $\mc X_1,\dots,\mc X_\ell$, such that each $\mc X_i$ only depends on the first $i$ rounds of the process, such that with probability $1-o(1)$ all the events $\mc X_1,\dots,\mc X_\ell$ hold, and such that:
    \begin{enumerate}
        \item If $\mc X_1\cap\dots\cap  \mc X_i$ holds, then $K_i\ge |Y_{v}(i)|$ for all $v\in A_i$ (i.e., the ``$\max$'' in the definition of the $\hat p_{v,i}$ is not necessary)
        \item If $\mc X_1\cap\dots\cap \mc X_i$ holds, then $|A_i|\ge (1-1/\log f)q_iN$.
        \item If we condition on an outcome of the first $i$ rounds of the process such that $\mc X_1\cap\dots\cap \mc X_i$ holds, then for each $v\in A_i$ we conditionally have $v\notin C_{i+1}^{(1)}\cup S_{i+1}$ with probability at most $q_{i+1}/q_i$.
    \end{enumerate}
\end{lemma}
Specifically, in \cite[Section~2.5]{GW}, the authors break down $\mc X_i$ into various sub-events (and state \cite[Theorem~16]{GW}, which says that these events hold together with probability $1-o(1)$). Property (1) above follows from the sub-event ``$\mc Y_i$'', and property (2) follows from the sub-event ``$\mc A_i$''. Then, (3) follows from \cite[Lemma~27]{GW}.

Finally, we need to know the typical sizes of the sets $Y_{v}(i)$. 
\begin{lemma}\label{lem:app-Y-size}
Let $\chi_i=(r-1)q_i\pi_i^{r-2}D^{-(r-2)/(r-1)}$, and let $\mc E_{\mr{\ref{lem:app-Y-size}}}$ be the event that \[|Y_v(i)|\le \chi_i(\deg(v)+D/(\log f)^2),\]
for all $i\le \ell$ and $v\in A_{i}$.
Then $\Pr[\mc E_{\mr{\ref{lem:app-Y-size}}}]=1-o(1)$.
\end{lemma}
Unfortunately, \cref{lem:app-Y-size} cannot be directly deduced from named lemmas in \cite{GW}: the event $\mc X_i$ does imply an upper bound on $Y_v(i)$, but it is not strong enough for our purposes. One needs to slightly adapt the proof of this bound; we briefly sketch how to do this\footnote{We would like to thank Lutz Warnke for initially confirming with us that this adaptation is possible.}.
\begin{proof}[Proof sketch of \cref{lem:app-Y-size}]
    In \cite[Section~2.5]{GW} the authors define events $\mc Y_i\supseteq \mc X_i$, which imply upper bounds on $|Y_v(i)|$. Namely, $\mc Y_i$ is the event that $|Y_{v,j}(i)|\le B_j(i)$ for all $v\in A_0$ and\footnote{Note that for the lemma statement we are only interested in $j=1$, but it turns out one needs to consider all $j$ in order to use $\mc Y_i$ to control $\mc Y_{i+1}$.} $1\le j\le r-1$, where $B_j(i)=\binom{r-1}{j}q_i^j\pi_i^{r-1-j}D^{j/(r-1)}$. In \cite[Theorem~31]{GW}, the authors show that these events $\mc Y_i$ are likely to hold: more precisely, given any outcome of the first $i$ steps of the process satisfying $\mc X_i\subseteq \mc Y_i$, the event $\mc N_{i+1}\setminus \mc Y_{i+1}$ is very unlikely, where $\mc N_{i+1}$ is an event that is separately shown to hold with very high probability. (So, one can take a union bound over all $i$).

    We simply need to introduce a sharper version of the event $\mc Y_i$. Let $\varepsilon=(\log f)^{-2}$ and $a_v=\deg(v)/D+\varepsilon$, and let $\mc Z_i$ be the event that $|Y_{v,j}(i)|\le a_vB_j(i)$ for all $v\in A_0$ and $1\le j\le r-1$. By definition $\mc Z_0$ holds with probability 1; we just need to adapt the proof of \cite[Theorem~31]{GW} to show that given any outcome of the first $i$ steps of the process satisfying $\mc X_i\cap \mc Z_i$, the event $\mc N_{i+1}\setminus \mc Z_{i+1}$ is very unlikely.

    This is basically just a matter of being more careful with the calculations in the proof of \cite[Theorem~31]{GW}. Briefly: the proof of \cite[Theorem~31]{GW} proceeds by first estimating the conditional expected value of $|Y_{v,j}(i+1)|$ given an outcome of the first $i$ steps of the process satisfying $\mc X_i$ (this uses the inequalities of the form $|Y_{v,j}(i)|\le B_j(i)$ from $\mc Y_i$), then using a concentration inequality (this uses the event $\mc N_{i+1}$). If we instead use the inequalities of the form $|Y_{v,j}(i)|\le a_vB_j(i)$ from $\mc Z_i$, we obtain a sharper bound on the conditional expected value of $|Y_{v,j}(i+1)|$: namely, essentially the same calculations show that it is at most $a_vB_j(i+1)-r a_vt$, where (as in \cite[Eq.\ (62)]{GW})
\[
t=t_{i,j}
=\frac{1}{16r}\binom{r-1}{j}
  \sigma^{1+\rho}q_{i+1}^{j}
  \pi_i^{r-1-j+\max\{r-3,0\}}
  D^{j/(r-1)}.
\]
Then, we can finish with the same concentration argument as in the proof of \cite[Theorem~31]{GW} (with deviation $a_v t\ge \varepsilon t$ in place of $t$; there is plenty of room to spare).
\end{proof}
\subsection{Deducing \cref{thm:GW-lower-bound}}
Recall from \cref{lem:app-typical-size} that typically our final independent set will have size $|I_\ell|=(1+o(1))\xi N\Big(D^{-1}\log f\Big)^{1/(r-1)}$, recall that $\mc G$ is $(N^{-\nu},C;\alpha)$-degree-bounded, and recall that $\alpha\le N^{-\nu}$.
So, we can choose \[D=(1+o(1))(\xi /\alpha)^{r-1}\log f\] such that
\begin{itemize}
    \item $\max(1,C)\cdot \alpha^{1-r}\le D\le N^{\nu/2}\alpha^{1-r}$,
    \item with probability $1-o(1)$, we have $|I_\ell|=(1+o(1))\alpha N$,
    \item with probability $1-o(1)$, we have $|\mc F_i[I_\ell]|= (1+o(1))\alpha^{k_i}|\mc F_i|$ for each test hypergraph $\mc F_i$.
\end{itemize}

Let $\mc E_{\mr{\ref{lem:app-typical-size}}}$ be the event that $|I_\ell|=(1+o(1))\alpha N$ and $|\mc F_i[I_\ell]|= (1+o(1))\alpha^{k_i}|\mc F_i|$ for each $\mc F_i$, where the functions implicit in ``$o(1)$'' are chosen to make $\Pr[\mc E_{\mr{\ref{lem:app-typical-size}}}]=1-o(1)$. Let $\mc E=\mc X_1\cap \dots\cap \mc X_\ell\cap \mc E_{\mr{\ref{lem:app-typical-size}}}
\cap \mc E_{\mr{\ref{lem:Gamma-approx-I}}}
\cap \mc E_{\mr{\ref{lem:app-Y-size}}}
$.

\begin{claim}\label{claim:app-prob-upper}
Consider any sequence of independent sets $J_1\subseteq \dots\subseteq J_\ell$.
Then
\[\log \Pr[(I_1,\dots,I_\ell)=(J_1,\dots,J_\ell)\text{ and }\mc E\text{ occurs}]\le -\log \left(\ell^{|J_\ell|}\binom N{|J_\ell|}\right)+\Lambda(J_1,\dots,J_\ell)+o(\alpha N),\]
where
\[\Lambda(J_1,\dots,J_\ell)=\sum_{i=0}^{\ell-1}\sum_{v\in J_{i+1}\setminus J_{i}}\deg(v)\sum_{h=0}^{i-1}p_h\chi_h.\]
\end{claim}
\begin{proof}
For each $i$, let $M_i$ be a random subset of $\Gamma_i$ obtained by including each element with probability $\theta=1-1/\log \log f=1-o(1)$. Let $J_{i+1}^*=J_{i+1}\setminus J_{i}$ (taking $J_0=\emptyset$). Then, for any outcome of $\Gamma_1,\dots,\Gamma_\ell$ we have \begin{align*}&\Pr[(M_1,\dots,M_\ell)=(J_1^*,\dots,J_\ell^*)\text{ and }\mc E\text{ occurs}\,|\,\Gamma_1,\dots,\Gamma_\ell]\\
&\qquad \ge \prod_{i=1}^\ell \theta^{|J_i^*|}(1-\theta)^{|\Gamma_i|-|J_i^*|}\cdot \Pr[(I_1,\dots,I_\ell)=(J_1,\dots,J_\ell)\text{ and }\mc E\text{ occurs}\,|\,\Gamma_1,\dots,\Gamma_\ell]\\
&\qquad \ge \theta^{|J_\ell|}(1-\theta)^{|J_\ell|/\log f}\cdot \Pr[(I_1,\dots,I_\ell)=(J_1,\dots,J_\ell)\text{ and }\mc E\text{ occurs}\,|\,\Gamma_1,\dots,\Gamma_\ell]\\
&\qquad \ge \exp(-o(\alpha N))\cdot  \Pr[(I_1,\dots,I_\ell)=(J_1,\dots,J_\ell)\text{ and }\mc E\text{ occurs}\,|\,\Gamma_1,\dots,\Gamma_\ell].\end{align*}
(Here we used $\mc E_{\mr{\ref{lem:Gamma-approx-I}}}$).
So, it suffices to show that 
\[\log \Pr[(M_1,\dots,M_\ell)=(J_1^*,\dots,J_\ell^*)\text{ and }\mc E\text{ occurs}]\le -\log \left(\ell^{|J_\ell|}
\binom N{|J_\ell|}
\right)+\Lambda(J_1,\dots,J_\ell)+o(\alpha N).\]

To have $(M_1,\dots,M_\ell)=(J_1^*,\dots,J_\ell^*)$, it must be the case that for each $i+1\le \ell$:
\begin{itemize}
    \item For each $v\in J_{i+1}^*$ we have $v\in M_{i+1}$,
    \item For each $v\in A_i\setminus J_{i+1}^*$ we have $v\notin M_{i+1}$,
    \item For each $v\in J_\ell\setminus J_{i+1}$ we have $v\notin S_{i+1}$.
\end{itemize}
If we condition on any outcome of the process up to step $i$ (consistent with $\mc X_i$, recalling \cref{lem:GW-good-events}(1)), these bullet points jointly occur with probability at most
\[(\theta p_i)^{|J_{i+1}^*|}(1-\theta p_i)^{|A_i|-|J_{i+1}^*|}\prod_{v\in J_\ell\setminus J_{i+1}}(1-\hat p_{v,i})\le p_i^{|J_{i+1}^*|}(1-\theta p_i)^{a_i-|J_{i+1}^*|}\prod_{v\in J_\ell\setminus J_{i+1}}(1-p_{i})^{K_i-y_v(i)},\]
where $a_i=(1-1/\log f)q_i N$ and $y_v(i)=\chi_i(\deg(v)+D/(\log f)^2)$.
So, multiplying over all steps and re-indexing,  
we see that
\[\Pr[(M_1,\dots,M_\ell)=(J_1^*,\dots,J_\ell^*)\text{ and }\mc E\text{ occurs}]\le \prod_{i=0}^{\ell-1}(1-\theta p_i)^{a_i-|J_{i+1}^*|}\prod_{v\in J_{i+1}^*}p_i\prod_{h=0}^{i-1}(1-p_{h})^{K_h-y_v(h)}\]
(here we are using $\mc E_{\mr{\ref{lem:app-Y-size}}}$ and property (2) in \cref{lem:GW-good-events}).
The logarithm of this is at most
\begin{equation}
    \sum_{i=0}^{\ell-1}\Bigg(-(a_i-|J_{i+1}^*|)\theta p_i+\sum_{v\in J_{i+1}^*}\bigg(\log p_i-\sum_{h=0}^{i-1}\Big(p_{h}K_h-p_{h}y_v(h)\Big)\bigg)\Bigg).\label{eq:semi-random-intermediate}
\end{equation}
Now, it basically remains to simplify this expression, recalling the definitions of all the relevant parameters. First, we have
\[\sum_{i=0}^{\ell-1} (a_i-|J_{i+1}^*|)\theta p_i=\frac{\ell \sigma N}{D^{1/(r-1)}}+o(\alpha N)=\alpha N+o(\alpha N),\]
using that 
\begin{equation}
\ell\sigma=\xi(\log f)^{1/(r-1)}+O(\sigma),\quad D^{1/(r-1)}=(1-o(1))\left(\frac{N}{|J_\ell|}\right)\xi(\log f)^{1/(r-1)}.\label{eq:recall-parameters}
\end{equation}
Then, by a Riemann sum approximation we have
\[\sum_{h=0}^{i-1}p_{h}K_h=\sigma(r-1)\sum_{h=0}^{i-1}\pi_h^{r-2}+O\Big(\sigma^{1+\rho}\sum_{h=0}^{i-1}\pi_h^{\max(r-3,0)}\Big)=(i\sigma)^{r-1}+o(1).\]
We also have the cruder estimate
\[\sum_{h=0}^{i-1}p_h\chi_h=\frac{r-1}D\sum_{h=0}^{i-1}\sigma \pi_h^{r-2}\le \frac{r-1}D\ell\sigma(\ell\sigma)^{r-2}\le \frac{(r-1)\xi^{r-1}\log f}D\](where we again recall \cref{eq:recall-parameters}), and we deduce
\[\sum_{i=0}^{\ell-1}\sum_{v\in J_{i+1}^*}\sum_{h=0}^{i-1}p_h\chi_h\cdot \frac{D}{(\log f)^2}\le |J_\ell|\cdot \frac{(r-1)\xi^{r-1}\log f}D\cdot \frac{D}{(\log f)^2}=o(\alpha N).\]
Also, note that $\log p_i=(\sigma i)^{r-1}+\log(\sigma/D^{1/(r-1)})$. Putting all this together, we see that \cref{eq:semi-random-intermediate} is at most
\[\sum_{i=0}^{\ell-1}\sum_{v\in J_{i+1}^*}\bigg(\log(\sigma/D^{1/(r-1)})+\deg(v)\sum_{h=0}^{i-1}p_{h}\chi_h\bigg)-\alpha N+o(\alpha N).\]
To finish the proof, we observe that
\[
-|J_{\ell}|\log(\sigma/D^{1/(r-1)})-\log(\ell^{|J_{\ell}|})+\alpha N=|J_{\ell}|\log\left(\frac{D^{1/(r-1)}}{\ell\sigma}\right)+\alpha N\ge\log\binom{N}{|J_{\ell}|}+o(\alpha N),
\]
again using \cref{eq:recall-parameters}.
\end{proof}
\begin{claim}\label{claim:app-Lambda}
    We have
    \[\mb E[\Lambda(I_1,\dots,I_\ell)\one_{\mc E}]\le \mc G(\alpha)+o(\alpha N).\]
\end{claim}
\begin{proof}
Recalling \cref{lem:GW-good-events}(3), we have \begin{align*}\mb E[\one_{v\in I_{i+1}\setminus I_{i}}\cdot \one_\mc E]&\le \mb P[v\in \Gamma_{i+1}\,|\,v\in A_{i}]\prod_{h=0}^{i-1}\Pr[v\in A_{h+1}\,|\,v\in A_{h}\text{ and }\mc X_1\cap \dots\cap \mc X_{h}\text{ occurs}]\\
&\le p_{i}\prod_{h=0}^{i-1}\frac {q_{h+1}}{q_{h}}\le p_{i}q_{i}=\frac{\sigma}{D^{1/(r-1)}}.\end{align*}
So,
\[\mb E[\Lambda(I_1,\dots,I_\ell)\one_{\mc E}]\le\frac{\sigma}{D^{1/(r-1)}}\sum_{v\in \on V(\mc G)}\left(\deg(v)+\frac{D}{(\log f)^2}\right)\left(\sum_{i=0}^{\ell-1}\sum_{h=0}^{i-1}p_h\chi_h\right).\]
We then note that
\begin{align*}
\frac{\sigma}{D^{1/(r-1)}}\sum^{\ell-1}_{i=0}\sum^{i-1}_{h=0}p_{h}\chi_{h} & =\frac{(r-1)\sigma^{r}}{D^{r/(r-1)}}\sum^{\ell-2}_{h=0}(\ell-1-h)(1+h)^{r-2}\\
 & =(1+o(1))\frac{(r-1)\sigma^{r}}{D^{r/(r-1)}}\int^{\ell}_{0}(\ell-h)h^{r-2}\,dh\\
 & =(1+o(1))\frac{(r-1)\sigma^{r}}{D^{r/(r-1)}}\left(\frac{\ell^{r}}{r-1}-\frac{\ell^{r}}{r}\right)\\
 & =(1+o(1))\frac{1}{r}\cdot\frac{\sigma^{r}\ell^{r}}{D^{r/(r-1)}}=(1+o(1))\frac{1}{r}\alpha^{r}
\end{align*}
and, recalling the definition of $D$,
\[\sum_{v\in \on V(\mc G)}\left(\deg(v)+\frac{D}{(\log f)^2}\right)=r|\mc G|+ND/(\log f)^2=r|\mc G|+o(N\alpha^{1-r}).\]
The desired result follows.
\end{proof}
\begin{proof}[Proof of \cref{thm:GW-lower-bound}]
By \cref{lem:app-typical-size,lem:app-Y-size,lem:Gamma-approx-I,lem:GW-good-events} we have $\Pr[\mc E]=1-o(1)$. For $m\in \{1,\dots,N\}$, let $\mc E^*_m$ be the event that $|I_\ell|=m$ and $\mc E$ occurs. By \cref{claim:app-Lambda}, we can fix $m$ such that $\Pr[\mc E_m^*]\ge N^{-2}$ and 
\begin{equation}
    \mb E[\Lambda(I_1,\dots,I_\ell)\,|\,\mc E^*_m]\le \frac{\mb E[\Lambda(I_1,\dots,I_\ell) \one_{\mc E}]}{\Pr[\mc E]-N\cdot N^{-2}}\le \mc G(\alpha)+o(\alpha N).\label{eq:Lambda-conditional}
\end{equation}

Now, it suffices to show that the number of possible outcomes of $\vec I=(I_1,\dots,I_\ell)$ satisfying $\mc E^*_m$ is at least 
\[\ell^{m}\binom N {m}\exp(-\mc G(\alpha)-o(m)).\]

Let $\mc J$ be the support of $(I_1,\dots,I_\ell)$ conditional on $\mc E_m^*$.
For $\vec J\in \mc J$, let \[P(\vec J)=\Pr[\vec I=\vec J\text{ and }\mc E^*_m\text{ occurs}],\quad P_{\rm{cond}}(\vec J)=\Pr[\vec I=\vec J\,|\,\mc E^*_m].\] Then
\[|\mc J|=\mb E[1/P_{\mr{cond}}(\vec I)\,|\,\mc E^*_m]\ge N^{-2}\mb E[1/P(\vec I)\,|\,\mc E^*_m]\ge N^{-2}\exp\big(-\mb E[\log P(\vec I)\,|\,\mc E^*_m]\big).\]
By \cref{claim:app-prob-upper} we have
\[-\mb E[\log P(\vec I)\,|\,\mc E^*_m]\ge \log\left(\ell^{m}\binom{N}{m}\right)-\mb E[\Lambda(I_1,\dots,I_\ell)\,|\,\mc E^*_m]+o(m),\]
and the desired result follows from \cref{eq:Lambda-conditional}.
\end{proof}
\end{document}